\documentclass{amsart}
\usepackage[margin=1.46in]{geometry}
\newtheorem{theorem}{Theorem}[section]
\newtheorem{lemma}[theorem]{Lemma}
\newtheorem{proposition}[theorem]{Proposition}

\theoremstyle{definition}
\newtheorem{definition}[theorem]{Definition}

\theoremstyle{remark}
\newtheorem{remark}[theorem]{Remark}

\numberwithin{equation}{section}

\begin{document}

\title[rotational local Maxwellian]{The Boltzmann equation near a rotational local Maxwellian}
%\author{Yan Guo}
%\address{Department of Applied Mathematics, Brown University, Providence, RI 02912}
%\email{Guo$\_$Yan@brown.edu}

\author{Chanwoo Kim}
\address{Department of Mathematics, Brown University, Providence, RI 02917, USA}
%\curraddr{Department of Mathematics and Statistics,
%Case Western Reserve University, Cleveland, Ohio 43403}
\email{ckim@math.brown.edu, ckim.pde@gmail.com}
%\thanks{The first author was supported in part by FRG Grant \#000000.}

\author{Seok-Bae Yun}
\address{Division of Applied Mathematics, Brown University, Providence, RI, 02812, USA}
\email{sbyun01@gmail.com, seokbae$\_$yun@brown.edu}
%    General info
\subjclass[2000]{Primary 35Q20, 82C40, 35M13}

%\date{January 1, 2001 and, in revised form, June 22, 2001.}


\keywords{Kinetic theory, Boltzmann equation, Rotational local Maxwellian, specular reflection boundary condition}

\begin{abstract}
In rotationally symmetric domains, the Boltzmann equation with specular reflection boundary condition has a
special type of equilibrium states called the rotational local Maxwellian which, unlike the uniform Maxwellian,
has an additional term related to the angular momentum of the gas. In this paper, we consider the initial boundary
value problem of the Boltzmann equation near the rotational local Maxwellian. Based on the $L^2$-$L^{\infty}$ framework
of \cite{GuoConti}, we establish the global well-posedness and the convergence toward such equilibrium states.
\end{abstract}

\maketitle

%%%%%%%%%%%%%%%%%%%%%%%%%%%%%%%%%%%%%%%%%%%%%%%%%%%%%%%%%%%%%%%%%%%%%%%%%%%%%%%%%%%%%%%%%%%%%%%%%%%%%%%%%%%%%%%%
% %%%%%%%%%%%%%%%%%%%%%%%%%%%%%%%%%%%%%%%%%%%%%%%%%%%%%%%%%%%%%%%%%%%%%%%%%%%%%%%%%%%%%%%%%%%%%%%%%%%%%%%%%%%%%%
% %
% %
% %                     Introduction
% %
% %
% %%%%%%%%%%%%%%%%%%%%%%%%%%%%%%%%%%%%%%%%%%%%%%%%%%%%%%%%%%%%%%%%%%%%%%%%%%%%%%%%%%%%%%%%%%%%%%%%%%%%%%%%%%%%%%
%%%%%%%%%%%%%%%%%%%%%%%%%%%%%%%%%%%%%%%%%%%%%%%%%%%%%%%%%%%%%%%%%%%%%%%%%%%%%%%%%%%%%%%%%%%%%%%%%%%%%%%%%%%%%%%%
\section{Introduction}
At the kinetic level, the dynamics of a non-ionized monatomic rarefied gase  is governed by the celebrated Boltzmann equation:
\begin{align}
\begin{aligned}\label{Boltzmann Equation}
\partial_t F+v\cdot\nabla_x F&=Q(F,F).
\end{aligned}
\end{align}
Here $F(x,v,t)$ denotes the number density  of particles at $(x,v)\in \Omega\times \mathbb{R}^3$ in the phase space at time $t$, and
$\Omega$ denotes a bounded open subset of $\mathbb{R}^3$.% whose geometric and analytic conditions will be specified later.
The Boltzmann equation describes the evolution of $F$ as a combination of the free transport of particles and the binary collisions.
The l.h.s of (\ref{Boltzmann Equation}) represent the free transport of particles in the absence of collisions while the collision process
is encoded in the collision operator $Q$, which takes the following explicit form:
\begin{eqnarray*}\label{Collision Op}
Q(F_1,F_2)&=&\int_{\mathbb{R}^3}\int_{\mathbb{S}^2}B(v-u,\omega)\big\{F_1(v^{\prime})F_2(u^{\prime})-F_1(v)F_2(u)\big\}d\omega du\cr
&=&Q_{gain}(F_1,F_2)-Q_{loss}(F_1,F_2).
\end{eqnarray*}
$(u,v)$ and  $(u^{\prime},v^{\prime})$ denotes the pre-collisional velocities and the post-collisional velocities respectively and
the microscopic conservation laws lead to the following relations between  $(v,u)$ and $(v^{\prime},u^{\prime})$,
with a free parameter $\omega\in \mathbb{S}^2$
\begin{eqnarray*}
u^{\prime}=u+[(v-u)\cdot \omega]\omega,\qquad v^{\prime}=v-[(v-u)\cdot \omega]\omega.
\end{eqnarray*}
%We assume that the intermolecular force is given by the inverse power law potential for hard interaction:
We assume that the Boltzmann collision kernel takes the product form as
\begin{eqnarray*}
B(v-u,\omega)=|v-u|^{\gamma}q_0(\theta),
\end{eqnarray*}
where the intermolecular potential is the hard potential $(0\leq \gamma\leq 1)$
and the collision cross section $q_0$ satisfies the Grad cut-off assumption:
\begin{eqnarray*}
0\leq q_0(\theta)\leq C|\cos\theta|,\qquad \cos\theta=\frac{(u-v)\cdot\omega}{|u-v|}.
\end{eqnarray*}
%The collision operator $Q$ satisfies the following symmetry property:
%\begin{eqnarray*}\label{Symmetry}&&\int_{\mathbb{R}^3}Q(F,F)(v)\phi(v) dv\cr&&\qquad=\int_{\mathbb{R}^6} B(v-u,\omega)\big\{F(v^{\prime})F(u^{\prime}_*)-F(v)F(u)\big\}\big\{\phi(v^{\prime})+\phi(v_*^{\prime})-\phi(v)-\phi(v_*)\big\}dvdv_*,
%\end{eqnarray*}
%which implies the conservation of mass, linear momentum, energy for $\phi(v)=1,v,|v|^2$ respectively:
%\begin{eqnarray}\label{ConservationLaws}\begin{split}&\frac{d}{dt}\int_{\Omega\times \mathbb{R}^3}Fdv=
%\frac{d}{dt}\int_{\Omega\times \mathbb{R}^3} vFdv =\frac{d}{dt}\int_{\Omega\times \mathbb{R}^3} F|v|^2dv=0,\end{split}\end{eqnarray}
%Here the linear momentum should be replace by the angular momentum when the domain is rotationally symmetric:
%\begin{eqnarray*}\frac{d}{dt}\int_{\Omega\times \mathbb{R}^3}x\times v Fdv=0.\end{eqnarray*}
%With the choice of $\phi(v)=\log f(v)$, (\ref{Symmetry}) also gives the entropy dissipation:
%:\begin{eqnarray*}\label{Entropy}
%\frac{d}{dt}\int_{\mathbb{R}^3}f\log f dv=\int_{\mathbb{R}^3}Q(f,f)\log fdv\leq 0.
%\end{eqnarray*}
In this paper, we are interested in the behavior of a rarefied gaseous system contained in a rotationally symmetric domain
with specular reflection boundary condition. Due to the fact that the specular reflection,  unlike the other boundary conditions, preserves the angular
momentum of the particle system, there exists a special type of equilibrium state called the rotational local Maxwellian, which reflect the rotational tendency
of the gases. After normalization, it takes the following form:
\begin{eqnarray}\label{RLM}
\mu_{\varpi}(v,x)=\frac{1}{\sqrt{(2\pi)^3}}\exp\Big(-{\frac{1}{2}}|v|^2+\varpi\times v\cdot x\Big),
\end{eqnarray}
where $\varpi$ denotes a vector in $\mathbb{R}^3$ related to the angular momentum of the system and the symmetry of the domain.
In this paper, we establish the global well-posedness of the initial boundary value problem and prove the convergence toward the
rotational local Maxwellian when the initial date is a small perturbation of the rotational local Maxwellian.
\newline
\indent The studies of the initial boundary value problem of the Boltzmann equation have rather short history, due mainly to
difficulties arising in understanding the complicated interaction
between the particles and the boundary, and its ongoing influence on the evolution of the statistical distribution of the gas.
Up to now, roughly two types of theories are available for the initial boundary value problem of the Boltzmann equation.
The first one is the theory of renormalized solutions by Diperna and Lions introduced in their seminar papers \cite{D-L1,D-L2}.
The extension of this result to the initial boundary value problem was carried out in \cite{Ham}. (See also \cite{A-C}).
The advantage of this approach is that the problem can be considered under the most general conditions, namely,
the finite mass, energy and entropy of the initial distribution. However, the uniqueness is not
guaranteed and remains as one of the most prominent problems in the mathematical kinetic theory.\newline
\indent On the other hand, the semi-group approaches combined with either energy estimate or spectral analysis are also available.
The spectral analysis of the linearized collision operator was initiated by Ukai in his pioneering work \cite{Ukai0}
and its application for the initial boundary value problem can be found in \cite{U-A}, where the external domain problem was considered
for specular reflection boundary conditions.
The diffusion reflection boundary condition in a bounded domain was studied in \cite{Guiraud}.\newline
%It is claimed in \cite{S-A} that the case of the specular reflection boundary condition in a convex regular domain was resolved but the proof was
%not published.\newline
\indent Liu and Yu developed the Green function approach in a series of papers \cite{L-Yu0,L-Yu1,L-Yu2,L-Yu3}. This approach enables one
to obtain the pointwise estimates of the Green function of the linearized Boltzmann equation and get detailed
information on how various types of fluid-kinetic waves propagate. For this, they introduced two types of decompositions, namely,
 the long wave-short wave decomposition and the particle-wave decomposition, which are analyzed separately and  combined using
the mixture lemma.\newline
\indent Recently, a novel $L^2$-$L^{\infty}$ framework has been developed by Guo \cite{GuoConti}. The name is self-descriptive: the coercive property of the
linearized collision operator is captured in $L^2$ space, whereas the weighted $L^{\infty}$ estimate is derived
by careful analysis of the iterated Duhamel formula \cite{Vidav} to control the bilinear
perturbation. This approach is robust in that all the four most widely used boundary conditions, namely, inflow, bounce-back,
specular reflection and diffusive boundary conditions
can be treated in a unified framework. It was then employed by the first author to study the formulation and propagation of singularity
for the initial boundary value problem for the Boltzmann equation \cite{Kim}.\newline
\indent Our work is based on the $L^2$-$L^{\infty}$ framework. There are several key difficulties.
First, since the linearized collision operator is dependent on spatial variable, it is not clear whether
the coercivity property holds uniformly in $x$, which is crucial to obtain the $L^2$ decay estimate.
It was resolved affirmatively in Lemma \ref{positivity}, and then the conservation law of angular momentum combined with the hyperbolic-type energy
method developed in \cite{GuoConti} gives the desired $L^2$ estimate.
Secondly, due to the complicated form of the compact operator and the multiplicative operator for the rotational local Maxwellian,
we should be more careful in studying the velocity decomposition for $L^{\infty}$ estimate.
The key observation is that in small velocity regime, where all the microscopic velocities involved are small, the effect of rotation is
negligible, and the estimate of the most difficult part can be treated similarly as in the case of the uniform Maxwellian.\newline
\indent Brief overview of the initial boundary value problems for other types of kinetic equations is in order. The initial boundary value problem of the
kinetic transport equation was considered in \cite{B-P}. Guo considered the half space problem of the Vlasov Maxwell system in \cite{GuoSing, GuoReg}.
Recently, the specular reflection boundary value problem was resolved by Hwang et al. \cite{Hwang,H-V}.
Similar problem for the Vlasov Poisson equation was studied in the framework of renormalzed solutions
was studied in \cite{Mis}. We will not attempt to present a complete set of references for the mathematical theory of the Boltzmann equation.
We refer the interested readers to \cite{Cer1,C-I-P,Glassy,Sone0,Sone,Ukai0,Villani} for nice overview
of mathematical and physical theory of kinetic equations.\newline
%Our main frame work is the Guo's frame work. Some explanations are in order.\newline
\indent This paper is organized as follows. In section 1, we consider the basic formulation of the problem.
In section 2. we prove preliminary lemmas which plays important
roles in the later sections. The section 3 is devoted to $L^2$ estimate of the linearized Boltzmann equation capturing the
coercivity of the linearized collision operator.
In section 4, using the iterated Duhamel formula, we obtain weighted $L^{\infty}$ estimates which are crucial to control the nonlinear terms.
Finally, the main result is proved in section 5. In appendix, we study how the rotational local Maxwellians are derived.
%%%%%%%%%%%%%%%%%%%%%%%%%%%%%%%%%%%%%%%%%%%%%%%%%%%%%%%%%%%%%%%%%%%%%%%%%%%%%%%%%%%%%%%%%%%%%%%%%%%%%%%%%%%%%%%%
% %%%%%%%%%%%%%%%%%%%%%%%%%%%%%%%%%%%%%%%%%%%%%%%%%%%%%%%%%%%%%%%%%%%%%%%%%%%%%%%%%%%%%%%%%%%%%%%%%%%%%%%%%%%%
% %
% %
% %                     subsection: Domain and Characteristics
% %
% %
% %%%%%%%%%%%%%%%%%%%%%%%%%%%%%%%%%%%%%%%%%%%%%%%%%%%%%%%%%%%%%%%%%%%%%%%%%%%%%%%%%%%%%%%%%%%%%%%%%%%%%%%%%%%%%%
%%%%%%%%%%%%%%%%%%%%%%%%%%%%%%%%%%%%%%%%%%%%%%%%%%%%%%%%%%%%%%%%%%%%%%%%%%%%%%%%%%%%%%%%%%%%%%%%%%%%%%%%%%%%%%%%
\subsection{Domain and characteristics}
Let $\Omega$ be a connected and bounded domain.
We assume that there exists a smooth function $\xi(x)$ such that $\Omega=\{x: \xi(x)<0\}$. We further assume $\nabla\xi(x)\neq 0$ at the
boundary $\partial\Omega=\{ x~:~ \xi(x)=0\}$. The outward normal vector at $\partial\Omega$ is given by
\begin{eqnarray*}
n(x)=\frac{\nabla\xi(x)}{|\nabla\xi(x)|}.
\end{eqnarray*}
We say $\Omega$ is real analytic if $\xi$ is real analytic in $x$. We define $\Omega$ as strictly convex if there exists $C_{\xi}>0$
such that
\begin{eqnarray*}
\partial_{ij}\xi(x)\zeta^i\zeta^j\geq C_{\xi}|\zeta|^2
\end{eqnarray*}
for all $x\in \mathbb{R}^3$ such that $\xi(x)\leq 0$, and $\zeta\in \mathbb{R}^3$.
We say that $\Omega$ is rotationally symmetric around $e_3=(0,0,1)$ if for all $x\in \partial\Omega$
\begin{eqnarray}\label{cylindricallySymm}
x\times e_3\cdot n(x)=0.\cr
\end{eqnarray}
{\bf Geometric assumption $\mathcal{(A)}$}:
Throughout this paper, we assume that $\Omega$ satisfies the following geometric conditions:
%\begin{description}\label{GeometricAssumptions}
%\item {\bf${\bf{\mathcal{(A)}}}$ $\Omega$ is strictly convex, analytic and rotationally symmetric around ${\bf e_3}$.}\newline
%\end{description}
\begin{center}\label{GeometricAssumptions}
\item {\bf${\bf{\mathcal{(A)}}}:$ $\Omega$ is strictly convex, analytic and rotationally symmetric around ${\bf e_3}$.}\newline
\end{center}
%{\bf$\Omega$ is strictly convex, analytic and rotationally symmetric around $e_3$.}\newline
We denote the phase boundary in the space $\Omega\times \mathbb{R}^3$ as $\gamma=\partial\Omega\times \mathbb{R}^3$, and split
it into an outgoing boundary $\gamma_+$, an incoming boundary $\gamma_-$,
and a singular boundary $\gamma_0$ for grazing velocities:
\begin{eqnarray*}
\gamma_{+}&=&\{(x,v)\in\partial \Omega\times \mathbb{R}^3:~n(x)\cdot v>0\},\cr
\gamma_{-}&=&\{(x,v)\in\partial \Omega\times \mathbb{R}^3:~n(x)\cdot v<0\},\cr
\gamma_{0}&=&\{(x,v)\in\partial \Omega\times \mathbb{R}^3:~n(x)\cdot v=0\}.
\end{eqnarray*}
In terms of $F$, we define the specular boundary condition as
\begin{eqnarray}\label{specularBoundaryC}
&&F(t,x,v)=F(t,x,R(x)v),\qquad(x,v)\in\gamma^-,
\end{eqnarray}
where
\begin{eqnarray*}
&&R(x)v=v-2\{n(x)\cdot v\}n(x).
\end{eqnarray*}
Given $(t,x,v)$, let $[X(s), V(s)]=[X(s;t,x,v), V(s;t,x,v)]=[x+(s-t)v,v]$ be the trajectory (or the characteristics)
for the Boltzmann equation (\ref{Boltzmann Equation}):
\begin{eqnarray*}\label{trajectoryEquation}
&&\frac{dX(s)}{ds}=V(s),\quad\frac{V(s)}{ds}=0
\end{eqnarray*}
with the initial condition $[X(t;t,x,v), V(t;t,x,v)]=[x,v]$.
\begin{definition}\cite{GuoConti} (Backward exit time)
For $(x,v)$ with $x\in \bar{\Omega}$ such that there exists some $\tau>0, x-sv\in \Omega$ for $0\leq s\leq \tau$,
we define $t_{\mathbf{b}}(x,v)>0$ to be the last moment at which the back-time straight line
$[X(0;t,x,v), V(0;t,x,v)]$ remains in the interior of $\Omega$:
\begin{eqnarray*}\label{back exit time}
t_{\mathbf{b}}(x,v)=\sup\{\tau>0: x-sv\in\Omega \mbox{ for all } 0\leq s\leq \tau\}
\end{eqnarray*}
We also define
\begin{eqnarray*}
x_{\mathbf{b}}(x,v)=x(t_{\mathbf{b}})=x-t_{\mathbf{b}}v\in\partial\Omega.
\end{eqnarray*}
We always have $v\cdot n(x_{\mathbf{b}})\leq 0$.
\end{definition}
%%%%%%%%%%%%%%%%%%%%%%%%%%%%%%%%%%%%%%%%%%%%%%%%%%%%%%%%%%%%%%%%%%%%%%%%%%%%%%%%%%%%%%%%%%%%%%%%%%%%%%%%%%%%%%%%
% %%%%%%%%%%%%%%%%%%%%%%%%%%%%%%%%%%%%%%%%%%%%%%%%%%%%%%%%%%%%%%%%%%%%%%%%%%%%%%%%%%%%%%%%%%%%%%%%%%%%%%%%%%%%%%
% %
% %
% %                     subsection: boundary conditions
% %
% %
% %%%%%%%%%%%%%%%%%%%%%%%%%%%%%%%%%%%%%%%%%%%%%%%%%%%%%%%%%%%%%%%%%%%%%%%%%%%%%%%%%%%%%%%%%%%%%%%%%%%%%%%%%%%%%%
%%%%%%%%%%%%%%%%%%%%%%%%%%%%%%%%%%%%%%%%%%%%%%%%%%%%%%%%%%%%%%%%%%%%%%%%%%%%%%%%%%%%%%%%%%%%%%%%%%%%%%%%%%%%%%%%
\subsection{Boundary conditions and conservation laws}
(\ref{RLM}) in the introduction, we defined the rotational local Maxwellian $\mu_{\varpi}$ as
\begin{eqnarray*}
\mu_{\varpi}(v,x)=\frac{1}{\sqrt{(2\pi)^3}}\exp\Big(-{\frac{1}{2}}|v|^2+\varpi\times v\cdot x\Big),
\end{eqnarray*}
where $\varpi$ denotes an fixed vector in $R^3$.
For notational simplicity, we introduce $m_{\varpi}$ as
\begin{eqnarray*}
m_{\varpi}(x,v)\equiv\frac{1}{\sqrt{(2\pi)^3}}e^{-\frac{|v-\varpi\times x|^2}{2}},
%&\equiv \frac{1}{\sqrt{(2\pi)^3}}e^{-\frac{1}{2}|v|^2+\varpi\times x\cdot v}\cr
%&=\frac{1}{\sqrt{(2\pi)^3}}e^{\frac{|\varpi\times x|^2}{2}}e^{-\frac{|v-\varpi\times x|^2}{2}}\cr
%&\equiv\frac{1}{\sqrt{(2\pi)^3}}e^{\frac{|\varpi\times x|^2}{2}}m_{\varpi}(x,v).
\end{eqnarray*}
satisfying
\begin{eqnarray*}
\mu_{\varpi}(x,v)\equiv e^{\frac{|\varpi\times x|^2}{2}}m_{\varpi}(x,v).
%&\equiv \frac{1}{\sqrt{(2\pi)^3}}e^{-\frac{1}{2}|v|^2+\varpi\times x\cdot v}\cr
%&=\frac{1}{\sqrt{(2\pi)^3}}e^{\frac{|\varpi\times x|^2}{2}}e^{-\frac{|v-\varpi\times x|^2}{2}}\cr
%&\equiv\frac{1}{\sqrt{(2\pi)^3}}e^{\frac{|\varpi\times x|^2}{2}}m_{\varpi}(x,v).
\end{eqnarray*}
Throughout this paper, we consider the following perturbation:
\begin{eqnarray*}
F=\mu_{\varpi}+\sqrt{m}_{\varpi}f.
\end{eqnarray*}
Then the Boltzmann equation can be rewritten as
\begin{eqnarray}
\{\partial_t+v\cdot\nabla+L^{\varpi}\}f=\Gamma^{\varpi}(f,f).\quad f(0,x,v)=f_0(x,v),
\end{eqnarray}
with the boundary condition:
\begin{eqnarray}\label{specularBoundaryC_f}
f(t,x,v)=f(t,x,R(x)v)~\mbox{ for }~(x,v)\in \gamma_-.
\end{eqnarray}
The linear Boltzmann operator is given by
\begin{eqnarray*}
L^{\varpi}f&\equiv&-\frac{1}{\sqrt{m_{\varpi}}}\{Q(\mu_{\varpi}, \sqrt{m_{\varpi}}f)+Q(\sqrt{m_{\varpi}}f,\mu_{\varpi})\}\cr
&=&\nu^{\varpi}f-K^{\varpi,x}f\cr
&=&\nu^{\varpi}f-\int K^{\varpi,x}_{w_{\varpi}}(v,v^{\prime})f(v^{\prime})dv^{\prime},
\end{eqnarray*}
with the collision frequency $\nu^{\varpi}(x,v)=\int_{\mathbb{R}^3}|v-u|^{\gamma}q_0(\theta)m_{\varpi}(u)du\sim (1+|v-\varpi\times x|^2)^{\frac{1}{2}}$ for $0\leq \gamma\leq 1$.
For convenience, we define $\mathcal{L}_{\varpi}$ as

\begin{eqnarray*}
\mathcal{L}^{\varpi}f&=&-\frac{1}{\sqrt{m_{\varpi}}}\{Q(m_{\varpi}, \sqrt{m_{\varpi}}f)+Q(\sqrt{m_{\varpi}}f,m_{\varpi})\}\cr
&=&\nu^{\varpi}f-\int \mathcal{K}^{\varpi,x}(v,v^{\prime})f(v^{\prime})dv^{\prime}.
\end{eqnarray*}
The following properties of $\mathcal{L}_{\varpi}$ and $L_{\varpi}$ can be readily checked.
\begin{eqnarray}\label{proeprites of L}
\begin{split}
L^{\varpi}f&=e^{\frac{|\varpi\times x|^2}{2}}\mathcal{L}^{\varpi}f,\cr
\nu^{\varpi}&=e^{\frac{|\varpi\times x|^2}{2}}\nu^{\varpi},\cr
K^{\varpi,x}(v,v^{\prime})&=e^{\frac{|\varpi\times x|^2}{2}}\mathcal{K}^{\varpi,x}(v,v^{\prime})\cr
&=e^{\frac{|\varpi\times x|^2}{2}}\mathcal{K}^{0,x}(v-\varpi\times x,v^{\prime}-\varpi\times x),
\end{split}
\end{eqnarray}
The kernels of $L$ and $\mathcal{L}$ are given as
\begin{eqnarray*}\label{kernel}
\begin{split}
\ker L^{\varpi}&=span\{\sqrt{\mu}_{\varpi},v\sqrt{\mu}_{\varpi},|v|^2\sqrt{\mu}_{\varpi}\},\cr
\ker \mathcal{L}^{\varpi}&=span\{\sqrt{m}_{\varpi},v\sqrt{m}_{\varpi},|v|^2\sqrt{m}_{\varpi}\}.
\end{split}
\end{eqnarray*}
We define the macroscopic projection $P_{\varpi}$ as
\begin{eqnarray}\label{MacroPro}
P_{\varpi}f=a(x,t)\sqrt{\mu}_{\varpi}+b(x,t)\cdot v\sqrt{\mu}_{\varpi}
+c(t,x)|v|^2\sqrt{\mu}_{\varpi},
\end{eqnarray}
where
\begin{eqnarray*}
&&a(x,t)=\int_{\mathbb{R}^3}f\sqrt{\mu}_{\varpi}dv,\quad
b(x,t)=\int_{\mathbb{R}^3}fv\sqrt{\mu}_{\varpi}dv,\quad
c(t,x)=\int_{\mathbb{R}^3}f\sqrt{\mu}_{\varpi}|v|^2dv.
\end{eqnarray*}
The bilinear perturbation is defined as
\begin{eqnarray}\label{Gamma}
\begin{split}
\Gamma^{\varpi}(f,f)&=\frac{1}{\sqrt{m_{\varpi}}}Q(\sqrt{m}_{\varpi}f, \sqrt{m}_{\varpi}f)\cr
&=\Gamma^{\varpi}_{gain}(f,f)-\Gamma^{\varpi}_{loss}(f,f).
\end{split}
\end{eqnarray}
%In terms of $f$, we formulate the boundary condition (\ref{specularBoundaryC}) as
With the specular reflection condition (\ref{specularBoundaryC_f}), it is well-known that mass, angular momentum, and energy are
conserved for (\ref{Boltzmann Equation}). Without loss of generality, we may always assume that the mass-angular momentum-energy conservation laws hold
for $t\geq 0$, in terms of the perturbation $f$:
\begin{eqnarray}\label{ConservationLaw}
\begin{split}
&\int_{\Omega\times \mathbb{R}^3}f(t,x,v)\sqrt{m}_{\varpi}dxdv=0,\cr
&\int_{\Omega\times \mathbb{R}^3}(x\times v)f(t,x,v)\sqrt{m}_{\varpi}dxdv=0,\cr
&\int_{\Omega\times \mathbb{R}^3}|v|^2f(t,x,v)\sqrt{m}_{\varpi}dxdv=0.
\end{split}
\end{eqnarray}
%Since the domain $\Omega$ has the axis of rotational symmetry, we assume the conservation of
%angular momentum:
%Without loss of generality, we define $\Lambda_i$ for the domain satisfying the geometric assumption $\mathcal{(A)}1$ as
%\begin{align}\begin{aligned}\Lambda_1&=\Lambda_2=(0,0,0),\cr\Lambda_1&=e_3=(0,0,1).\end{aligned}\end{align}
%and for domain satisfying $\mathcal{(A)}2$\begin{align}\begin{aligned}\Lambda_1&=e_1=(1,0,0),\cr
%\Lambda_2&=e_2=(0,1,0),\cr\Lambda_3&=e_3=(0,0,1).\end{aligned}\end{align}
%%%%%%%%%%%%%%%%%%%%%%%%%%%%%%%%%%%%%%%%%%%%%%%%%%%%%%%%%%%%%%%%%%%%%%%%%%%%%%%%%%%%%%%%%%%%%%%%%%%%%%%%%%%%%%%%
% %%%%%%%%%%%%%%%%%%%%%%%%%%%%%%%%%%%%%%%%%%%%%%%%%%%%%%%%%%%%%%%%%%%%%%%%%%%%%%%%%%%%%%%%%%%%%%%%%%%%%%%%%%%%%%
% %
% %
% %                    Subsection: Main result
% %
% %
% %%%%%%%%%%%%%%%%%%%%%%%%%%%%%%%%%%%%%%%%%%%%%%%%%%%%%%%%%%%%%%%%%%%%%%%%%%%%%%%%%%%%%%%%%%%%%%%%%%%%%%%%%%%%%%
%%%%%%%%%%%%%%%%%%%%%%%%%%%%%%%%%%%%%%%%%%%%%%%%%%%%%%%%%%%%%%%%%%%%%%%%%%%%%%%%%%%%%%%%%%%%%%%%%%%%%%%%%%%%%%%%
\subsection{Main results}
We are now in a place to state our main theorem. We first introduce the weight function for $\beta>\frac{3}{2}$:
\begin{eqnarray}\label{Weight}
w_{\varpi}(v)=(1+|v-\varpi\times x|^2)^{\beta}.
\end{eqnarray}
\begin{theorem}\label{MainResult}
$(1)$ Assume $\Omega$ satisfies the geometric assumption $\mathcal{(A)}$.
Then any Maxwellian solution of the Boltzmann equation with the specular boundary condition in $\Omega$
takes the following form after the normalization:
\begin{eqnarray*}
\frac{1}{\sqrt{(2\pi)^3}}\exp\Big(-{\frac{1}{2}}|v|^2+\varpi\times v\cdot x\Big)
\end{eqnarray*}
for some fixed vector $\varpi\in \mathbb{R}^3$. Moreover, except for the case $\Omega=\mathbb{S}^3$, $\varpi$ is given by
$\varpi=(0,0,\varpi_3)$ for some $\varpi_3\in \mathbb{R}$.
\newline
$(2)$ Assume $\Omega$ satisfies the geometric assumption $\mathcal{(A)}$ and $\varpi$ is parallel to $e_3$.
Then there exists $\delta>0$ such that if $F_0(x,v)=\mu_{\varpi}+\sqrt{m}_{\varpi}f_0(x,v)\geq 0$
and $\|w_{\varpi}f_0\|_{\infty}\leq \delta$, there exists a unique solution
$F(t,x,v)=\mu_{\varpi}+\sqrt{m}_{\varpi}f(t,x,v)\geq 0$ to the specular boundary value problem (\ref{specularBoundaryC_f}) to
the Boltzmann equation (\ref{Boltzmann Equation})
such that
\begin{eqnarray*}
\sup_{0\leq t\leq\infty}e^{\lambda t}\|w_{\varpi}f(t)\|_{\infty}\leq C\|w_{\varpi}f_0\|_{\infty}.
\end{eqnarray*}
Moreover, if $f_0(x,v)$ is continuous except on $\gamma_0$ and
\begin{eqnarray*}
f_0(x,v)=f_0(x,R(x)v)\quad\mbox{on }\partial\Omega,
\end{eqnarray*}
then $f(t,x,v)$ is continuous in $[0,\infty)\times\{\bar\Omega\times \mathbb{R}^3\backslash\gamma_0\}$.
\end{theorem}
\begin{remark}
(1) Theorem 1.2 (1) states that the rotational local Maxwellian is essentially the only possible equilibrium state of the Boltzmann equation
in a rotationally symmetric domain. The proof will be given in Appendix.\newline
(2) The assumption on $\varpi$ imposed in Theorem 1.2 (2) is justified by Theorem 1.2 (1), because in the case of $S^2$, we may redefine $\varpi$ as $e_3$ by the symmetry of $\mathbb{S}^3$. \newline
(3) Different proofs of  Theorem 1.2 (1) can also be found in \cite{Desvil} and \cite{Sone} .
\end{remark}
\subsection{Notations}
Before we proceed further, we set some notational conventions for various norms and inner product that
will be used in later sections.
We define $L^2$-norm, $L^{\infty}$-norm and weighted $L^2$-norm as:
\begin{eqnarray*}
&&\|f\|_{\infty}=ess\sup |f|,\cr
&&\|f\|^2=\int_{\mathbb{R}^3}f^2(v)dv,\mbox{ or }\int_{\Omega\times \mathbb{R}^3}f^2(t,x,v)dxdv,\cr
&&\|f\|^2_{v_{\varpi}}=\int_{\mathbb{R}^3}f^2(t,x,v)(1+|v-\varpi\times x|^2)^{\frac{1}{2}}dv,\mbox{ or}\cr
&&\hspace{1.4cm}\int_{\Omega\times \mathbb{R}^3}f^2(t,x,v)(1+|v-\varpi\times x|^2)^{\frac{1}{2}}dxdv.
\end{eqnarray*}
We use the following standard notation for corresponding inner products.
\begin{eqnarray*}
&&(f,g)_{L^2_v}=\int_{\mathbb{R}^3}f(v)g(v)dv,\cr
&&(f,g)_{L^2_{x,v}}=\int_{\Omega\times \mathbb{R}^3}f(x,v)g(x,v)dxdv,\cr
&&(f,g)_{v_{\varpi}}=\int_{\mathbb{R}^3}f(v)g(v)(1+|v-\varpi\times x|^2)^{\frac{1}{2}}dv,\mbox{ or}\cr
&&\hspace{1.4cm}\int_{\Omega\times \mathbb{R}^3}f(x,v)g(x,v)(1+|v-\varpi\times x|^2)^{\frac{1}{2}}dxdv.
\end{eqnarray*}
%%%%%%%%%%%%%%%%%%%%%%%%%%%%%%%%%%%%%%%%%%%%%%%%%%%%%%%%%%%%%%%%%%%%%%%%%%%%%%%%%%%%%%%%%%%%%%%%%%%%%%%%%%%%%%%%
% %%%%%%%%%%%%%%%%%%%%%%%%%%%%%%%%%%%%%%%%%%%%%%%%%%%%%%%%%%%%%%%%%%%%%%%%%%%%%%%%%%%%%%%%%%%%%%%%%%%%%%%%%%%%%%
% %
% %
% %                    Section: Preliminary
% %
% %
% %%%%%%%%%%%%%%%%%%%%%%%%%%%%%%%%%%%%%%%%%%%%%%%%%%%%%%%%%%%%%%%%%%%%%%%%%%%%%%%%%%%%%%%%%%%%%%%%%%%%%%%%%%%%%
%%%%%%%%%%%%%%%%%%%%%%%%%%%%%%%%%%%%%%%%%%%%%%%%%%%%%%%%%%%%%%%%%%%%%%%%%%%%%%%%%%%%%%%%%%%%%%%%%%%%%%%%%%%%%%%
\section{preliminary}
%%%%%%%%%%%%%%%%%%%%%%%%%%%%%%%%%%%%%%%%%%%%%%%%%%%%%%%%%%%%%%%%%%%%%%%%%%%%%%%%%%%%%%%%%%%%%%%%%%%%%%%%%%%%%%5
%
%                   Velocity Lemma
%
%%%%%%%%%%%%%%%%%%%%%%%%%%%%%%%%%%%%%%%%%%%%%%%%%%%%%%%%%%%%%%%%%%%%%%%%%%%%%%%%%%%%%%%%%%%%%%%%%%%%%%%%%%%%%%%
In this section, we establish important technical lemmas for later use. The following lemma shows that
if $\Omega$ satisfies $\mathcal{(A)}$,
the trajectory
cannot reach the singular boundary $\gamma_0$ if it is not grazing initially.
\begin{lemma}\label{Velocity Lemma}$($\emph{Velocity Lemma}$)$ \cite{GuoConti}
Let $\Omega$ be a domain satisfying the geometric assumption $\mathcal{(A)}$. Define the functional
along the trajectory as
\begin{eqnarray}\label{functional}
&&\alpha(s)=\xi^2(X(s))+[V(x)\cdot \nabla \xi(X(s))]^2-2\{V(s)\cdot\xi(X(s))\cdot V(s)\}\xi(X(x)).
\end{eqnarray}
Let $X(s)\in \bar\Omega$ for $t_1\leq t\leq t_2$. Then there exists constant $C_{\xi}$ such that
\begin{align*}
\begin{aligned}
e^{C_{\xi}(|V(t_1)|+1)t_1}\alpha(t_1)&\leq e^{C_{\xi}(|V(t_1)|+1)t_2}\alpha(t_2),\cr
e^{-C_{\xi}(|V(t_1)|+1)t_1}\alpha(t_1)&\geq e^{-C_{\xi}(|V(t_1)|+1)t_2}\alpha(t_2).
\end{aligned}
\end{align*}
\end{lemma}
\begin{proof}
See \cite{GuoConti}.
\end{proof}
%%%%%%%%%%%%%%%%%%%%%%%%%%%%%%%%%%%%%%%%%%%%%%%%%%%%%%%%%%%%%%%%%%%%%%%%%%%%%%%%%%%%%%%%%%%%%
%
%                   Lemma
%
%%%%%%%%%%%%%%%%%%%%%%%%%%%%%%%%%%%%%%%%%%%%%%%%%%%%%%%%%%%%%%%%%%%%%%%%%%%%%%%%%%%%%%%%%%%%%
\begin{lemma}\label{1234}\cite{GuoConti}
Let $(t,x,v)$ be connected with $(t-t_b,x-b_{\mathbf{b}}v,v)$ backward in time through a trajectory
\begin{enumerate}
\item The backward exit time $t_{\mathbf{b}}(x,v)$ is lower semicontinuous.
\item If
\begin{eqnarray*}
v\cdot n(x_{\mathbf{b}})<0,
\end{eqnarray*}
then $(t_{\mathbf{b}}(x,v), x_{\mathbf{b}}(x,v))$ are smooth functions of $(x,v)$ so that
\begin{eqnarray*}
&&\nabla_x t_{\mathbf{b}}=\frac{n(x_{\mathbf{b}})}{v\cdot n(x_{\mathbf{b}})},\quad \nabla_vt_{\mathbf{b}}=\frac{t_{\mathbf{b}}n(x_{\mathbf{b}})}{v\cdot n(x_{\mathbf{b}})},\cr
&&\nabla_x x_{\mathbf{b}}=I+\nabla_x t_{\mathbf{b}}\otimes v,\quad \nabla_v x_{\mathbf{b}}=t_{\mathbf{b}}I+\nabla_vt_{\mathbf{b}}\otimes v.
\end{eqnarray*}
Furthermore, if $\xi$ is real analytic, then $(t_{\mathbf{b}}(x,v), x_{\mathbf{b}}(x,v))$ are also real analytic.
\item Let $x_i\in \partial\Omega$, for $i=1,2$ and let $(t_1,x_1,v)$ and $(t_2,x_2,v)$ be connected with the trajectory. %\ref{trajectory}
Then there exists a constant $C_{\xi}$ such that
\begin{eqnarray*}\label{Lower bd of t}
|t_1-t_2|\geq\frac{|n(x_1)\cdot v|}{C_{\xi}|v|^2}.
\end{eqnarray*}
\item Define the boundary mapping
\begin{eqnarray*}
\Phi_{\mathbf{b}}:(t,x,v)\rightarrow (t-t_{\mathbf{b}},x_{\mathbf{b}}(x,v),v)\in \mathbb{R}\times\{\gamma_0\cup\gamma_-\}.
\end{eqnarray*}
Then $\Phi_{\mathbf{b}}$ and $\Phi^{-1}_{\mathbf{b}}$ maps zero-measure sets to zero-measure sets between either
$\{t\}\times\Omega \times \mathbb{R}^3$ and $\mathbb{R}\times\{\gamma_0\cup\gamma_-\}$ or $\mathbb{R}\times\{\gamma_0\cup\gamma_+\}
\rightarrow \mathbb{R}\times \{\gamma_0\cup\gamma_-\}$.
\end{enumerate}
\end{lemma}
\begin{proof}
See \cite{GuoConti}
\end{proof}
%%%%%%%%%%%%%%%%%%%%%%%%%%%%%%%%%%%%%%%%%%%%%%%%%%%%%%%%%%%%%%%%%%%%%%%%%%%%%%%%%%%%%%%%%%%%%
%
%                   Lemma
%
%%%%%%%%%%%%%%%%%%%%%%%%%%%%%%%%%%%%%%%%%%%%%%%%%%%%%%%%%%%%%%%%%%%%%%%%%%%%%%%%%%%%%%%%%%%%%
\begin{lemma}\label{K_wEstimate}There exists constants $C>0$ such that
\begin{eqnarray*}
K^{\varpi,x}_{\varpi}(v,v^{\prime})\leq C\{|v-v^{\prime}|+|v-v^{\prime}|^{-1}\}e^{-\frac{1}{8}|v-v^{\prime}|^2
-\frac{1}{8}\frac{||v-\varpi\times x|^2-|v^{\prime}-\varpi\times x|^2|^2}{|v-v^{\prime}|^2}},
\end{eqnarray*}
and
\begin{eqnarray*}\label{45}
&&\int_{\mathbb{R}^3} \{|v-v^{\prime}|+|v-v^{\prime}|^{-1}\}e^{-\frac{1}{8}|v-v^{\prime}|^2
-\frac{1}{8}\frac{||v-\varpi\times x|^2-|v^{\prime}-\varpi\times x|^2|^2}{|v-v^{\prime}|^2}}
\frac{w_{\varpi}(v)}{w_{\varpi}(v^{\prime})}dv^{\prime}\cr
&&\qquad\leq \frac{C}{(1+|v-\varpi\times x|)}.
\end{eqnarray*}
\end{lemma}
\begin{proof}
We first notice that for some $C_{\rho,\beta}>0$,
\begin{eqnarray*}
\left|\frac{w_{\varpi}(v)}{w_{\varpi}(v^{\prime})}\right|\leq C\{1+|v-v^{\prime}|^2\}^{|\beta|}.
\end{eqnarray*}
Let $v-v^{\prime}=\eta$ and $v^{\prime}=v-\eta$. We now compute as
\begin{eqnarray*}
&&-\frac{1}{8}|v-v^{\prime}|^2-\frac{1}{8}\frac{||v-\varpi\times x|^2-|v^{\prime}-\varpi\times x|^2|^2}{|v-v^{\prime}|^2}\cr
&&\qquad=-\frac{1}{8}|\eta|^2-\frac{1}{8}\frac{||\eta|^2-2(v-\varpi\times x)\cdot\eta|^2}{|\eta|^2}\cr%-\theta\{|v-\varpi\times x-\eta|^2-|v|^2\}\cr
&&\qquad=-\frac{1}{4}|\eta|^2+\frac{1}{2}(v-\varpi\times x)\cdot\eta-\frac{1}{2}\frac{|(v-\varpi\times x)\cdot\eta|^2}{|\eta|^2}\cr
&&\qquad=-\frac{1}{16}|\eta|^2+\frac{1}{2}(v-\varpi\times x)\cdot\eta
-\Big\{\frac{3}{16}|\eta|^2+\frac{1}{2}\frac{|(v-\varpi\times x)\cdot\eta|^2}{|\eta|^2}\Big\}\cr
&&\qquad\leq-\frac{1}{8}|\eta|^2-\frac{1}{2}\Big(\sqrt{\frac{3}{2}}-1\Big)|(v-\varpi\times x)\cdot\eta|\cr
%&&\qquad\leq-\frac{1}{6}|\eta|^2-\frac{1}{8}\frac{|(v-\varpi\times x)\cdot\eta|^2}{|\eta|^2}\cr%-\theta\{|\eta|^2-2v\cdot\eta\}\cr
&&\qquad\leq-C\Big\{|\eta|^2+\big|(v-\varpi\times x)\cdot\eta\big|\Big\},
\end{eqnarray*}
%Since $\theta<\frac{1}{4}$, the discriminant of the above quadratic form of $|\eta|$ $\frac{v\cdot \eta}{|\eta|}$ is
%\begin{eqnarray*}
%\Delta=\left(\frac{1}{2}+2\theta\right)^{2}-2\left(-\theta-\frac{1}{4}\right)=4\theta^2-\frac{1}{4}<0.
%\end{eqnarray*}
for some constant $C>0$.
%Here we used\begin{eqnarray*}
%\frac{3}{16}|\eta|^2+\frac{1}{2}\frac{|\{v-\varpi\times x\}\cdot\eta|^2}{|\eta|^2}\geq \frac{1}{2}\sqrt{\frac{3}{2}}|(v-\varpi\times x)\cdot\eta|.
%\end{eqnarray*}
Hence it is sufficient to estimate
\begin{eqnarray*}\label{45}
I=\int_{\mathbb{R}^3} \Big\{1+\frac{1}{|\eta|}\Big\}e^{-C|\eta|^2}e^{-C|(v-\varpi\times x)\cdot\eta|}d\eta.
\end{eqnarray*}
For $|v-\varpi\times x|\leq 1$, we have
\begin{eqnarray*}
I\leq\int_{\mathbb{R}^3} \Big\{1+\frac{1}{|\eta|}\Big\}e^{-C|\eta|^2}d\eta<\infty.
\end{eqnarray*}
%\begin{eqnarray*}&&-\frac{1-\varepsilon}{8}|\eta|^2-\frac{1-\varepsilon}{8}\frac{||\eta|^2-2v\cdot\eta|^2}{|\eta|^2}-\theta\{|v-\eta|^2-|v|^2\}\cr
%&&\leq -C_{\theta}\left\{|\eta|^2+\frac{|v\cdot\eta|^2}{|\eta|^2}\right\}\cr
%&&\leq-C_{\theta}\left\{\frac{|\eta|^2}{2}+\left(\frac{|\eta|^2}{2}+\frac{|v\cdot\eta|^2}{|\eta|^2}\right)\right\}\cr
%&&\leq-C_{\theta}\left\{\frac{|\eta|^2}{2}|v\cdot\eta|\right\}.\end{eqnarray*}
For $|v-\varpi\times x|\geq 1$, we make a change of variable $\eta_{\|}=\{\eta\cdot \frac{v-\varpi\times x}{|v-\varpi\times x|}\}\frac{v-\varpi\times x}{|v-\varpi\times x|}$, which gives
$\eta_{\perp}=\eta-\eta_{\|}$ so that $|(v-\varpi\times x)\cdot \eta|=|v-\varpi\times x||\eta_{\|}|$ and $|v-v^{\prime}|\geq|\eta_{\perp}|$.
%We can %absorb $\{1+|\eta|^2\}^{|\beta|}$, $|\eta|\{1+|\eta|^2\}^{|\beta|}$ by $e^{-\frac{C|\eta|^2}{4}}$, and bound the integral by
\begin{eqnarray*}\begin{split}
I&\leq C_{\beta}\int_{\mathbb{R}^2}\left(\frac{1}{|\eta_{\perp}|}+1\right)e^{-\frac{C}{4}|\eta|^2}
\left\{\int^{\infty}_{0}e^{-C|v-\varpi\times x|\times|v_{\|}|}d|\eta_{\|}|\right\}d\eta_{\perp}\cr
&= \frac{C_{\beta}}{|v-\varpi\times x|}\int_{\mathbb{R}^2}\left(\frac{1}{|\eta_{\perp}|}+1\right)e^{-\frac{C}{4}|\eta|^2}
\left\{\int^{\infty}_{0}e^{-C|y|}dy\right\}d\eta_{\perp}\cr
&\leq \frac{C_{\beta}}{|v-\varpi\times x|}.
%&&\hspace{2cm}.
\end{split}
\end{eqnarray*}
In the last line, we used the change of variable: $y=|v-\varpi\times x|\times|\eta_{\|}|$.
This completes the proof.
\end{proof}
The following lemma shows that the coercivity estimate holds uniformly with respect to $\varpi$ and $x$.
%%%%%%%%%%%%%%%%%%%%%%%%%%%%%%%%%%%%%%%%%%%%%%%%%%%%%%%%%%%%%%%%%%%%%%%%%%%%%%%%%%%%%%%%%%%%%%%%%%%%%%%%
%
%                            Lemma: Positivity
%
%%%%%%%%%%%%%%%%%%%%%%%%%%%%%%%%%%%%%%%%%%%%%%%%%%%%%%%%%%%%%%%%%%%%%%%%%%%%%%%%%%%%%%%%%%%%%%%%%%%%%%%%
\begin{lemma}[Positivity of $L^{\varpi}$]\label{positivity} There exists a constant $\delta_0>0$ which is independent of $x$ and $\varpi$ such that
\begin{eqnarray*}
(L^{\varpi}f,f)_{L^2_{x,v}}&\geq&\delta_0\int_{\mathbb{R}^3}e^{\frac{|\varpi\times x|^2}{2}}\|(I-P_{\varpi})f\|^2_{\nu_{\varpi}}dx\cr
&\geq& \delta_0\|(I-P_{\varpi})f\|^2_{\nu_{\varpi}}.
\end{eqnarray*}
\end{lemma}
\begin{proof}
We define the shift operator $\tau_y$ as
\[
\tau_y: f(x)\rightarrow f(x+y).
\]
Then we have from (\ref{proeprites of L}) and the Fubini's theorem,
\begin{eqnarray*}
\int\nu^{\varpi}f(v)dv&=&\int |v-v_*|e^{-\frac{|v_*-\varpi\times x|^2}{4}}f(v)d\omega dvdv_*\cr
&=&\int |v-v_*|e^{-\frac{|v_*|^2}{4}}f(v+\varpi\times x)d\omega dvdv_*\cr
&=&\nu^0\tau_{\varpi\times x}f(v),
\end{eqnarray*}
and
\begin{eqnarray*}
&&\int K^{\varpi,x}(v,v^{\prime})f(v)f(v^{\prime})dvdv^{\prime}
=\int K^{0,x}(v,v^{\prime})\tau_{\varpi\times x}f(v)\tau_{\varpi\times x}f(v^{\prime})dvdv^{\prime}.
\end{eqnarray*}
Hence we have
\begin{eqnarray*}
(\mathcal{L}_{\varpi}f,f)_{L^2_v}&=&-\int\nu^{\varpi}f(v)dv+\int_{R^6} K^{\varpi,x}(v,v^{\prime})f(v)f(v^{\prime})dvdv^{\prime}\cr
%&=&-\int \nu^0 f(v+\varpi\times x)dv+\int K^0(v,v^{\prime})f(v+\varpi\times x)f(v^{\prime}+\varpi\times x)dvdv^{\prime}\cr
%&=&-\int \nu^0 \tau_{\varpi\times x} f(v)dv+\int K^0(v,v^{\prime})\tau_{\varpi\times x}f(v)\tau_{\varpi\times x}f(v^{\prime})dvdv^{\prime}\cr
&=&\int \left\{-\nu^0\tau_{\varpi\times x}f+\int K^{0,x}(v,v^{\prime})\tau_{\varpi\times x}f(v^{\prime})dv^{\prime}\right\} \tau_{\varpi\times x}f(v)dv\cr
&=& (\mathcal{L}_0\tau_{\varpi\times x}f,\tau_{\varpi\times x}f)_{L^2_v}.
\end{eqnarray*}
By the positivity estimate of the standard linearized collision operator %of the global Maxwellian $e^{-\frac{|v|^2}{2}}$
(see \cite{Glassy}), there exists a constant $\delta_0$, independent of $x$ and $\varpi$, such that
\begin{eqnarray*}
\mathcal{L}_0(\tau_{\varpi\times x}f,\tau_{\varpi\times x}f)\geq\delta_0\|(I-P_0)\tau_{\varpi\times x}f\|^2_{\nu_0}.
\end{eqnarray*}
On the other hand, the Fubini's theorem gives
\begin{eqnarray*}
\|(I-P_0)\tau_{\varpi\times x}f\|^2_{\nu_0}&=&\int (I-P_{0})f(v+\varpi\times x)(1+|v|^2)^{\frac{1}{2}}dv\cr
&=&\int (I-P_{\varpi})f(v)(1+|v-\varpi\times x|^2)^{\frac{1}{2}}dv\cr
&=&\|(I-P_{\varpi\times x})f\|^2_{\nu_{\varpi}},
\end{eqnarray*}
which yields
\begin{eqnarray*}
(\mathcal{L}_{\varpi}f,f)_{L^2_v}\geq\delta_0\|(I-P_{\varpi\times x})f\|^2_{\nu_{\varpi}}.
\end{eqnarray*}
We then multiply $e^{\frac{|\varpi\times x|^2}{2}}$ and integrate with respect to $x$ to see
\begin{eqnarray*}
(L^{\varpi}f,f)_{L^2_{x,v}}&=&\int e^{\frac{|\varpi\times x|^2}{2}}(\mathcal{L}_{\varpi}f,f)_{L^2_v}dx\cr
&\geq&\delta_0\int_{\mathbb{R}^3}e^{\frac{|\varpi\times x|^2}{2}}\|(I-P_{\varpi})f\|^2_{\nu_{\varpi}}dx\cr
&\geq& \delta_0\|(I-P_{\varpi})f\|^2_{\nu_{\varpi}}.
\end{eqnarray*}
This completes the proof.
\end{proof}
%%%%%%%%%%%%%%%%%%%%%%%%%%%%%%%%%%%%%%%%%%%%%%%%%%%%%%%%%%%%%%%%%%%%%%%%%%%%%%%%%%%%%%%%%%%%%%%%%%%%%%%%%%%%%%%
%
%
%%%%%%%%%%%%%%%%%%%%%%%%%%%%%%%%%%%%%%%%%%%%%%%%%%%%%%%%%%%%%%%%%%%%%%%%%%%%%%%%%%%%%%%%%%%%%%%%%%%%%%%%%%%%%%%
%\begin{lemma}\label{MeasureZeroMap}Let $\kappa(y)$ be a real analytic function of $y\in R^n$ in a region such that $\kappa(y^0)$ is not identically zero.
%Then the set $\{y: \kappa(y)=0\}$ has zero n-dimensional Lebesque measure.\end{lemma}\begin{proof}See \cite{GuoConti}.\end{proof}
%%%%%%%%%%%%%%%%%%%%%%%%%%%%%%%%%%%%%%%%%%%%%%%%%%%%%%%%%%%%%%%%%%%%%%%%%%%%%%%%%%%%%%%%%%%%%%
%
%                                   Gamma Estimate
%
%%%%%%%%%%%%%%%%%%%%%%%%%%%%%%%%%%%%%%%%%%%%%%%%%%%%%%%%%%%%%%%%%%%%%%%%%%%%%%%%%%%%%%%%%%%%%%
\begin{lemma}\label{Gamma Estimate} Recall (\ref{Gamma}) and (\ref{Weight}). We have
\begin{eqnarray*}
|w_{\varpi}\Gamma(g_1,g_2)(v)|\leq C(|v-\varpi\times x|+1)^{\gamma}\|w_{\varpi}g_1\|_{\infty}\|w_{\varpi}g_2\|_{\infty}.
\end{eqnarray*}
\end{lemma}
\begin{proof}
First consider the second term $\Gamma_{loss}$. We have
\begin{eqnarray*}
\int_{\mathbb{R}^3}|u-v|^{\gamma}|\sqrt{m}_{\varpi}g_2(x,u)|du\leq C\{|v-\varpi\times x|+1\}^{\gamma}\|wg_2\|_{\infty}.
\end{eqnarray*}
Hence $w\Gamma_{loss}[g_1,g_2]$ is bounded by
\begin{eqnarray*}
w_{\varpi}|g_1|\int_{\mathbb{R}^3}|u-v|^{\gamma}|\sqrt{m}_{\varpi}g_2(x,u)|du\leq Cw_{\varpi}(v)\{|v-\varpi\times x|+1\}^{\gamma}\|w_{\varpi}g_1\|_{\infty}.
\end{eqnarray*}
For $\Gamma_{gain}$, by $|u^{\prime}|^2+|v^{\prime}|^2=|u|^2+|v|^2$, $w_{\varpi}(v)\leq Cw_{\varpi}(v^{\prime})w_{\varpi}(u^{\prime})$, and
\begin{eqnarray*}
&&\int q_0(\theta)|u|^{\gamma}e^{-\frac{1}{4}|(v-\varpi\times x)u|^2}w_{\varpi}(v)|g_1(u^{\prime})||g_2(v^{\prime})|d\omega du\cr
&&\quad\leq C\int q_0(\theta)|u|^{\gamma}e^{-\frac{1}{4}|(v-\varpi\times x)-u|^2}w_{\varpi}(v^{\prime})w_{\varpi}(u^{\prime})|g_1(u^{\prime})||g_2(v^{\prime})|d\omega du\cr
&&\quad\leq C\|w_{\varpi}g_1\|_{\infty}\|w_{\varpi}g_2\|_{\infty}\int |u|^{\gamma}e^{-\frac{|u-\varpi\times x-v|^2}{2}}du\cr
&&\quad\leq C(|v-\varpi\times x|+1)^{\gamma}\|w_{\varpi}g_1\|_{\infty}\|w_{\varpi}g_2\|_{\infty}.
\end{eqnarray*}
This completes the proof.
\end{proof}

%%%%%%%%%%%%%%%%%%%%%%%%%%%%%%%%%%%%%%%%%%%%%%%%%%%%%%%%%%%%%%%%%%%%%%%%%%%%%%%%%%%%%%%%%%%%%%%%%%%%%%%%%%%%%%%%%%%%%%%%%%%%%%%%%
% %%%%%%%%%%%%%%%%%%%%%%%%%%%%%%%%%%%%%%%%%%%%%%%%%%%%%%%%%%%%%%%%%%%%%%%%%%%%%%%%%%%%%%%%%%%%%%%%%%%%%%%%%%%%%%%%%%%%%%%%%%%%%%%
% %
% %
% %                            Section L2 decay theory
% %
% %
% %%%%%%%%%%%%%%%%%%%%%%%%%%%%%%%%%%%%%%%%%%%%%%%%%%%%%%%%%%%%%%%%%%%%%%%%%%%%%%%%%%%%%%%%%%%%%%%%%%%%%%%%%%%%%%%%%%%%%%%%%%%%%%%
%%%%%%%%%%%%%%%%%%%%%%%%%%%%%%%%%%%%%%%%%%%%%%%%%%%%%%%%%%%%%%%%%%%%%%%%%%%%%%%%%%%%%%%%%%%%%%%%%%%%%%%%%%%%%%%%%%%%%%%%%%%%%%%%%
\section{$L^2$ decay theory}
In this section, we study the $L^2$ estimate of the linear Boltzmann equation:
\begin{eqnarray}\label{LinearBE}
\partial_tf+v\cdot\nabla_x f+L^{\varpi}f=0,\quad f(0,x,v)=f_0(x,v).
\end{eqnarray}
We define the boundary integration for $g(x,v),~ x\in \partial \Omega$
\begin{eqnarray*}
\int_{\gamma_{\pm}}gd\gamma=\int_{\pm v\cdot n(x)>0}|n(x)\cdot v|dS_xdv,
\end{eqnarray*}
where $dS_x$ is the standard surface measure on $\partial \Omega$. We also define
\begin{eqnarray*}
\|h\|_{\gamma}=\|h\|_{\gamma_+}+\|h\|_{\gamma_-}
\end{eqnarray*}
to be the $L^2(\gamma)$ with respect to the measure $|n(x)\cdot v|dS_xdv$.
%For fixed $x\in \partial\Omega$, denote the boundary inner product over $\gamma_{\pm}$ in $v$ as
%\begin{eqnarray*}\langle g_1,g_2\rangle_{\gamma}(t,x)=\int_{\pm v\cdot n(x)>0}g_1(t,x,v)g_2(t,x,v)|n(x)\cdot v|dv.\end{eqnarray*}
Our main theorem of this section is
%%%%%%%%%%%%%%%%%%%%%%%%%%%%%%%%%%%%%%%%%%%%%%%%%%%%%%%%%%%%%%%%%%%%%%%%%%%%%%%%%%%%%%%%%%%%
%
%                                   Theorem L2 Estimate
%
%%%%%%%%%%%%%%%%%%%%%%%%%%%%%%%%%%%%%%%%%%%%%%%%%%%%%%%%%%%%%%%%%%%%%%%%%%%%%%%%%%%%%%%%%%%%
\begin{theorem}\label{L2DecayTheorem}
Suppose $\Omega$ satisfies the assumption $\mathcal {(A)}$.
Let $f(t,x,v)\in L^2$ be the $($unique$)$ solution to the linear Boltzmann equation (\ref{LinearBE}) with trace $f_{\gamma}\in L^2_{loc}(\gamma)$.
We assume the conservation of mass, angular momentum
and energy (\ref{ConservationLaw}). Then there exists  $\lambda$ and $C>0$ such that
\begin{eqnarray*}
\sup_{0\leq t\leq\infty}\{e^{\lambda t}\|f(t)\|^2\}\leq 2\|f(0)\|^2.
\end{eqnarray*}
\end{theorem}
For this, we establish the following proposition
%%%%%%%%%%%%%%%%%%%%%%%%%%%%%%%%%%%%%%%%%%%%%%%%%%%%%%%%%%%%%%%%%%%%%%%%%%%%%%%%%%%%%%%%%%%%%%
%
%                                     I-P vs P estimate
%
%%%%%%%%%%%%%%%%%%%%%%%%%%%%%%%%%%%%%%%%%%%%%%%%%%%%%%%%%%%%%%%%%%%%%%%%%%%%%%%%%%%%%%%%%%%%%%
\begin{proposition}\label{PropI-P P Estimate}
Suppose $\Omega$ satisfies the geometric assumption $\mathcal{(A)}$.
Then there exists $M>0$ such that for any solution $f(t,x,v)$ to the linearized Boltzmann equation (\ref{LinearBE})
satisfying the specular reflection condition (\ref{specularBoundaryC_f}) and the conservation laws of mass, angular momentum
and energy (\ref{ConservationLaw}), we have the following estimate:
\begin{eqnarray}\label{I-P P Estimate}
\int^1_0\|P_{\varpi}f(f)\|_{\nu_{\varpi}}\leq \int^1_0\|(I-P_{\varpi})f(s)\|_{\nu_{\varpi}}ds.
\end{eqnarray}
\end{proposition}
\begin{proof}
We first show that Proposition \ref{PropI-P P Estimate} implies Theorem \ref{L2DecayTheorem}.\newline
\noindent{\bf Proof of Theorem \ref{L2DecayTheorem}}\newline
\noindent (i) The case of $s\in [N, t]$: We multiply $f$ to (\ref{LinearBE}) and integrate on $[N, t]$ to obtain
\begin{eqnarray*}
\|f(t)\|^2+2\int^t_N(L^{\varpi}f,f)ds=\|f(N)\|^2,
\end{eqnarray*}
which gives from Lemma \ref{positivity}
\begin{eqnarray}\label{ttoN}
\|f(t)\|\leq\|f(N)\|.
\end{eqnarray}
\noindent (ii) The case of $s\in [0, N]$: Let $f$ be a solution to (\ref{LinearBE}), then $e^{\lambda t}f$ satisfies
\begin{eqnarray}\label{LBEexp}
\{\partial_t+v\cdot \nabla_x+L\}\{e^{\lambda t}f\}-\lambda e^{\lambda t}f=0.
\end{eqnarray}
We multiply $e^{\lambda t}f$ and integrate over $0\leq s\leq N$. Then we have from
$f_{\gamma}\in L^2_{loc}(R_+; L^2(\gamma))$
\begin{eqnarray*}
&&e^{2\lambda N}\|f(N)\|^2+2\int^N_0e^{2\lambda s}(L^{\varpi}f,f)ds-\lambda\int^N_0e^{2\lambda s}\|f(s)\|^2ds\cr
&&\qquad=\|f(0)\|^2+\int^N_0\int_{\gamma_-}e^{2\lambda s}f^2(s)d\gamma ds-\int^N_0\int_{\gamma_+}e^{2\lambda s}f^2(s)d\gamma ds\cr
&&\qquad=\|f(0)\|^2.
\end{eqnarray*}
In the last line, we used the fact that the total boundary contribution vanishes for specular boundary reflection:
\begin{eqnarray*}
\int^N_0\int_{\gamma_-}e^{2\lambda s}f^2(s)d\gamma ds=\int^N_0\int_{\gamma_+}e^{2\lambda s}f^2(s)d\gamma ds.
\end{eqnarray*}
Dividing the time interval into $\cup^{N-1}_{k=0}[k,k+1)$ and letting $f_k(s,x,v)\equiv f(k+s,x,v)$ for $k=0,1,2,\cdots$, we deduce
\begin{eqnarray}\label{N-1to1}
&&e^{2\lambda N}\|f(N)\|^2+\sum^{N-1}_{k=0}\int^1_0\left\{2e^{2\lambda (k+s)}(L^{\varpi}f_k(s),f_k(s))ds-\lambda e^{2\lambda (k+s)}\|f(s)\|^2\right\}ds
=\|f(0)\|^2.%+\sum^{N-1}_{k=0}\Big\{\int^1_0\int_{\gamma_-}e^{2\lambda (k+s)}f^2_k(s)d\gamma ds-\int^1_0\int_{\gamma_+}e^{2\lambda (k+s)}f^2_k(s)d\gamma ds\Big\}.
\end{eqnarray}
We then apply the positivity estimate in Lemma \ref{positivity} to get
\begin{eqnarray*}
\int^1_0e^{2\lambda (k+s)}(L^{\varpi}f_k(s),f_k(s))ds
&\geq&\delta_0\int^1_0e^{2\lambda (k+s)}\|(I-P_{\varpi})f_k\|^2_{\nu_{\varpi}}ds\cr
&\geq&\delta_0\int^1_0e^{2\lambda k}\|(I-P_{\varpi})f_k\|^2_{\nu_{\varpi}}ds.
\end{eqnarray*}
We employ Proposition \ref{PropI-P P Estimate} as follows
\begin{eqnarray*}
\begin{split}
\int^1_0e^{2\lambda k}\|(I-P_{\varpi})f_k\|^2_{\nu_{\varpi}}ds
&=\frac{\delta_0}{2}\int^1_0e^{2\lambda k}\|(I-P_{\varpi})f_k\|^2_{\nu_{\varpi}}ds+\frac{\delta_0}{2}\int^1_0e^{2\lambda k}\|(I-P_{\varpi})f_k\|^2_{\nu_{\varpi}}ds\cr
&\geq\frac{\delta_0}{2}\int^1_0e^{2\lambda k}\|(I-P_{\varpi})f_k\|^2_{\nu_{\varpi}}ds+\frac{\delta_0}{2M}\int^1_0e^{2\lambda k}\|P_{\varpi}f_k\|^2_{\nu_{\varpi}}ds\cr
&\geq\frac{1}{2}\min\left\{\frac{\delta_0}{M}, \delta_0\right\}e^{-2\lambda}\int^t_0e^{2\lambda(k+s)}\|f\|^2_{\nu_{\varpi}}ds.
\end{split}
\end{eqnarray*}
%Note that we have used $e^{2\lambda (k+s)}\geq e^{2\lambda k}=e^{-2\lambda}e^{\lambda(k+s)}$.
We substitute this into (\ref{N-1to1}) to get
\begin{eqnarray*}
&&e^{2\lambda N}\|f(N)\|^2+
\left\{\frac{1}{2}\min\Big\{\frac{\delta_0}{M},~\delta_0\Big\}e^{-2\lambda}-C_{\nu_{\varpi}}\lambda\right\}\sum^{N-1}_{k=0}
\int^1_0e^{2\lambda(k+s)}\|f\|^2_{\nu_{\varpi}}
=\|f(0)\|^2.
%+\sum^{N-1}_{k=0}\Big\{\int^1_0\int_{\gamma_-}e^{2\lambda (k+s)}f^2_k(s)d\gamma ds-\int^1_0\int_{\gamma_+}e^{2\lambda (k+s)}f^2_k(s)d\gamma ds\Big\}.
\end{eqnarray*}
Here we have used $\|\cdot\|\leq C_{\nu_{\varpi}}\|\cdot\|_{\nu_{\varpi}}$. We choose $\lambda>0$ small enough such that 
$\frac{1}{2}\min\big\{\frac{\delta_0}{M}, \delta_0\big\}e^{-2\lambda}-C_{\nu_{\varpi}}\lambda>0$,
to obtain
\begin{eqnarray}\label{Nto0}
e^{2\lambda N}\|f(N)\|^2\leq \|f(0)\|^2.
\end{eqnarray}
Combining (\ref{ttoN}) of Case i) and (\ref{Nto0}) of Case ii), we have
\begin{eqnarray*}
e^{2\lambda t}\|f(t)\|^2\leq e^{2\lambda t}\|f(N)\|^2\leq e^{2\lambda (t-N)}e^{2\lambda N}\|f(t)\|^2\leq e^{2\lambda (t-N)}\|f(0)\|^2\leq \|f(0)\|^2.
\end{eqnarray*}
This completes the proof.
\end{proof}
\subsection{Proof of Proposition \ref{I-P P Estimate}} We prove Proposition \ref{I-P P Estimate}  by a contradiction argument.
For this, we suppose in the contrary that for each $k$, there exists $f_k$ which satisfies
the conservation laws (\ref{ConservationLaw}), the specular boundary condition (\ref{specularBoundaryC_f}) and
\begin{eqnarray}\label{Contradiction}
\int^1_0\|P_{\varpi}f_k\|^2_{\nu_{\varpi}}ds\geq k\int ^1_0\|(I-P_{\varpi})f_k\|^2_{\nu_{\varpi}}ds.
\end{eqnarray}
We normalize it as
\begin{eqnarray*}%\label{Normalized}
Z_k=\frac{f_k}{\sqrt{\int^1_0\|P_{\varpi}f_k\|^2_{\nu_{\varpi}}ds}},
\end{eqnarray*}
which satisfies
\begin{eqnarray}\label{NormalizedZ=1}
\int^1_0\|P_{\varpi}Z_k\|^2_{\nu_{\varpi}}ds=1.
\end{eqnarray}
From $\{\partial_t+v\cdot \nabla+L^{\varpi}\}f_k=0$, we have
\begin{eqnarray}\label{Nor Z Transport}
\{\partial_t+v\cdot \nabla+L^{\varpi}\}Z_k=\frac{\{\partial_t+v\cdot \nabla+L^{\varpi}\}f_k}{\sqrt{\int^1_0\|P_{\varpi}f_k\|^2ds}}=0.
\end{eqnarray}
Dividing both sides of (\ref{Contradiction}), we obatin
\begin{eqnarray}\label{Norm (I-P)Z}
\int^1_0\|(I-P_{\varpi})Z_k\|^2_{\nu_{\varpi}}ds\leq\frac{1}{k}.
\end{eqnarray}
From (\ref{NormalizedZ=1}) and (\ref{Norm (I-P)Z}), there exists $Z(t,x,v)$ such that
\begin{eqnarray}\label{such that}
Z_k\rightarrow Z\mbox{ weakly in }\int^1_0\|\cdot\|^2_{\nu_{\varpi}}ds,
\end{eqnarray}
and
\begin{eqnarray}\label{(I-P)Z_k=0}
\int^1_0\|(I-P_{\varpi})Z_k\|^2_{\nu_{\varpi}}ds\rightarrow 0.
\end{eqnarray}
On the other hand, it is straightforward to verify
\begin{eqnarray}\label{ontheotherhand}
P_{\varpi}Z_k\rightarrow P_{\varpi}Z\mbox{ weakly in }\int^1_0\|\cdot\|^2_{\nu_{\varpi}}ds.
\end{eqnarray}
Due to (\ref{such that}), (\ref{(I-P)Z_k=0}) and (\ref{ontheotherhand}), we have
%\begin{eqnarray*}P_{\varpi}Z_k\rightarrow P_{\varpi}Z \mbox{ weakly }\end{eqnarray*}and\begin{eqnarray*}
%(I-P_{\varpi})Z=0\end{eqnarray*}so that
\begin{eqnarray*}
Z(t,v,x)=\{a(t,x)+b(t,x)\cdot v+c(t,x)|v|^2\}\sqrt{m}_{\varpi}.
\end{eqnarray*}
It also implies, in the sense of distribution,
\begin{eqnarray}\label{Nor Z Transport2}
\partial_tZ+v\cdot \nabla Z=0.
\end{eqnarray}
We will first prove that, using the hyperbolic type energy method developed in \cite{GuoConti} that $Z_k$ converges strongly to $Z$ in $\int^1_0\|\cdot\|^2 ds$, and $\int^1_0\|Z\|^2ds\neq0$.
Then it will be shown that (\ref{Nor Z Transport2}) and the specular reflection boundary condition imply $Z=0$, which is a
 contradiction. First, we consider how does $Z$ look like.
\begin{lemma}\label{Lemma Max Explicit}\cite{GuoConti} There exist constants $a_0, c_1, c_2$, and constant vectors $b_0, b_1$ and $\varpi$
such that
\begin{eqnarray}\label{Max Explicit2}
\begin{split}
Z(t,x,v)=&\left(\left\{\frac{c_0}{2}|x|^2-b_0\cdot x+a_0\right\}
+\left\{-c_0tx-c_1x+\varpi\times x+b_0t+b_1\right\}\cdot v\right.\cr
&\left.+\left\{\frac{c_0t^2}{2}+c_1t+c_2\right\}|v|^2\right)\sqrt{m}_{\varpi}.
\end{split}
\end{eqnarray}
Moreover, these constants are finite:
\begin{eqnarray}\label{abcBounds}
|a_0|,|b_0|,|b_1|,|c_0|,|c_1|,|c_2|,|\varpi|<\infty.
\end{eqnarray}
\end{lemma}
\begin{proof}
See \cite{GuoConti}.
\end{proof}
%%%%%%%%%%%%%%%%%%%%%%%%%%%%%%%%%%%%%%%%%%%%%%%%%%%%%%%%%%%%%%%%%%%%%%%%%%%%%%%%%%%%%%%%%%%%%%%%
%
%                  subsection: Interior compactness
%
%%%%%%%%%%%%%%%%%%%%%%%%%%%%%%%%%%%%%%%%%%%%%%%%%%%%%%%%%%%%%%%%%%%%%%%%%%%%%%%%%%%%%%%%%%%%%%%%
\subsection{Interior compactness}
\begin{lemma}\label{Interior Cpt}
For any smooth function $\mathcal{X}$ such that $\sup\mathcal{X}\in(0,1)\times\Omega$, we have up to a subsequence
\begin{eqnarray*}
\lim_{k\rightarrow\infty}\int^1_0\|\mathcal{X}\{Z_k-Z\}(s)\|^2ds=0.
\end{eqnarray*}
\end{lemma}
\begin{proof}
We multiply the equation (\ref{Nor Z Transport}) by $\chi$ to get
\begin{eqnarray*}
\{\partial_t+v\cdot \nabla_x\}\{\mathcal{X} Z_k\}=\{[\partial_t+v\cdot \nabla_x]\mathcal{X}\}Z_k-\mathcal{X} L^{\varpi}Z_k.
\end{eqnarray*}
We first note that
\begin{eqnarray*}
\begin{split}
\{\partial_t+v\cdot \nabla_x\}\{\chi\}&\leq (1+|v|)(\sup_{[0,1]\times \Omega}|\partial_t\mathcal{X}|+\sup_{[0,1]\times \Omega}|\nabla_x\mathcal{X}|)\cr
&\leq C_{\mathcal{X}}(1+|v-\varpi\times x|+|\varpi\times x|)\cr
&\leq C_{\mathcal{X},\varpi,\Omega}(1+|v-\varpi\times x|).
\end{split}
\end{eqnarray*}
Therefore, we have
\begin{eqnarray*}
\int^1_0\|\{[\partial_t+v\cdot \nabla_x]\mathcal{X}\}Z_k-\mathcal{X} L^{\varpi}Z_k\|^2_{L^2}ds\leq \int^1_0\|Z_k\|^2_{\nu_{\varpi}}ds.
\end{eqnarray*}
Since $\int^1_0\|Z_k(s)\|^2ds<\infty$, we can deduce from the averaging lemma (\cite{D-L1, D-L2}) that
$\int\xi(t,x)Z_k(v)\phi(v)dv$ are compact in $L^2([0,1]\times \Omega)$ for any smooth cut-off function $\phi(v)$. It then follows that
\begin{eqnarray*}
\int\mathcal{X} Z_k(v)[1,v,|v|^2]\sqrt{m}_{\varpi}dv
\end{eqnarray*}
are compact in $L^2([0,1]\times \Omega)$. Therefore, up to a subsequence, the macroscopic parts of $Z_k$ satisfy
\begin{eqnarray*}
\mathcal{X} P_{\varpi}Z_k\rightarrow \mathcal{X} P_{\varpi}Z=\mathcal{X} Z ~\mbox{ strongly in }L^2([0,1]\times \Omega\times \mathbb{R}^3).
\end{eqnarray*}
Therefore, in light of $\int^1_0\|(I-P_{\varpi})Z_k(s)\|^2_{\nu_{\varpi}}\rightarrow 0$ in (\ref{(I-P)Z_k=0}),
the remaining  microscopic parts $\mathcal{X} Z_k$ satisfy $\lim_{k\rightarrow \infty}\int^1_0\|\mathcal{X}\{I-P_{\varpi}\}Z_{k}(s)\|^2ds=0$,
and our lemma follows.
\end{proof}
%%%%%%%%%%%%%%%%%%%%%%%%%%%%%%%%%%%%%%%%%%%%%%%%%%%%%%%%%%%%%%%%%%%%%%%%%%%%%%%%%%%%%%%%%%%%%%%%
%
%                  subsection: No time concentration
%
%%%%%%%%%%%%%%%%%%%%%%%%%%%%%%%%%%%%%%%%%%%%%%%%%%%%%%%%%%%%%%%%%%%%%%%%%%%%%%%%%%%%%%%%%%%%%%%%
\subsection{No time concentration}
We first establish $L^{\infty}$ in time estimate for $Z_k$ to rule out possible concentration in time, near either $t=0$ or $t=1$.
\begin{lemma}\label{NoTimeCon}
\begin{eqnarray*}
\sup_{0\leq t\leq 1, k\geq 1}\|Z_k(t)\|<\infty.
\end{eqnarray*}
\end{lemma}
\begin{proof}
We first note that $\int^1_0\|f_k(s)\|^2_{\gamma}<\infty$ implies $\int^1_0\|Z_k(s)\|^2_{\gamma}<\infty$. Therefore, by the standard $L^2$ estimate for
(\ref{Nor Z Transport2})
\begin{eqnarray}\label{that}
\|Z_k(t)\|^2+2\int^t_0(L^{\varpi}Z_k,Z_k)ds=\|Z_k(0)\|^2.
\end{eqnarray}
This gives from Lemma \ref{positivity}
\begin{eqnarray}\label{emm}
\|Z_k(t)\|\leq\|Z_k(0)\|.
\end{eqnarray}
On the other hand, we note that
\begin{eqnarray}\label{this}
\int^t_0(L^{\varpi}Z_k(t),Z_k(t))dt\leq C\int^1_0\|\{I-P_{\varpi}\}Z_k\|^2_{\nu_{\varpi}}dt\leq\frac{C}{k}.
\end{eqnarray}
Therefore, by (\ref{that}) and (\ref{this})
\begin{eqnarray}\label{83}
\begin{split}
\|Z_k(t)\|^2&=\|Z_k(0)\|^2-2\int^t_0(L^{\varpi}Z_k,Z_k)ds\cr
&\geq\|Z_k(0)\|^2-\frac{C}{k}.
\end{split}
\end{eqnarray}
Since $\int^1_0\|Z_k(t)\|^2 dt\leq \int ^t_0\|Z_k(t)\|^2_{\nu_{\varpi}}\leq C\{1+\frac{1}{k}\}$ for hard potentials,
integrating (\ref{83}) over $0\leq t\leq 1$ yields
\begin{eqnarray}\label{84}
\begin{split}
\|Z_k(0)\|^2&=\int^1_0\|Z_k(t)\|^2dt+\frac{C}{k}\cr
&\leq C\int^1_0\|Z_k(t)\|^2_{\nu_{\varpi}}dt+\frac{C}{k}\cr
&\leq C\left\{1+\frac{1}{k}\right\}+\frac{C}{k}.
\end{split}
\end{eqnarray}
Then the result follows from (\ref{emm}) and (\ref{84}).
\end{proof}
%%%%%%%%%%%%%%%%%%%%%%%%%%%%%%%%%%%%%%%%%%%%%%%%%%%%%%%%%%%%%%%%%%%%%%%%%%%%%%%%%%%%%%%%%%%%%%%%
%
%                  subsection: No boundary concentration
%
%%%%%%%%%%%%%%%%%%%%%%%%%%%%%%%%%%%%%%%%%%%%%%%%%%%%%%%%%%%%%%%%%%%%%%%%%%%%%%%%%%%%%%%%%%%%%%%%
\subsection{No boundary concentration}
In this section, we prove that there is no concentration at the boundary $\partial \Omega$. Let
\begin{eqnarray*}
\Omega_{\varepsilon^4}\equiv\{x\in\Omega: \xi(x)<-\varepsilon^4\}.
\end{eqnarray*}
To this end, we will establish a careful energy estimate near the boundary in the thin shell-like region
$[0,1]\times\{\Omega\backslash\Omega_{\varepsilon^4}\}\times R^4$. For $m>\frac{1}{2}$, for any $(x,v)$,
we define the outward moving (inwarding moving) indicator function $\chi_+(\chi_-)$ as
\begin{eqnarray*}
&&\chi_+(x,v)={\bf1}_{\Omega\backslash\Omega_{\varepsilon^4}}(x){\bf1}_{\{|v|\leq \varepsilon^{-m},n(x)\cdot v>\varepsilon\}}(v),\cr
&&\chi_-(x,v)={\bf1}_{\Omega\backslash\Omega_{\varepsilon^4}}(x){\bf1}_{\{|v|\leq \varepsilon^{-m},n(x)\cdot v<-\varepsilon\}}(v).
\end{eqnarray*}
The following lemma shows that $Z_k$ can be controlled in a shell like region near grazing phase boundary with large velocity.
%%%%%%%%%%%%%%%%%%%%%%%%%%%%%%%%%%%%%%%%%%%%%%%%%%%%%%%%%%%%%%%%%%%%%%%%%%%%%%%%%%%%%%%%%%%%%%%%%
%
%
%
%%%%%%%%%%%%%%%%%%%%%%%%%%%%%%%%%%%%%%%%%%%%%%%%%%%%%%%%%%%%%%%%%%%%%%%%%%%%%%%%%%%%%%%%%%%%%%%%%%
\begin{lemma}\label{9}
\begin{eqnarray*}
\sup_{k\geq 1}\int^1_0\int_{\Omega\backslash
\Omega_{\varepsilon^4}}
\int_{\substack{|n(x)\cdot v|\leq\varepsilon,\cr|v|\geq \varepsilon^{-m}}}
|Z_k(s,x,v)|^2dvdxds\leq C\varepsilon.
\end{eqnarray*}
\end{lemma}
\begin{proof}
Recall from (\ref{MacroPro})
\begin{eqnarray*}
P_{\varpi}Z_k=\{a_k(t,x)+v\cdot b_k(t,x)+|v|^2c_k(t,x)\}\sqrt{m}_{\varpi}.
\end{eqnarray*}
We split the estimate and use (\ref{Norm (I-P)Z}) to see
\begin{eqnarray*}
&&\int^1_0\int_{\Omega\backslash
\Omega_{\varepsilon^4}}
\int_{\substack{|n(x)\cdot v|\leq\varepsilon,\cr|v|\geq \varepsilon^{-m}}}
|Z_k(s,x,v)|^2dvdxds\cr
&&\quad\leq
\int^1_0\int_{\Omega\backslash\Omega_{\varepsilon^4}}
\int_{\substack{|n(x)\cdot v|\leq\varepsilon,\cr|v|\geq \varepsilon^{-m}}}
|P_{\varpi}Z_k|^2+|(I-P_{\varpi})Z_k|^2dvdxds\cr
&&\quad\leq
\int^1_0\int_{\Omega\backslash
\Omega_{\varepsilon^4}}
\int_{\substack{|n(x)\cdot v|\leq\varepsilon,\cr|v|\geq \varepsilon^{-m}}}
|P_{\varpi}Z_k(s,x,v)|^2dvdxds+\frac{C}{k}.
\end{eqnarray*}
The first term can be bounded by the Fubini Theorem as
\begin{eqnarray*}
&&\int^1_0\int_{\Omega\backslash
\Omega_{\varepsilon^4}}
\int_{\substack{|n(x)\cdot v|\leq\varepsilon,\cr|v|\geq \varepsilon^{-m}}}
|P_{\varpi}Z_k(s,x,v)|^{2}dvdxds\cr
&&\quad\leq \int^1_0\int_{\Omega\backslash
\Omega_{\varepsilon^4}}
\int_{\substack{|n(x)\cdot v|\leq\varepsilon,\cr|v|\geq \varepsilon^{-m}}}
|P_{\varpi}Z_k(s,x,v)|^2(1+|v-\varpi\times x|^2)^{\ell}dvdxds\cr
&&\quad\leq\int^1_0\int_{\Omega\backslash
\Omega_{\varepsilon^4}}\{|a_k(s,x)|^2+|b_k(s,x)|^2+|c_k(s,x)|^2\}\cr
&&\qquad\times\Big\{\int_{\substack{|n(x)\cdot v|\leq\varepsilon,\cr|v|\geq \varepsilon^{-m}}}
(1+|v-\varpi\times x|^2)^{\ell}\sqrt{m}_{\varpi}dv\Big\}dxds.
\end{eqnarray*}
We observe that
\begin{eqnarray*}
(1+|v-\varpi\times x|^2)^{\ell}\sqrt{m}_{\varpi}&=&(1+|v-\varpi\times x|^2)^{\ell}e^{-\frac{|v-\varpi\times x|^2}{4}}\cr
&\leq&C(1+|v|^2)^{\ell}(1+|\varpi|^2|x|^2)^{\ell}e^{-\frac{|v|^2}{4}}e^{\frac{|\varpi\times x|^2}{8}}\cr
&\leq&C_{\varpi,\Omega}(1+|v|^2)^{\ell}e^{-\frac{|v|^2}{4}}.
\end{eqnarray*}
Therefore we have
\begin{eqnarray*}
\int_{\substack{|n(x)\cdot v|\leq\varepsilon,\cr|v|\geq \varepsilon^{-m}}}
(1+|v-\varpi\times x|^2)^{\ell}\sqrt{m_{\varpi}}dv
\leq C_{\varpi,\Omega}\int_{\substack{|n(x)\cdot v|\leq\varepsilon,\cr|v|\geq \varepsilon^{-m}}}(1+|v|^2)^{\ell}\sqrt{m_{0}}dv.
%&&\quad=\int_{\substack{|n(x)\cdot v|\leq\varepsilon,\cr|v|\geq \varepsilon^{-m}}}(1+|v|^2)^{\ell}\sqrt{m_{0}}dvdxds.
\end{eqnarray*}
By a change of variable $v_{\|}=\{n(x)\cdot v\}n(x)$, and $v_{\perp}=v-v_{\|}$ for $|n(x)\cdot v|\leq \varepsilon$, the inner integral is bounded
by the sum of
\begin{eqnarray*}
\int_{|n(x)\cdot v|\leq\varepsilon}(1+|v|^2)^{\ell}\sqrt{m_{0}}dv \leq\int^{\varepsilon}_{-\varepsilon}dv_{\|}\int_{R^2}
e^{-\frac{|v_{\perp}|^2}{8}}dv_{\perp}\leq C\varepsilon
\end{eqnarray*}
and
\begin{eqnarray*}
\int_{|v|\geq\varepsilon^{-m}}\{1+|v|^2\}^{\ell+2}\sqrt{m_{\varpi}}dv\leq C\varepsilon.
\end{eqnarray*}
This completes the proof.
\end{proof}
%%%%%%%%%%%%%%%%%%%%%%%%%%%%%%%%%%%%%%%%%%%%%%%%%%%%%%%%%%%endproof
We now study the non-grazing parts $\chi_{\pm}Z_k$. For this, we fix $(x,v)\in\{\Omega\backslash \Omega_{\varepsilon^4}\times \mathbb{R}^3\}$ and
$s\in[\varepsilon, 1-\varepsilon]$.
We then define for backward in time $0\leq t\leq s$
\begin{eqnarray*}
&&\tilde{\chi}_+(t,x,v)={\bf1}_{\Omega\backslash\Omega_{\varepsilon^4}}(x-v(t-s))
{\bf1}_{|v|\leq \varepsilon^{-m},n(x-v(t-s))\cdot v>\varepsilon\}}(v),
\end{eqnarray*}
and forward in time $0\leq s\leq t:$
\begin{eqnarray*}
&&\tilde{\chi}_-(t,x,v)={\bf1}_{\Omega\backslash\Omega_{\varepsilon^4}}(x-v(t-s))
{\bf1}_{\{|v|\leq \varepsilon^{-m},n(x-(t-s)v)\cdot v<-\varepsilon\}}(v).
\end{eqnarray*}
%%%%%%%%%%%%%%%%%%%%%%%%%%%%%%%%%%%%%%%%%%%%%%%%%%%%%%%%%%%%%%%%%%%%%%%%%%%%%%%%%%%%%%%%%%%%%
%
%
%
%%%%%%%%%%%%%%%%%%%%%%%%%%%%%%%%%%%%%%%%%%%%%%%%%%%%%%%%%%%%%%%%%%%%%%%%%%%%%%%%%%%%%%%%%%%%%
\begin{lemma}{\label{s-varepsilon2=0}}
(1)~For $0\leq s-\varepsilon^2\leq t\leq s$, if $\tilde{\chi}_{+}(t,x,v)\neq 0$ then $n(x)\cdot v>\frac{\varepsilon}{2}>0$.
Moreover, $\tilde{\chi}_{+}(s-\varepsilon^2,x,v)= 0$, for $x\in\Omega\backslash \Omega_{\varepsilon^4}$.\newline
\noindent(2)~For $s\leq t \leq s+\varepsilon^2\leq 1$, if $\tilde{\chi}_{-}(t,x,v)\neq 0$ then $n(x)\cdot v<-\frac{\varepsilon}{2}<0$.
Moreover, $\tilde{\chi}_{-}(s+\varepsilon^2,x,v)=0$, for $x\in\Omega\backslash \Omega_{\varepsilon^4}$.
\end{lemma}
\begin{proof}
See \cite{GuoConti}
%\begin{eqnarray*}n(x)\cdot v&=&n(x-v(t-s))\cdot v-\{n(x-v(t-s))-n(x)\}\cdot(v-\varpi\times x)\cr
%&>&\varepsilon-\sup_{0\leq\theta\leq 1}|\nabla\{n(x-\theta v(t-s))\}|t-s||v-\varpi\times x|^2\cr
%&>&\varepsilon-C\varepsilon^{-2m+2}\cr&=&\varepsilon[1-C\varepsilon^{-2m+1}]\cr&\geq&\frac{\varepsilon}{2},\end{eqnarray*}for $2m<1$.
\end{proof}
%%%%%%%%%%%%%%%%%%%%%%%%%%%%%%%%%%%%%%%%%%%%%%%%%%%%%%%%%%%%%%%%%%%%%%%%%%%%%%%%%%%%%%%%%%%%%
%
%
%
%%%%%%%%%%%%%%%%%%%%%%%%%%%%%%%%%%%%%%%%%%%%%%%%%%%%%%%%%%%%%%%%%%%%%%%%%%%%%%%%%%%%%%%%%%%%%
\begin{lemma} We have the strong convergence
\begin{eqnarray*}
\lim_{k\rightarrow \infty}\int^1_0\|Z_k(s)-Z(s)\|^2ds=0.
\end{eqnarray*}
and $Z(s)$ is not identically zero in the sense that
\begin{eqnarray*}
\int^1_0\|Z(s)\|^2_{\nu_{\varpi}}>0.
\end{eqnarray*}
Moreover, $Z$ satisfies the specular boundary condition:
\begin{eqnarray*}
Z(t,x,v)|_{\gamma^{-}}=Z(t,x,R(x)v)|_{\gamma^+}.
\end{eqnarray*}
\end{lemma}
\begin{proof}
We multiply $\tilde{\chi}_{\pm}$ with (\ref{Nor Z Transport}) to get
\begin{eqnarray}
\{\partial_t+v\cdot\nabla_x\}\{\tilde{\chi}_{\pm} Z_k\}=-\tilde{\chi}_{\pm} L^{\varpi}Z_k.
\end{eqnarray}
Since $\int^1_0\|Z_k(t)\|^2_{\gamma}ds<\infty$, applying the $L^2$ estimate backward in time over the shell-like region
$[s-\varepsilon^2,s]\times[\Omega\backslash\Omega_{\varepsilon^4}]\times \mathbb{R}^3$ for outgoing part $\chi_+$, we obtain
\begin{eqnarray}\label{95}
\begin{split}
&\|\chi_+Z_k(s)\|^2+\int^s_{s-\varepsilon^2}\|\bar{\chi}_+Z_l(t)\|^2_{\gamma_+}dt
-\int^s_{s-\varepsilon^2}\|\bar{\chi}_+Z_k(t)\|^2_{\gamma^{\varepsilon}_+}dt\cr
&=\|\bar{\chi}_+Z_k(s-\varepsilon^2)\|^2+\int^s_{s-\varepsilon^2}\|\bar{\chi}_+Z_l(t)\|^2_{\gamma_-}dt\cr
&-\int^s_{s-\varepsilon^2}\|\bar{\chi}_+Z_k(t)\|^2_{\gamma^{\varepsilon}_-}dt
-2\int^s_{s-\varepsilon^2}(\bar{\chi}_+L^{\varpi}Z_k,Z_k)(t)dt,
\end{split}
\end{eqnarray}
where $\gamma^{\varepsilon}\equiv\{x:\chi(x)=-\varepsilon^4\}$. We observe from Lemma \ref{s-varepsilon2=0} that
$\bar{\chi}_+(x,v,s-\varepsilon^2)=0$ and $\tilde{\chi}_+=0$ on $\gamma_-$ and  $\gamma_-^{\varepsilon}$, which gives
\begin{eqnarray*}
\|\bar{\chi}_+Z_l(s-\varepsilon^2)\|=0,
\end{eqnarray*}
and
\begin{eqnarray*}
\int^s_{s-\varepsilon^2}\|\bar{\chi}_+Z_l(t)\|^2_{\gamma_-}dt=\int^s_{s-\varepsilon^2}\|\bar{\chi}_+Z_k(t)\|^2_{\gamma^{\varepsilon}_-}dt=0.
\end{eqnarray*}
For $k$ sufficiently large,  we use (\ref{(I-P)Z_k=0}) to get
\begin{eqnarray*}
\int^s_{s-\varepsilon^2}(\bar{\chi}_+L^{\varpi}Z_k,Z_k)(t)dt&\leq&\int^1_0\int_{\Omega\times \mathbb{R}^3}|L\{I-P_{\varpi}Z_k(t)\}||Z_k(t)|dxdvdt\cr
&\leq&C\Big\{\int^1_0\|\{I-P_{\varpi}Z_k(t)\|_{\nu_{\varpi}}dt\Big\}^{\frac{1}{2}}\Big\{\int^1_0\|Z_k(t)\|^2_{\nu_{\varpi}}dt\Big\}^{\frac{1}{2}}\cr
&\leq&\frac{C}{\sqrt{k}}.
\end{eqnarray*}
Therefore, (\ref{95}) reduces to
\begin{eqnarray}\label{955}
&&\|\chi_+Z_k(s)\|^2+\int^s_{s-\varepsilon^2}\|\bar{\chi}_+Z_k(t)\|^2_{\gamma_+}dt
=\int^s_{s-\varepsilon^2}\|\bar{\chi}_+Z_k(t)\|^2_{\gamma^{\varepsilon}_+}dt+\frac{C}{\sqrt{k}}.
\end{eqnarray}
By the same argument, we can obtain similar estimate for $\tilde{\chi}_-Z_k$:
\begin{eqnarray}\label{9555}
&&\|\chi_-Z_k(s)\|^2+\int_s^{s+\varepsilon^2}\|\bar{\chi}_+Z_k(t)\|^2_{\gamma_-}dt
=\int^{s+\varepsilon^2}_s\|\bar{\chi}_-Z_k(t)\|^2_{\gamma^{\varepsilon}_-}dt+\frac{C}{\sqrt{k}}.
\end{eqnarray}
We then combine (\ref{955}) and (\ref{9555}) to see
\begin{eqnarray*}
&&\int_{\Omega\backslash\Omega_{\varepsilon^4}}\int_{\substack{|n(x)\cdot v|>\varepsilon\cr|v|\leq \varepsilon^{-m}}}|Z_k(s,x,v)|^2dxdv\cr
&&\qquad=\int |\chi_-Z_k(s)|^2+|\chi_+Z_k(s)|^2dxdv\cr
&&\qquad\leq\int^{s+\varepsilon^2}_s\|\bar{\chi}_-Z_k(t)\|^2_{\gamma^{\varepsilon}_-}dt
+\int^s_{s-\varepsilon^2}\|\bar{\chi}_+Z_k(t)\|^2_{\gamma^{\varepsilon}_+}dt+\frac{C}{\sqrt{k}}\cr
&&\qquad\leq\int^{s+\varepsilon^2}_{s-\varepsilon^2}\Big\|{\bf 1}_{\{|v|\leq\varepsilon^{-m}, |n(x)\cdot v|\geq\frac{\varepsilon}{2}\}}Z_k(t)
\Big\|^2_{\gamma^{\varepsilon}}dt+\frac{C}{\sqrt{k}}.
\end{eqnarray*}
To estimate the first term, we first observe from Lemma \ref{Lemma Max Explicit}
\begin{eqnarray}\label{98}
&&\{\partial_t+v\cdot\nabla_x\}\Big\{{\bf1}_{|v|\leq \varepsilon^{-m}}(Z_k-Z)\Big\}=-{\bf1}_{|v|\leq\varepsilon^{-m}}L\{I-P_{\varpi}\}Z_k.
\end{eqnarray}
We can apply Ukai's trace theorem (theorem 5.1.1, \cite{Ukai}) to $\{1_{|v|\leq \varepsilon^{-m}}\}(Z_k-Z)$ over $\bar{\Omega}_{\varepsilon^4}$ to get
\begin{eqnarray*}
&&\int^{s+\varepsilon^2}_{s-\varepsilon^2}\Big\|{\bf1}_{\{|v|\leq \varepsilon^{-m},|v\cdot n(x)\geq\frac{\varepsilon}{2}|\}}
\{Z_k(t)-Z(t)\}\Big\|^2_{\gamma^{\varepsilon}}ds\cr
&&\quad=\int^{s+\varepsilon^2}_{s-\varepsilon^2}\Big\|{\bf1}_{\{|v|\leq \varepsilon^{-m}\}}\big\{{\bf1}_{\{|v\cdot n(x)\geq\frac{\varepsilon}{2}|\}}
\big(Z_k(t)-Z(t)\big)\big\}\Big\|^2_{\gamma^{\varepsilon}}ds\cr
&&\quad\leq C_{\varepsilon}\int^{1-\varepsilon}_{\varepsilon}\Big\|{\bf1}_{\{|v|\leq \varepsilon^{-m}\}}
\{Z_k(t)-Z_k(t)\big\}\Big\|^2_{L^2(\Omega_{\varepsilon^4}\times \mathbb{R}^3)}ds\cr
&&\quad+ C_{\varepsilon}\int^{1-\varepsilon}_{\varepsilon}\Big\|{\bf1}_{\{|v|\leq \varepsilon^{-m}\}}\{
\{L^{\varpi}\{I-P_{\varpi}\}Z_k(t)\big\}\Big\|^2_{L^2(\Omega_{\varepsilon^4}\times \mathbb{R}^3)}ds\cr
&&\quad\leq C_{\varepsilon}\int^{1-\varepsilon}_{\varepsilon}\|Z_k(t)-Z_k(t)\|^2_{L^2(\Omega_{\varepsilon^4}\times \mathbb{R}^3)}ds
+ C_{\varepsilon}\int^{1-\varepsilon}_{\varepsilon}\|\{I-P_{\varpi}\}Z_k(t)\|^2_{\nu_{\varpi}}ds\cr
&&\quad\leq C_{\varepsilon}\int^{1-\varepsilon}_{\varepsilon}\|Z_k(t)-Z_k(t)\|^2_{L^2(\Omega_{\varepsilon^4}\times \mathbb{R}^3)}ds
+ \frac{C_{\varepsilon}}{k}
\end{eqnarray*}
Therefore, for fixed $\varepsilon$, we have from the interior compactness in Lemma \ref{Interior Cpt}
\begin{eqnarray*}
\lim_{k\rightarrow \infty}\int^{s+\varepsilon^2}_{s-\varepsilon^2}\Big\|{\bf1}_{\{|v|\leq \varepsilon^{-m}, |n(x)\cdot v|\geq\varepsilon\}}
\{Z_k(t)-Z(t)\}\Big\|^2_{\gamma^{\varepsilon}}dt=0.
\end{eqnarray*}
Hence, for $k$ large, and for any $\varepsilon\leq s \leq 1-\varepsilon$, we have
\begin{eqnarray*}
&&\hspace{-0.5cm}\int_{\Omega\backslash\Omega_{\varepsilon^4}}\int_{\substack{|v|\leq \varepsilon^{-m},\cr|v\cdot n(x)|\geq\frac{\varepsilon}{2}}}
|Z_k(s,x,v)|^2dxdv\cr
&&\hspace{-0.2cm}\leq 2\int^{s+\varepsilon^2}_{s-\varepsilon^2}
\Big\|{\bf1}_{\substack{|v|\leq \varepsilon^{-m},\cr|v\cdot n(x)|\geq\frac{\varepsilon}{2}}}\{Z_k(t)-Z_k\}\Big\|^2_{\gamma^{\varepsilon}}dt
+2\int^{s+\varepsilon^2}_{s-\varepsilon^2}
\Big\|{\bf1}_{\substack{|v|\leq \varepsilon^{-m},\cr|v\cdot n(x)|\geq\frac{\varepsilon}{2}}}Z(t)\Big\|^2_{\gamma^{\varepsilon}}dt+\frac{C}{\sqrt{k}}\cr
&&\hspace{-0.2cm}\leq \varepsilon+\int^{s+\varepsilon^2}_{s-\varepsilon^2}\|Z(t)\|^2_{\gamma^{\varepsilon}}ds\cr
&&\hspace{-0.2cm}\leq\varepsilon+2\varepsilon^2\sup_{0\leq t\leq 1}\|Z(t)\|^2_{\gamma^{\varepsilon}}\cr
&&\hspace{-0.2cm}\leq\varepsilon+C\varepsilon^2.
\end{eqnarray*}
This leads to
\begin{eqnarray}\label{99}
&&\int_{\Omega\backslash\Omega_{\varepsilon^4}}\int_{\substack{|v|\leq \varepsilon^{-m},\cr|v\cdot n(x)|\geq\frac{\varepsilon}{2}}}
|Z_k(s,x,v)|^2dxdv\leq C\varepsilon.
\end{eqnarray}
for sufficiently small $\varepsilon$.
In the last line, we used the fact that $Z_k$ is smooth due to (\ref{Max Explicit2}) and its trace is
given by (\ref{Max Explicit2}) as well.\newline
{\bf Proof of the strong compactness of  $Z_k$:}
By interior compactness, we have
\begin{eqnarray*}
&&\int^{1}_{0}\int_{\Omega\times \mathbb{R}^3}|Z_k(s,x,v)-Z(s,x,v)|^2dxdvds\cr
&&\qquad=\Big(\int^{\varepsilon}_0+\int^{1-\varepsilon}_{\varepsilon}+\int^1_{1-\varepsilon}\Big)\int_{\Omega\times \mathbb{R}^3}|Z_k(s,x,v)-Z(s,x,v)|^2dxdvds\cr
&&\qquad\leq C\varepsilon+\int^{1-\varepsilon}_{\varepsilon}\int_{\Omega\times \mathbb{R}^3}|Z_k(s,x,v)-Z(s,x,v)|^2dxdvds.
\end{eqnarray*}
We split the main part $\int^{1-\varepsilon}_{\varepsilon}$ as
\begin{eqnarray*}
&&\hspace{-0.5cm}\int^{1-\varepsilon}_{\varepsilon}\int_{\Omega\times \mathbb{R}^3}|Z_k(s,x,v)-Z(s,x,v)|^2dxdvds\cr
&&=2\int^{1-\varepsilon}_{\varepsilon}\!\!\!\!\int_{\Omega\backslash\Omega_{\varepsilon^4}}\!\int_{\substack{|v|\geq \varepsilon^{-m},\cr|v\cdot n(x)|\leq\frac{\varepsilon}{2}}}
|Z_k(s,x,v)|^2dxdvds
+2\int^{1-\varepsilon}_{\varepsilon}\!\!\!\!\int_{\Omega\backslash\Omega_{\varepsilon^4}}\!\int_{\substack{|v|\leq \varepsilon^{-m},\cr or |v\cdot n(x)|\geq\frac{\varepsilon}{2}}}
\!\!\!|Z_k(s,x,v)|^2dxdvds\cr
&&+2\int^{1-\varepsilon}_{\varepsilon}\int_{\Omega\backslash\Omega_{\varepsilon}}
|Z_k(s,x,v)|^2dxdvds
+\int^{1-\varepsilon}_{\varepsilon}\int_{\Omega_{\varepsilon^4}}|Z_k(s,x,v)-Z(s,x,v)|^2dxdvds.
\end{eqnarray*}
By (\ref{99}) and Lemma \ref{9}, the first two terms can be bounded by $C\varepsilon$.
 The third term is
bounded by
\begin{eqnarray*}
\int_{\Omega\backslash\Omega_{\varepsilon^4}}\int|Z(t,x,v)|^2dvdx\leq C|\Omega\backslash\Omega_{\varepsilon^4}|\leq C\varepsilon.
\end{eqnarray*}
The last term goes to zero as $k\rightarrow \infty$ by Lemma \ref{Interior Cpt}.
We hence deduce the strong convergence
\begin{eqnarray*}
\int^{1}_{0}\int_{\Omega_{\varepsilon}\times \mathbb{R}^3}|Z_k(s,x,v)-Z(s,x,v)|^2dxdvds\rightarrow 0.
\end{eqnarray*}
%%%%%%%%%%%%%%%%%%%%%%%%%%%%%%%%%%%%%%%%Positivity%%%%%%%%%%%%%%%%%%%%%%%%%%%%%%%%%%%%%%%%%%%%%%
{\bf Proof of the positivity of Z:}
We first note that both $\int^1_0\|P_{\varpi}Z_k\|^2dt$ and $\int^1_0\|P_{\varpi}Z_k\|^2_{\nu_{\varpi}}dt$ are equivalent to
\begin{eqnarray*}
\int^1_0|a_k(s,x)|^2dxds+\int^1_0|b_k(s,x)|^2dxds+\int^1_0|c_k(s,x)|^2dxds.
\end{eqnarray*}
Hence we have
\begin{eqnarray*}
\int^1_0\|P_{\varpi}Z_k\|^2ds\geq C\int^1_0\|P_{\varpi}Z_k\|^2_{\nu_{\varpi}}ds= C,
\end{eqnarray*}
where we used $\int^1_0\|P_{\varpi}Z_k(s)\|^2_{\nu_{\varpi}}=1$.  Therefore, we have from the strong convergence of $Z_k$ into $Z$ in
$\int^1_0\|\cdot\|^2ds$
\begin{eqnarray*}
\int^1_0\|Z(s)\|^2ds=\lim_{k\rightarrow\infty} \int^1_0\|Z_k(s)\|^2ds\geq C>0.
\end{eqnarray*}
%%%%%%%%%%%%%%%%%%%%%%%%%%%%%%Boundary Condition%%%%%%%%%%%%%%%%%%%%%%%%%%%%%%%%%%%%%%%%%5
{\bf Boundary condition:}
Recalling (\ref{98}) and $\int^1_0\|Z_k(t)-Z(t)\|^2dt\rightarrow 0$, we use Ukai's theorem to conclude, for any fixed $\varepsilon>0$,
\begin{eqnarray*}
&&\lim_{k\rightarrow \infty}\int^{1}_{0}\left\|{\bf1}_{\{|v|\leq \varepsilon^{-m},|v\cdot n(x)\geq\frac{\varepsilon}{2}|\}}
\{Z_k(t)-Z(t)\}\right\|^2_{\gamma}ds\cr
&&\qquad\leq C \lim_{k\rightarrow \infty}\left[\int^{1}_{0}\Big\|{\bf1}_{\{|v|\leq \varepsilon^{-m}|\}}
\{Z_k(t)-Z(t)\}\right\|^2_{\gamma}ds\cr
&&\qquad+\int^{1}_{0}\left\|[\partial_t+v\cdot\nabla_x]{\bf1}_{\{|v|\leq \varepsilon^{-m}|\}}\{Z_k(t)-Z(t)\}\Big\|^2ds\right]\cr
&&\qquad\leq C\int^{1}_{0}\left\|{\bf1}_{\{|v|\leq \varepsilon^{-m}|\}}\{L^{\varpi}\{I-P_{\varpi}\}Z_k(t)\}\right\|^2ds\cr
&&\qquad=0.
\end{eqnarray*}
$Z_k(t,x,v)=Z_k(t,x,R(x)v)$  letting $k\rightarrow\infty$, we deduce that $Z$ satisfies
$Z(t,x,v)=Z(t,x,R(x)v)$ for $\{v\cdot n(x)\geq\frac{\varepsilon}{2}\}$.
Therefore $Z(t,x,v)=Z(t,x,R(x)v)$ by the continuity of $Z$.
\end{proof}
%%%%%%%%%%%%%%%%%%%%%%%%%%%%%%%%%%%%%%%%%%%%%%%%%end of proof%%%%%%%%%%%%%%%%%%%%%%%%%%%%%%%%%%%%%%%%%%%%%
%%%%%%%%%%%%%%%%%%%%%%%%%%%%%%%%%%%%%%%%%%%%%%%%%%%%%%%%%%%%%%%%%%%%%%%%%%%%%%%%%%%%%%%%%%%%%%%%%%%%%%%%%%%
%
%        subsection: Boundary condition leads to Z=0
%
%%%%%%%%%%%%%%%%%%%%%%%%%%%%%%%%%%%%%%%%%%%%%%%%%%%%%%%%%%%%%%%%%%%%%%%%%%%%%%%%%%%%%%%%%%%%%%%%%%%%%%%%%%%
\subsection{Boundary condition leads to $Z=0$}
We now show that $Z$ has to be zero, which leads to a contradiction.
For any fixed t, we recall
\begin{eqnarray}
&&\{\partial_t+v\cdot\nabla_x\}\Big\{{\bf1}_{|v|\leq \varepsilon^{-m}}(Z_k-Z)\Big\}=-{\bf1}_{|v|\leq\varepsilon^{-m}}L\{I-P_{\varpi}\}Z_k,
\end{eqnarray}
and apply Ukai's trace theorem over $[0,t]\times \Omega\times \mathbb{R}^3$ to get, for any $0\leq t\leq 1$,
\begin{eqnarray*}
\lim_{k\rightarrow \infty}\|{\bf1}_{|v|\leq \varepsilon^{-1}}\{Z_k(t)-Z(t)\}\|=0.
\end{eqnarray*}
Therefore,
\begin{eqnarray}\label{104}
\begin{split}
\qquad\int Z(t)\sqrt{m}_{\varpi}&=\lim_{k\rightarrow \infty}\int_{|v|\leq \varepsilon^{-1}}Z_k(t)\sqrt{m}_{\varpi}+
\lim_{k\rightarrow \infty}\int_{|v|\geq \varepsilon^{-1}}Z_k(t)\sqrt{m}_{\varpi},\cr
\qquad\int Z(t)|v|^2\sqrt{m}_{\varpi}&=\lim_{k\rightarrow \infty}\int_{|v|\leq \varepsilon^{-1}}Z_k(t)|v|^2\sqrt{m}_{\varpi}+
\lim_{k\rightarrow \infty}\int_{|v|\geq \varepsilon^{-1}}Z_k(t)|v|^2\sqrt{m}_{\varpi}.
\end{split}
\end{eqnarray}
We note that $C\|Z_k(t)\|\int_{|v|\geq\frac{1}{\varepsilon}}m^{\frac{1}{4}}_{\varpi}=C\varepsilon$,
from the $L^{\infty}$ estimates in Lemma \ref{NoTimeCon}. This gives
\begin{eqnarray*}
\lim_{k\rightarrow \infty}\int_{|v|\geq \varepsilon^{-1}}Z_k(t)\sqrt{m}_{\varpi}=\lim_{k\rightarrow \infty}\int_{|v|\geq \varepsilon^{-1}}Z_k(t)|v|^2\sqrt{m}_{\varpi}=0.
\end{eqnarray*}
We therefore obtain that for all $t$
\begin{eqnarray*}
\int \left\{\frac{c_0}{2}|x|^2-b_0\cdot x+a_0\right\}\sqrt{m}_{\varpi}+\left\{\frac{c_0t^2}{2}+c_1t+c_2\right\}|v|^2\sqrt{m}_{\varpi}\equiv0,\cr
\int \left\{\frac{c_0}{2}|x|^2-b_0\cdot x+a_0\right\}|v|^2\sqrt{m}_{\varpi}+\left\{\frac{c_0t^2}{2}+c_1t+c_2\right\}|v|^4\sqrt{m}_{\varpi}\equiv0.
\end{eqnarray*}
This implies $c_0=c_1=0$. We therefore have $b=\varpi\times x+b_0t+b_1$. Now, from the specular reflection, we have for any $x\in \partial\Omega$, $b(t,x)\cdot n(x)=0$ or
\begin{eqnarray*}
\{\varpi\times x+b_0t+b_1\}\cdot n(x)\equiv0.
\end{eqnarray*}
%Hence $b_0=0$ for all $x\in \partial\Omega$ and
%\begin{eqnarray}\label{1088}
%\{\omega\times x\}\cdot n(x)+b_1\cdot n(x)=0.
%\end{eqnarray}
From here, by the same argument employed to derive the rotational local Maxwellian in the Appendix, we arrive at
\begin{eqnarray}\label{109}
Z=\varpi\times x \cdot v\sqrt{m}_{\varpi}.
\end{eqnarray}
Then, from the conservation of the angular momentum (\ref{ConservationLaw}), we have
\begin{eqnarray*}
\int_{\Omega\times \mathbb{R}^3} (x\times v) Z(t)\sqrt{m}_{\varpi}dxdv=\lim_{k\rightarrow\infty}\int_{\Omega\times \mathbb{R}^3} (x\times v) Z_k(t)\sqrt{m}_{\varpi}dxdv=0.
\end{eqnarray*}
%obtained by taking the limit $k\rightarrow\infty$ for the same expression for $Z_k$ (see the proof of (\ref{104})).
Therefore, we combine (\ref{109}) to get
\begin{eqnarray*}
\int_{\Omega\times \mathbb{R}^3} (x\times v)(\varpi\times x\cdot v)\sqrt{m}_{\varpi}dxdv=0.
\end{eqnarray*}
We now take inner product with $\varpi$ to obtain
\begin{eqnarray*}
0=\varpi\cdot\int_{\Omega\times \mathbb{R}^3} (x\times v)(\varpi\times x\cdot v)\sqrt{m}_{\varpi}dxdv=\int\{\varpi\times x\cdot v\}^2\sqrt{m}_{\varpi}dv.
\end{eqnarray*}
Therefore $\varpi=0$ and $Z\equiv0$.
%%%%%%%%%%%%%%%%%%%%%%%%%%%%%%%%%%%%%%%%%%%%%%%%%%%%%%%%%%%%%%%%%%%%%%%%%%%%%%%%%%%%%%%%%%%%%%%%
%
%                  subsection: No time concentration
%
%%%%%%%%%%%%%%%%%%%%%%%%%%%%%%%%%%%%%%%%%%%%%%%%%%%%%%%%%%%%%%%%%%%%%%%%%%%%%%%%%%%%%%%%%%%%%%%%
\section{$L^{\infty}$ decay theory}
In this section, we establish the weighted $L^{\infty}$-estimate which is crucial to control the nonlinear perturbation.
For this, we set $h(t,x,v)=w_{\varpi}f(t,x,v)$ and study the equivalent linear Boltzmann equation
\begin{eqnarray*}
&&\{\partial_t+v\cdot \nabla_x-\nu_{\varpi}+K^{\varpi,x}_{w_{\varpi}}\}h=0,\cr
&&\quad h(0,v,x)=h_0(x,v)\equiv w_{\varpi}f_0,
\end{eqnarray*}
where $K^{\varpi,x}_{w_{\varpi}}$ is defined as
\begin{eqnarray*}
K^{\varpi,x}_{w_{\varpi}}h&=&w_{\varpi}K^{\varpi,x}\left(\frac{h}{w_{\varpi}}\right)\cr
&=&(1+|v-\varpi\times x|^2)^{\frac{1}{2}}\int_{\mathbb{R}^3} K^{\varpi,x}(v,v^{\prime})\frac{h(t,x,v^{\prime})}{(1+|v^{\prime}-\varpi\times x|^2)^{\frac{1}{2}}}dv^{\prime}.
\end{eqnarray*}
\begin{definition} Let $\Omega$ satisfies the geometric assumption $\mathcal{(A)}$. Fix any point
$(t,x,v)\in\gamma_0\cap\gamma_-$ and define $(t_0,x_0,v_0)=(t,x,v)$ and for $k\geq 1$
\begin{eqnarray}
(t_{k+1},x_{k+1},v_{k+1})=(t_k-t_{\bf b}(t_k,x_k,v_k),x_{\bf b}(x_k,v_k),R(x_{k+1})v_k)
\end{eqnarray}
And we define the specular back-time cycle
\begin{eqnarray}\label{SpecularBackTimeCycle}
&&X_{\bf{cl}}(s)\equiv\sum^{\infty}_{k=1}{\bf 1}_{[t_{k+1},t_{k})}(s)\{x_k+v_k(s-t_k)\},\quad
V_{\bf{cl}}(s)\equiv\sum^{\infty}_{k=1}{\bf 1}_{[t_{k+1},t_{k})}(s)v_k.
\end{eqnarray}
\end{definition}
%%%%%%%%%%%%%%%%%%%%%%%%%%%%%%%%%%%%%%%%%%%%%%%%%%%%%%%%%%%%%%%%%%%%%%%%%%%%%%%%%%%%%%%
%
%
%%%%%%%%%%%%%%%%%%%%%%%%%%%%%%%%%%%%%%%%%%%%%%%%%%%%%%%%%%%%%%%%%%%%%%%%%%%%%%%%%%%%%%%
\begin{lemma}\label{Abst}
Let $\mathcal{M}$ be an operator on $L^{\infty}(\gamma_+)\rightarrow L^{\infty}(\gamma_-)$ such that
$\|\mathcal{M}\|_{\mathcal{L}(L^{\infty},L^{\infty})}=1$. Then for any $\varepsilon>0$, there exists
$h(t)\in L^{\infty}$ and $h_{\gamma}\in L^{\infty}$ solving
\begin{eqnarray}
&&\{\partial_t+v\cdot\nabla_x+\nu_{\varpi}\}h=0,\quad h_{\gamma_-}=(1-\varepsilon)Mh_{\gamma_+},
\quad h(0,x,v)=h_0\in L^{\infty}.
\end{eqnarray}
\end{lemma}
\begin{proof}
See \cite{GuoConti}.
\end{proof}
%%%%%%%%%%%%%%%%%%%%%%%%%%%%%%%%%%%%%%%%%%%%%%%%%%%%%%%%%%%%%%%%%%%%%%%%%%%%%%%%%%%%%%%
%
%
%%%%%%%%%%%%%%%%%%%%%%%%%%%%%%%%%%%%%%%%%%%%%%%%%%%%%%%%%%%%%%%%%%%%%%%%%%%%%%%%%%%%%%%
\begin{lemma}\label{G-formula}
Let $\Omega$ satisfies the geometric assumption $\mathcal{(A)}$.
Let $h_0\in L^{\infty}$ and $G(t)h_0$ solves
\begin{eqnarray}\label{transport}
\{\partial_t+v\cdot\nabla_x+\nu\}\{G(t)h_0\}=0,~\{G(0)h_0\}=h_0,
\end{eqnarray}
 with specular boundary condition $\{G(t)h_0\}(t,x,v)=\{G(t)h_0\}(t,x,R(x)v)$ for $x\in\partial\Omega$.
Then for almost all $(x,v)\notin\gamma_0$ and $t\in [t_{m+1},t_m)$
\begin{eqnarray*}\label{G-Explicit}
\begin{split}
\quad\displaystyle\{G(t)h_0\}(t,x,v)
&=e^{-\sum^{m}_{k=1}\nu_{\varpi}(x_k,v_k)(t_{k}-t_{k+1})}h_0(X_{{\bf cl}}(0),V_{{\bf cl}}(0))\cr
&=\sum^{\infty}_{k=1}{\bf 1}_{[t_{k+1},t_{k})}(0)e^{-\sum^{k}_{j=1}\nu_{\varpi}(x_j,v_j)(t_j-t_{j+1})} h_0(x_k-t_kv_k,v_k).
\end{split}
\end{eqnarray*}
Here, we defined $t_k=0$ if $t_k<0$. %and $\mathcal{V}_m$ denotes
%\begin{eqnarray}\label{mathcalV}
%\mathcal{V}(t,x,v)\equiv -\sum^{m}_{k=1}\nu_{\varpi}(x_k,v_k)(t_{k-1}-t_{k})-\nu_{\varpi}(x_m,v_m)t_m.
%\end{eqnarray}
Moreover, $e^{\nu_0t}\|G(t)h_0\|_{\infty}\leq \|h_0\|_{\infty}.$
\end{lemma}
\begin{proof}
For any $\varepsilon>0$, by Lemma \ref{Abst}, there exists a solution $h^{\varepsilon}$ of
\begin{eqnarray*}
&&\{\partial_t+v\cdot\nabla_x+\nu^{\varpi}(x,v)\}h^{\varepsilon}(t,x,v)=0,\cr
&&\qquad h^{\varepsilon}(t,x,v)|_{\gamma_-}=(1-\varepsilon)h^{\varepsilon}(t,x,R(x)v),\cr
&&\qquad h^{\varepsilon}(0,x,v)=h_0\in L^{\infty}.
\end{eqnarray*}
with $\|h^{\varepsilon}(t,\cdot)\|_{\infty}<\infty$ and $\sup_t\|h^{\varepsilon}_{\gamma}(t,\cdot)\|_{\infty}<\infty$.
Such a solution is necessarily unique.
This is because we can choose $w^{-2}_{\varpi}\{1+|v-\varpi\times x|\}\in L^1$ so that
$f^{\varepsilon}=\frac{h^{\varepsilon}}{w_{\varpi}}\in L^2$ is a $L^2$ solution to the same equation with the same boundary condition:
\begin{eqnarray*}
&&\{\partial_t+v\cdot\nabla_x+\nu^{\varpi}(x,v)\}f^{\varepsilon}(t,x,v)=0,\cr
&&\qquad f^{\varepsilon}(t,x,v)|_{\gamma_-}=(1-\varepsilon)f^{\varepsilon}(t,x,R(x)v),\cr
&&\qquad f^{\varepsilon}(0,x,v)=f_0\in L^{\infty}.
\end{eqnarray*}
with an additional property
$\int^t_0\|f^{\varepsilon}(s)\|^2_{\gamma}ds<\infty$. Then uniqueness follows from the energy identity for $f^{\varepsilon}$.\newline
\indent Given any point $(t,x,v)\notin \gamma^0$ and its back time cycle $[X_{\bf cl}(s), V_{\bf cl}(s)]$,
we notice $|V_{\bf cl}(s)|=|v|$ and for $t_{k+1}\leq s\leq t_k$,
\begin{eqnarray*}\label{wxvInvariant1}
\begin{split}
\varpi\times X_{\bf cl}(s)\cdot V_{\bf cl}(s)&=\Big( X_{\bf cl}(s)\times V_{\bf cl}(s)\Big)\cdot \varpi\cr
&=\big\{(x_{k}+(s-t_k)v_k)\times v_{k}\big\}\cdot\varpi\cr
&=(x_{k}\times v_{k})\cdot \varpi\cr
&=\varpi\times x_k\cdot v_k.
\end{split}
\end{eqnarray*}
Therefore, we have $\nu^{\varpi}(X_{\bf cl}(s), V_{\bf cl}(s))=\nu^{\varpi}(x_k,v_k)$ on $[t_{k+1},t_k]$ and the following identity holds
\begin{eqnarray*}
\frac{d}{ds}G(s)h_0=-\nu_{\varpi}(x_k,v_k)G(s)h_0
\end{eqnarray*}
along the back time cycle $[X_{\bf cl}(s), V_{\bf cl}(s)]$ for $t_{k+1}\leq s<t_k$.
Together with the boundary condition at $s=t_k$, and part (4) of Lemma \ref{1234}, we deduce that for almost every $(x,v)$,
$\frac{d}{ds}G(s)h_0$ is constant along its back-time cycle $[X_{\bf cl}(s), V_{\bf cl}(s)]$. If $(x,v)\in\bar{\Omega}\times \mathbb{R}^3\backslash \gamma_0$,
then $t_{\mathbf{b}}(t,v)>0$, and
\begin{eqnarray*}
h^{\varepsilon}(t,x,v)=\sum_{k}{\bf 1}_{[t_{k+1},t_k)}(0)(1-\varepsilon)^k
e^{-\sum^{k}_{j=1}\nu_{\varpi}(x_j,v_j)(t_j-t_{j+1})}h_0(x_k-t_k v_k, v_k),
\end{eqnarray*}
where we put $t_m=0$ if $t_m\leq 0$.\newline
\indent We now show that for fixed $(t,x,v)$, the number of bounces $k$ is finite. Since $(x,v)\in \bar\Omega\times \mathbb{R}^3\backslash\gamma_0$,
by (\ref{functional}), we conclude $\alpha (t)>0$,. By repeatedly applying velocity Lemma \ref{Velocity Lemma} along the back-time cycle $[X_{\bf cl}(s), V_{\bf cl}(s)]$, we have for all $k\geq 1$,
\begin{eqnarray*}
e^{-Ct_k}\alpha(t_k)\geq e^{-Ct_{k-1}}\alpha(t_{k-1})\geq\ldots\geq e^{-Ct}\alpha(t)>0.
\end{eqnarray*}
Then, since $\alpha(t_k)=\{v_k\cdot\nabla\xi(x_k)\}^2$,  we have
\begin{eqnarray}\label{163}
\{v_k\cdot n(x_k)\}^2\geq C\alpha(t)>0,
\end{eqnarray}
for all $k\geq1$, where $C$ depends on $t$ and $v$. Therefore by Lemma \ref{1234} (3), we have $t_k-t_{k-1}\geq\frac{\delta(t)}{C(t,v)|v|^2}$.
From this we can conclude that the summation over k is finite.\newline
\indent Now, since  $e^{\nu_0t}\|h^{\varepsilon}\|_{\infty}$ and $\sup_{t\geq s,\gamma_-}|h^{\varepsilon}(t,x,v)|$ are uniformly bounded for all $\varepsilon$:
\begin{eqnarray*}
e^{\nu_0t}\|h^{\varepsilon}\|_{\infty}\leq \|h_0\|_{\infty},
\quad\sup_{t\geq s,\gamma_-}|h^{\varepsilon}(t,x,v)|\leq\sup_{t\geq s,\gamma_+}|h^{\varepsilon}(t,x,v)|\leq \|h_0\|_{\infty},
\end{eqnarray*}
we can construct the solution $h$ to (\ref{transport}) with the original specular reflection boundary condition by taking
weak-* limit: $h(t,x,v)=\lim_{\varepsilon}h^{\varepsilon}(t,x,v)$ and $h_{\gamma}(t,x,v)=\lim_{\varepsilon}h^{\varepsilon}_{\gamma}(t,x,v)$.
We thus deduce our lemma by letting $\varepsilon\rightarrow 0$. Once again, such a solution $h(t,x,v)$ is necessarily unique in the $L^{\infty}$
class because $f_{\gamma}=\frac{h_{\gamma}}{w_{\varpi}}\in L^2_{loc}(L^2(\gamma))$.
\end{proof}
%%%%%%%%%%%%%%%%%%%%%%%%%%%%%%%%%%%%%%%%%%%%%%%%%%%%%%%%%%%%%%%%%%%%%%%%%%%%%%%%%%%%%%%
%
%
%%%%%%%%%%%%%%%%%%%%%%%%%%%%%%%%%%%%%%%%%%%%%%%%%%%%%%%%%%%%%%%%%%%%%%%%%%%%%%%%%%%%%%%
\begin{lemma}\label{continuity}
Let $\xi$ be convex. Let $h_0$ be continuous in $\bar{\Omega}\times \mathbb{R}^3\backslash \gamma_0$ and $q(t,x,v)$ be continuous
in the interior of $[0,\infty)\times\Omega\times\times \mathbb{R}^3$ and
\begin{eqnarray*}
\sup_{(0,\infty]\times\Omega \mathbb{R}^3}\left|\frac{q(t,x,v)}{\nu^{\varpi}(t,v,x)}\right|<\infty.
\end{eqnarray*}
Assume that on $\gamma_-$, $h_0(x,v)=h_0(x,R(x)v)$. Then the specular solution $h(t,x,v)$ to
\begin{eqnarray}
\{\partial_t+v\cdot\nabla_x+\nu^{\varpi}\}h=q,\qquad h(0,x,v)=h_0(x,v)
\end{eqnarray}
is continuous on
$[0,\infty)\times \{\bar{\Omega}\times \mathbb{R}^3\backslash\gamma_0\}$.
\end{lemma}
\begin{proof}
Take any point $(t,x,v)\notin [0,\infty)\times \gamma_0$ and consider its specular back-time $[X_{\bf cl}(s), V_{\bf cl}(s)]$.
By repeatedly applying Lemma \ref{Velocity Lemma} and Lemma \ref{1234}, it follows that
$t_k(t,x,v), x_k(t,x,v)$ and $v_k(t,x,v)$ are all smooth functions of $(t,x,v)$. We assume that
$t_{m+1}\leq 0<t_m$, then $h(t,x,v)$ is given by
\begin{eqnarray}\label{133}
\begin{split}
h(t,x,v)&=e^{-\mathcal{V}(t,x,v)}h_0(x_m-t_m v_m,v_m)\cr
&+\sum^{m-1}_{k=0}\int^{t_k}_{t_{k+1}}e^{-\mathcal{V}(t-s,x,v)}q(s,x_k+(s-t_k)v_k,v_k)ds\cr
&+\int^{t_{m}}_0e^{-\mathcal{V}(t-s,x,v)}q(s,x_m+(s-t_m)v_m, v_m)ds,
\end{split}
\end{eqnarray}
where we used the following simplified notation for $t\in [t_{k+1},t_k]$:
\begin{eqnarray*}
\mathcal{V}(t, x, v)&=&\nu^{\varpi}( x, v)(t-t_1)+\nu^{\varpi}(x_1,v_1)(t_1-t_{2})+\cr
&&\cdots+\nu^{\varpi}(x_{m-1},v_{m-1})(t_{m-1}-t_m)+\nu^{\varpi}(x_m,v_m)t_m.
\end{eqnarray*}
For any point $(\bar t,\bar x,\bar v)$ which is close to $(t,x,v)$, we now show that $h(\bar{t},\bar{x},\bar{v})$ is
close to $h(t,x,v)$ by separating two different cases.\newline
\noindent{\bf Case I: $t_{m+1}<0$:} or equivalently, $x_m+(s-t_m)v_m\in\Omega$, away from the boundary. By continuity, $\bar t_{m+1}<0$. Therefore,
we have
\begin{eqnarray}\label{134}
\begin{split}
h(\bar t,\bar x,\bar v)&=e^{-\mathcal{V}(\bar t,\bar x,\bar v)}h_0(\bar x_m-\bar t_m \bar v_m,\bar v_m)\cr
&+\sum^{m-1}_{k=0}\int^{\bar t_k}_{\bar t_{k+1}}e^{-\mathcal{V}(\bar t-s,\bar x,\bar v)}q(s,\bar x_k+(s-\bar t_k)\bar v_k,\bar v_k)ds\cr
&+\int^{\bar t_m}_0e^{-\mathcal{V}(\bar t-s,\bar x,\bar v)}q(s,\bar x_m+(s-\bar t_m)\bar v_m, \bar v_m)ds,
\end{split}
\end{eqnarray}
with
\begin{eqnarray*}
\mathcal{V}(\bar t,\bar x,\bar v)&=&\nu^{\varpi}( \bar x, \bar v)(\bar t-\bar t_1)+\nu^{\varpi}(\bar x_1,\bar v_1)(\bar t_1-\bar t_{2})+\cr
&&\cdots+\nu^{\varpi}(\bar x_{m-1},\bar v_{m-1})(\bar t_{m-1}-\bar t_m)+\nu^{\varpi}(\bar x_m,\bar v_m)\bar t_m.
\end{eqnarray*}
Therefore, $h(\bar t,\bar x,\bar v)\rightarrow h(t,x,v)$ follows from the continuity of
$\bar t _{\ell}\rightarrow t_{\ell}, \bar x _{\ell}, \rightarrow x _{\ell},  \bar v_{\ell}\rightarrow v _{\ell}$.\newline
\noindent{\bf Case II: $t_{m+1}=0$:}, or equivalently, $x_m+(s-t_m)v_m\in\partial\Omega$. From (\ref{163}), $(x_{k+1}v_k)\notin \gamma_0$. Then
by continuity, we know that $\bar t_m>0$, and $\bar t_{m+1}$ is close to zero. However, if $\bar t_{m+1}>0$, then $\bar t_{m+2}<0$, due to
$t_{m+2}<t_{m+1}=0$. Therefore $h(\bar t,\bar x,\bar v)$ is given by
\begin{eqnarray}\label{135}
\begin{split}
h(\bar t,\bar x,\bar v)&=e^{-\mathcal{V}(\bar t,\bar x,\bar v)}h_0(\bar x_{m+1}-\bar t_{m+1} \bar v_{m+1},\bar v_{m+1})\cr
&+\sum^{m}_{k=0}\int^{\bar t_k}_{\bar t_{k+1}}e^{-\mathcal{V}(\bar t-s,\bar x,\bar v)}q(s,\bar x_k+(s-\bar t_k)\bar v_k,\bar v_k)ds\cr
&+\int^{\bar t_{m+1}}_0e^{-\mathcal{V}( \bar t-s, \bar x,\bar v)}q(s, \bar x_{m+1}+(s- \bar t_{m+1}) \bar v_{m+1},  \bar v_{m+1})ds,
\end{split}
\end{eqnarray}
with
\begin{eqnarray*}
\mathcal{V}(\bar t,\bar x,\bar v)&=&\nu^{\varpi}( \bar x, \bar v)(\bar t-\bar t_1)+\nu^{\varpi}(\bar x_1,\bar v_1)(\bar t_1-\bar t_{2})+\cr
&&\cdots+\nu^{\varpi}(\bar x_{m},\bar v_{m})(\bar t_{m}-\bar t_{m+1})+\nu^{\varpi}(\bar x_{m+1},\bar v_{m+1})\bar t_{m+1}.
\end{eqnarray*}
We notice that since $\bar t_{m+1}\rightarrow 0$, the $q$-integrals of (\ref{135}) tends to the $q$-integrals of (\ref{133}) because of
$\mathcal{V}(\bar t,\bar x,\bar v)\rightarrow \mathcal{V}( t, x,v)$ and
$\int^{t_m}_0=\int^{t_m}_{t_{m+1}}$.
On the other hand, since
\begin{eqnarray*}
\bar x_{m+1}-\bar t_{m+1}\bar v_{m+1}\rightarrow  x_{m+1},\quad \bar v_{m+1}\rightarrow v_{m+1}=R(x_{m})v_{m},
\end{eqnarray*}
we have
\begin{eqnarray*}
h_0(\bar{x}_{m+1}-\bar{t}_{m+1}\bar{v}_{m+1},\bar{v}_{m+1})\rightarrow h_0(x_{m+1},R(x_{m})v_{m})=h_0(x_{m},v_{m})
\end{eqnarray*}
from $h_0(x,v)=h_0(x,R(x)v)$ on $\gamma$. We thus complete the proof.
\end{proof}
\subsection{$\det\{\frac{\partial v_k}{\partial v_{\ell}}\}$ near $\partial\Omega$}
 We now compute $\det\{\frac{\partial v_k}{\partial v_{\ell}}\}$
for carefully chosen specular back-time cycle near the boundary $\partial\Omega$. Given small $\varepsilon_0$, we choose $v_1$ such that
\begin{eqnarray*}
|v_1|=\varepsilon_0,\quad v_1\cdot n(x_1)=\frac{v\cdot\nabla\xi(x_1)}{|\nabla\xi(x_1)|}=\varepsilon^2_0.
\end{eqnarray*}
We shall analyze the specular back-time cycle of $(0,x_1,v_1):(t_k,x_k, v_k)$. Lettting
\begin{eqnarray*}
\xi\left(x_1-\sum^k_{j=1}s_j v_j\right)=0, \quad v_k=R(x_k)v_{k-1},\quad x_k=x_{k-1}-s_k v_k\in\partial\Omega.
\end{eqnarray*}
\begin{proposition} \emph{\cite{GuoConti}}
For any finite $k\geq 1$,
\begin{eqnarray*}
\frac{\partial v^{i}_k}{\partial v^{\ell}_1}=\delta_{ij}+\zeta(k)n^i(x_1)n^{\ell}(x_1)+O(\varepsilon_0).
\end{eqnarray*}
where $O$ depends on $k$ and $\zeta$ is defined as $\zeta(1)=0$.
\begin{eqnarray*}
\zeta(k)=4\sum^{k-2}_{p=1}(-1)^{k-p+1}+4\sum^{k-2}_{p=1}(-1)^{k-p-1}\zeta(p)+2+3\zeta(k-1),\mbox{ for }k\geq 2.
\end{eqnarray*}
In particular, $\zeta(t)$ is an even integer so that
\begin{eqnarray*}
\det\left(\frac{\partial v^{i}_k}{\partial v^{\ell}_1}\right)=\{\zeta(k)+1\}+O(\varepsilon_0)\neq 0.
\end{eqnarray*}
\end{proposition}
%%%%%%%%%%%%%%%%%%%%%%%%%%%%%%%%%%%%%%%%%%%%%%%%%%%%%%%%%%%%%%%%%%%%%%%%%%%%%%%%%%%%%%%%%%%%%%%%%%%%%%%%%%%%%%%%%%%%%%%%
% %%%%%%%%%%%%%%%%%%%%%%%%%%%%%%%%%%%%%%%%%%%%%%%%%%%%%%%%%%%%%%%%%%%%%%%%%%%%%%%%%%%%%%%%%%%%%%%%%%%%%%%%%%%%%%%%%%%%%%
% %
% %
% %                         Section: L^infty Decay for U(t)
% %
% %
% %%%%%%%%%%%%%%%%%%%%%%%%%%%%%%%%%%%%%%%%%%%%%%%%%%%%%%%%%%%%%%%%%%%%%%%%%%%%%%%%%%%%%%%%%%%%%%%%%%%%%%%%%%%%%%%%%%%%%%
%%%%%%%%%%%%%%%%%%%%%%%%%%%%%%%%%%%%%%%%%%%%%%%%%%%%%%%%%%%%%%%%%%%%%%%%%%%%%%%%%%%%%%%%%%%%%%%%%%%%%%%%%%%%%%%%%%%%%%%%
\subsection{$L^{\infty}$ Decay for $U(t)$}
From Duhamel's principle, we have
\begin{eqnarray*}
h(t)&=&G(t)h_0+\int^t_0G(t-s)K^{x,\varpi}_{w_{\varpi}}h(s)ds_1\cr
&=&G(t)h_0+\int^t_0G(t-s_1)K^{x,\varpi}_{w_{\varpi}}G(s_1)h(s_1)ds_1\cr
&&+\int^t_0\int^{s_1}_0G(t-s_1)K^{x,\varpi}_{w_{\varpi}}G(s_1-s)K^{x,\varpi}_{w_{\varpi}}h(s)dsds_1
\end{eqnarray*}
 Therefore, $h$ can be written as follows
\begin{eqnarray*}\label{KGKG}
\begin{split}
\hspace{0.3cm}h(t,x,v)&=e^{-\mathcal{V}(t,x,v)}h_0(X_{\bf{cl}}(0),V_{\bf{cl}}(0))\cr
&+\int^t_0e^{-\mathcal{V}(t-s_1,x,v)}\int K^{\varpi, X_{\bf{cl}}(s_1)}_{w}\big(V_{\bf{cl}}(s_1),v^{\prime}\big)
e^{-\mathcal{V}(s_1,X_{\bf cl}(s_1),v^{\prime}}h_0(X^{\prime}_{\bf{cl}}(0),V^{\prime}_{\bf{cl}}(0))dv^{\prime}\cr
&+\int^t_0\int^{s_1}_0\int_{R^6}e^{-\mathcal{V}(t-s_1,x,v)
-\mathcal{V}(s_1-s,X_{\bf cl}(s_1),v^{\prime})}\cr
&\hspace{1.2cm}\times K^{\varpi,X_{\bf cl}(s_1)}_{w_{\varpi}}(V_{\bf cl}(s_1),v^{\prime})
K^{\varpi,X^{\prime}_{\bf cl}(s)}_{w_{\varpi}}(V^{\prime}_{\bf cl}(s),v^{\prime\prime})
h(s,X^{\prime}_{\bf cl}(s),v^{\prime\prime})dsds_1dv^{\prime}dv^{\prime\prime}.\cr
&\equiv\mathcal{(A)}_{\varpi}+\mathcal{B}_{\varpi}+\mathcal{C}_{\varpi}.
\end{split}
\end{eqnarray*}
where the back-time specular cycle from $(s_1,X_{\bf cl}(s_1),v^{\prime})$ is denoted by
\begin{eqnarray}
X^{\prime}_{\bf cl}(s)=X_{\bf cl}(s ; s_1, X_{\bf cl}(s_1),v^{\prime}),\quad V^{\prime}_{\bf cl}= V_{\bf cl}(s ; s_1, X_{\bf cl}(s_1),v^{\prime}).
\end{eqnarray}
More explicitly, let $t_k$ and $t^{\prime}_k$ be the corresponding times for both specular cycles as
in (\ref{SpecularBackTimeCycle}). For $t_{k+1}\leq s\leq t_k$, $t^{\prime}_{k^{\prime}+1}\leq s\leq t^{\prime}_{k^{\prime}}$
\begin{eqnarray*}
X^{\prime}_{\bf cl}(s)=X_{\bf xl}(s;s_1,X_{\bf cl}(s_1),v^{\prime})\equiv x_{k^{\prime}}+(s-t^{\prime}_{k^{\prime}})v^{\prime}_{k^{\prime}},
\end{eqnarray*}
where
\begin{eqnarray*}
x^{\prime}_{k^{\prime}}=X_{\bf cl}(t_{k^{\prime}}:s_1,x_k+(s_1-t_k)v_k,v^{\prime}),\quad v^{\prime}_{k^{\prime}}
=V_{\bf cl}(t_{k^{\prime}}:s_1,x_k+(s_1-t_k)v_k,v^{\prime}).
\end{eqnarray*}
The collision frequency $\mathcal{V}$ is defined as in the proof of Lemma \ref{continuity}:
\begin{eqnarray*}
\mathcal{V}(t, x, v)&=&\nu^{\varpi}( x, v)(t-t_1)+\nu^{\varpi}(x_1,v_1)(t_1-t_{2})+\cr
&&\cdots+\nu^{\varpi}(x_{m-1},v_{m-1})(t_{m-1}-t_m)+\nu^{\varpi}(x_m,v_m)t_m.
\end{eqnarray*}
Recall
\begin{eqnarray}\label{functiona2}
&&\alpha(s)=\xi^2(X(s))+[V(x)\cdot \nabla \xi(X(s))]^2-2\{V(s)\cdot\xi(X(s))\cdot V(s)\}\xi(X(x)).
\end{eqnarray}
We are now in a place to state the main theorem of this section.
%%%%%%%%%%%%%%%%%%%%%%%%%%%%%%%%%%%%%%%%%%%%%%%%%%%%%%%%%%%%%%%%%%%%%%%%%%%%%%%%%%%%%%%%%%%%%%%%%%%%%%%%%%%%%%%%
%
%
%                     Theorem L^{infty} Decay of U(t)
%
%
%%%%%%%%%%%%%%%%%%%%%%%%%%%%%%%%%%%%%%%%%%%%%%%%%%%%%%%%%%%%%%%%%%%%%%%%%%%%%%%%%%%%%%%%%%%%%%%%%%%%%%%%%%%%%%%%
\begin{theorem}\label{LInftyDecayTheorem}
Assume $w^{-2}_{\varpi}(1+|v-\varpi\times x|)\in L^1$. Assume that $\Omega$ satisfies the geometric assumption $\mathcal{(A)}$
and the mass, angular momentum and energy are conserved as in (\ref{ConservationLaw}).
Then there exists a unique solution $f(t,x,v)$ to
\begin{eqnarray*}\label{LinearBE2}
\begin{split}
&&\{\partial_t+v\cdot \nabla_x+\mathcal{\nu}^{\varpi}- K^{\varpi,x}\}f=0,\quad f(0,v,x)=f_0,\cr
&&\qquad f(t,x,v)|_{\gamma_+}=f(t,x,R(x)v)|_{\gamma_-}.
\end{split}
\end{eqnarray*}
and $h(t,x,v)=U(t)h_0$ to the weighted linear Boltzmann equation:
\begin{eqnarray}\label{LinearBE3_Weighted}
&&\{\partial_t+v\cdot \nabla_x+\nu
^{\varpi}-K^{\varpi,x}_{w_{\varpi}}\}\{U(t,0)h_0\}=0,\quad U(0,0)h_0=h_0,\cr
&&\qquad \{U(t)h_0\}(t,x,v)|_{\gamma_+}=\{U(t)h_0\}(t,x,R(x)v)|_{\gamma_-}.
\end{eqnarray}
Moreover, there exists $0<\lambda<\lambda_0$ such that
\begin{eqnarray*}
\sup_{t\geq 0}e^{\lambda t}\|U(t,0)h_0\|_{\infty}\leq C\|h_0\|_{\infty}
\end{eqnarray*}
\end{theorem}
For the proof of the theorem, we need the following two technical lemmas. The proof can be found in \cite{GuoConti}.
%%%%%%%%%%%%%%%%%%%%%%%%%%%%%%%%%%%%%%%%%%%%%%%%%%%%%%%%%%%%%%%%%%%%%%%%%%%%%%%%%%%%%%%%%%%%%%%%%%%%%%%%%%%%%%%%
%
%                                              Time Decay
%
%%%%%%%%%%%%%%%%%%%%%%%%%%%%%%%%%%%%%%%%%%%%%%%%%%%%%%%%%%%%%%%%%%%%%%%%%%%%%%%%%%%%%%%%%%%%%%%%%%%%%%%%%%%%%%%%
\begin{lemma}\label{finite time estimate Lemma}
Assume that there exists $\lambda>0$ so that the solution $f(t,x,v)$ of the linearized Boltzmann equation
(\ref{LinearBE}) satisfies $e^{\lambda t}\|f(t)\|\leq C\|f_0\|$.
Let $h_0=wf_0\in L^{\infty}$ and $h(t)=U(t)h_0=wf(t)$ is the solution of (\ref{LinearBE2}) where $w^{-2}\in L^1$. Assume there exist $T_0>0$
and $C_{T_0}>0$ such that
\begin{eqnarray}\label{finite time estimate}
\|U(T_0)h_0\|_{\infty}\leq e^{-\lambda T_0}\|h_0\|_{\infty}+C_{T_0}\int^{T_0}_0\|f(s)\|ds.
\end{eqnarray}
Then we have for all $t\geq 0$
\begin{eqnarray*}
\sup_{t\geq 0}e^{\lambda t}\|U(t,0)h_0\|_{\infty}\leq C\|h_0\|_{\infty}.
\end{eqnarray*}
\end{lemma}
\begin{proof}
See \cite{GuoConti}.
\end{proof}
We now define the main set:
\begin{eqnarray*}
&&A_{\alpha}(x,v)=\left\{(x,v): x\in\bar{\Omega},~ \frac{1}{2N}\leq|v|\leq 2N, \mbox{ and } \alpha(x,v)\geq \frac{1}{N} \right\}.
\end{eqnarray*}
%%%%%%%%%%%%%%%%%%%%%%%%%%%%%%%%%%%%%%%%%%%%%%%%%%%%%%%%%%%%%%%%%%%%%%%%%%%%%%%%%%%%%%%%%%%%%
%
%                         Jacobian
%
%%%%%%%%%%%%%%%%%%%%%%%%%%%%%%%%%%%%%%%%%%%%%%%%%%%%%%%%%%%%%%%%%%%%%%%%%%%%%%%%%%%%%%%%%%%%%
\begin{lemma}\label{CoveringLemma}
Fix $k$ and $k^{\prime}$. Define for $t_{k+1}\leq s_1\leq t_k$ and $s\in R$
\begin{eqnarray*}
J=J_{k,k^{\prime}}(t,x,v,s_1,s,v^{\prime})
=\det\left(\frac{\partial\{x^{\prime}_{k^{\prime}}+(s-t^{\prime}_{k^{\prime}})v^{\prime}_{k^{\prime}}\}}{\partial v^{\prime}}\right).
\end{eqnarray*}
For any $\varepsilon>0$ sufficiently small, there is $\delta(N,\varepsilon,T_0, k, k^{\prime})>0$ and an open covering $\cup^m_{i=1}B(t_i,x_i,v_i: r_i)$
of $[0,T_0]\times A_{\alpha}$ and corresponding open sets (with $x_i\in\bar{\Omega}$) $O_{t_i,x_i,v_i}$
for $[t_{k+1}+\varepsilon, t_k-\varepsilon]\times R\times \mathbb{R}^3$ with $|O_{t_i,x_i,v_i}|<\varepsilon$, such that
\begin{eqnarray*}
|J_{k,k^{\prime}}(t,x,v,s_1,s,v^{\prime})|\geq\delta>0,
\end{eqnarray*}
for $0\leq t\leq T_0$, $(x,v)\in A_{\alpha}$ and $(s_1,s,v^{\prime})$ in
\begin{eqnarray*}
O^c_{t_i,x_i,v_i}\cap[t_{k+1}+\varepsilon, t_{k}-\varepsilon]\times [0,T_0]\times \{|v^{\prime}|\leq 2N\}.
\end{eqnarray*}
\end{lemma}
\begin{proof}
See \cite{GuoConti}.
\end{proof}
%%%%%%%%%%%%%%%%%%%%%%%%%%%%%%%%%%%%%%%%%%%%%%%%%%%%%%%%%%%%%%%%%%%%%%%%%%%%%%%%%%%%%%%%%%%%%%%%%%%%%%%%%%%%%%%%%%%%%%%%%%%%%%
%
%
%                                                   Begin Proof
%
%
%%%%%%%%%%%%%%%%%%%%%%%%%%%%%%%%%%%%%%%%%%%%%%%%%%%%%%%%%%%%%%%%%%%%%%%%%%%%%%%%%%%%%%%%%%%%%%%%%%%%%%%%%%%%%%%%%%%%%%%%%%%%%%%
\begin{proof}[{\bf Proof of Theorem \ref{LInftyDecayTheorem}:}]
%%%%%%%%%%%%%%%%%%%%%%%%%%%%%%%%%%CASE I%%%%%%%%%%%%%%%%%%%%%%%%%%%%%%%%%%%%%%%%%%%%%%%%%%%%%%%%%%%%%%%%%%%%%%%%%%%%%%%%%%
We first notice that there exists a constant $\nu_0$ such that $\nu_0\leq \nu(x_i,v_i)$ and
\begin{eqnarray*}
-\mathcal{V}(t,x,v)&\leq&-\nu_0(t-t_1)-\nu_0(t_1-t_{2})-\cdots-\nu_0(t_{m-1}-t_m)-\nu_0t_m\cr
&=&-\nu_0t.
\end{eqnarray*}
From this, the estimates for $\mathcal{(A)}_{\varpi}$:
\begin{eqnarray*}
\mathcal{(A)}_{\varpi}\leq e^{-\nu_0t}\|h_0\|_{\infty}.
\end{eqnarray*}
For the second term $\mathcal{B}_{\varpi}$ in (\ref{KGKG}), we note that $\|K^{\varpi,x}_{w}h\|_{\infty}\leq C\|h\|_{\infty}$.
Then by Lemma \ref{K_wEstimate},
\begin{eqnarray*}
\mathcal{B}_{\varpi}&=&\left\|\int^t_0G(t-s_1)K^{\varpi,x}_{\varpi}G(s_1)h_0ds_1\right\|_{\infty}\cr
&\leq& \int^t_0e^{-\nu_0(t-s_1)}\|K^{\varpi,x}_{\varpi}G(s_1)h_0\|_{\infty}ds_1\cr
&\leq& Cte^{-\nu_0t}\|h_0\|_{\infty}.
\end{eqnarray*}
We now concentrate on $\mathcal{C}_{\varpi}{\bf 1}_{A_{\alpha}}$. For this, we divide $A_{\alpha}$ into the following two cases.
We suppose $N$ is large enough such that
\begin{eqnarray}\label{Ncondition}
N>|\varpi|\max_{x\in\Omega}|x|
\end{eqnarray}
\newline
\noindent{\bf Case 1:} For $\{|v-\varpi\times x|\leq 3N\}\cap\{|v^{\prime}-\varpi\times x|\geq 6N\}\cap A_{\alpha}$
or $\{|v^{\prime}-\varpi\times x|\leq 6N\}\cap \{|v^{\prime\prime}-\varpi\times x|\geq 9N\}\cap A_{\alpha}$.\newline
For the first case, we recall $|v|=|V_{\bf cl}(s_1)|$ to see
\begin{eqnarray*}
|V_{\bf cl}(s_1)-v^{\prime}|&\geq&|v^{\prime}|-|V_{\bf cl}(s_1)|\cr
&=&|v^{\prime}|-|v|\cr
&\geq&|v^{\prime}-\varpi\times x|-|v-\varpi\times x|-2|\varpi|\max_{x\in \Omega}|x|\cr
&\geq&6N-3N-2N= N.
\end{eqnarray*}
In a similar way, we have for the second case
\begin{eqnarray*}
|V^{\prime}_{\bf cl}(s)-v^{\prime\prime}|\geq N.
\end{eqnarray*}
Therefore either one of the following are valid correspondingly:
\begin{align}
\begin{aligned}
|K^{\varpi, X_{\bf cl}(s_1)}_{w_{\varpi}}\big(V_{\bf cl}(s_1),~v^{\prime}\big)|&\leq
e^{-\frac{\varepsilon}{8}N^2}|K^{\varpi, X_{\bf cl}(s_1)}_{w_{\varpi}}(V_{\bf cl}(s_1),~v^{\prime}\big)|
e^{\frac{\varepsilon}{8}|V_{\bf cl}(s_1)-v^{\prime}|^2},\cr
|K^{\varpi,X^{\prime}_{\bf cl}(s)}_{w_{\varpi}}\big(V^{\prime}_{\bf cl}(s),~v^{\prime\prime}\big)|
&\leq e^{-\frac{\varepsilon}{8}N^2}|K^{\varpi, X^{\prime}_{\bf cl}(s)}_{w_{\varpi}}(V^{\prime}_{\bf cl}(s), ~v^{\prime\prime})|e^{\frac{\varepsilon}{8}
|V^{\prime}_{\bf cl}(s)-v^{\prime\prime}|^2}.
\end{aligned}
\end{align}
From Lemma \ref{K_wEstimate},
\begin{eqnarray*}
&&\int|K^{\varpi\times X_{\bf cl}(s_1)}_{w_{\varpi}}\left(V_{\bf cl}(s_1),v^{\prime}\right)|
e^{\frac{\varepsilon}{8}|V_{\bf cl}(s_1)-v^{\prime}|^2}dv^{\prime}<\infty,\cr
&&\int |K^{\varpi, X^{\prime}_{\bf cl}(s)}_{w_{\varpi}}\left(V^{\prime}_{\bf cl}(s), v^{\prime\prime}\right)|e^{\frac{\varepsilon}{8}
|V^{\prime}_{\bf cl}(s)-v^{\prime\prime}|^2}dv^{\prime\prime}<\infty.
\end{eqnarray*}
Using this, we split the estimate as follows
\begin{eqnarray*}
&&\hspace{-0.6cm}\int^t_0\int^{s_1}_0\left\{\int_{\substack{|v-\varpi\times x|\leq 3N,\cr |v^{\prime}-\varpi\times x|\geq 6N}}
+\int_{\substack{|v^{\prime}-\varpi\times x|\leq 6N,\cr |v^{\prime\prime}-\varpi\times x|\geq 9N}}\right\}\cr
&&\leq C_K\int^t_0\int^{s_1}_0e^{-\nu_0(t-s)}\|U(s,0)h_0\|_{\infty}
\left\{\sup_{v}\int%_{\substack{|v|\leq 2N,\cr |v^{\prime}|\geq 3N}}
K^{\varpi, X_{\bf cl}(s_1)}_{w_{\varpi}}\big(V_{\bf cl}(s_1),~v^{\prime}\big)dv^{\prime}\right.\cr
&&\qquad\left.+\sup_{v^{\prime}}\int%_{\substack{|v^{\prime}|\leq 3N,\cr |v^{\prime\prime}|\geq 4N}}
K^{\varpi, X^{\prime}_{\bf cl}(s)}_{w_{\varpi}}\big(V^{\prime}_{\bf cl}(s),~v^{\prime\prime}\big)dv^{\prime\prime}\right\}\cr
&&\leq C_{\varepsilon,K}e^{-\frac{\varepsilon}{8}N^2}\int^t_0\int^{s_1}_0e^{-\nu_0(t-s)}\|U(s,0)h_0\|_{\infty}dsds_1\cr
&&\leq C_{\varepsilon,K}e^{-\frac{\varepsilon}{8}N^2}\sup_{s\geq 0}\{e^{\frac{\nu_0}{2}s}\|U(s,0)h_0\|_{\infty}\}.
\end{eqnarray*}

%%%%%%%%%%%%%%%%%%%%%%%%%%%%%%%%%CASE III%%%%%%%%%%%%%%%%%%%%%%%%%%%%%%%%%%%%%%%
\noindent{\bf Case 2:} $\{|v^{\prime}-\varpi\times x|\leq 6N\}\cap\{|v^{\prime\prime}-\varpi\times x|\leq 9N\}\cap A_{\alpha}$.\newline
We have from $|V_{\bf cl}(s_1)|=|v|$
\begin{eqnarray*}
|V_{\bf cl}(s_1)-\varpi\times X_{\bf cl}(s_1)|&\leq&|V_{\bf cl}(s_1)|+|\varpi||X_{\bf cl}(s_1)|\cr
&\leq&|v|+|\varpi|\max_{x\in \Omega}|x|\cr
&\leq&|v-\varpi\times x|+2|\varpi|\max_{x\in \Omega}|x|\cr
&\leq&3N+2N=5N.
\end{eqnarray*}
Likewise,
\begin{eqnarray*}
|V^{\prime}_{\bf cl}(s)-\varpi\times X^{\prime}_{\bf cl}(s_1)|\leq 8N.
\end{eqnarray*}
We can choose $K^{\varpi,x}_{wN}(v,v^{\prime})$ smooth with compact support such that
\begin{eqnarray}\label{125}
\begin{split}
\sup_{\substack{x\in \Omega,\cr|p-\varpi\times x|\leq 3N}}\int_{|v^{\prime}-\varpi\times x|\leq 5N}|K^{\varpi,x}_{w_{\varpi}}(p,v^{\prime})-K^{\varpi,x}_{N}(p,v^{\prime})|dv^{\prime}\leq\frac{1}{N},\cr
\sup_{\substack{x\in \Omega,\cr|p-\varpi\times x|\leq 6N}}\int_{|v^{\prime}-\varpi\times x|\leq 8N}|K^{\varpi,x}_{w_{\varpi}}(p,v^{\prime})-K^{\varpi,x}_{N}(p,v^{\prime})|dv^{\prime}\leq\frac{1}{N}.
\end{split}
\end{eqnarray}
We use this to obtain the following splitting:
\begin{eqnarray*}
&&K^{\varpi, X_{\bf cl}(s_1)}_{w_{\varpi}}\big(V_{\bf cl}(s_1), v^{\prime}\big)
K^{\varpi, X^{\prime}_{\bf cl}(s)}_{w_{\varpi}}\big(V_{\bf cl}^{\prime}(s), v^{\prime\prime}\big)\cr
&&\quad=\Big\{K^{\varpi, X_{\bf cl}(s_1)}_{w_{\varpi}}\big(V_{\bf cl}(s_1),v^{\prime}\big)
-K^{\varpi, X_{\bf cl}(s_1)}_{N}\big(V_{\bf cl}(s_1),v^{\prime}\big)\Big\}
 K^{\varpi, X^{\prime}_{\bf cl}(s)}_w(V_{\bf cl}^{\prime}(s),v^{\prime\prime}\big)\cr
&&\quad+ K^{\varpi, X_{\bf cl}(s_1)}_{w_{\varpi}}(V_{\bf cl}(s_1),v^{\prime}\big)
\Big\{K^{\varpi, X^{\prime}_{\bf cl}(s)}_{N}\big(V_{\bf cl}(s),v^{\prime\prime}\big)
-K^{\varpi, X^{\prime}_{\bf cl}(s)}_{N}\big(V^{\prime}_{\bf cl}(s),v^{\prime\prime}\big)\Big\}\cr
&&\quad+K^{\varpi, X_{\bf cl}(s_1)}_{N}\big(V_{\bf cl}(s_1),v^{\prime}\big)
K^{\varpi, X^{\prime}_{\bf cl}(s)}_{N}\big(V_{\bf cl}^{\prime}(s),v^{\prime\prime}\big).
\end{eqnarray*}
We can use the approximation (\ref{125}) to estimate as
\begin{eqnarray*}
\mathcal{C}_{\varpi}{\bf1}_{A_{\alpha}}
%\int^t_0\int^{s_1}_0\int_{\substack{|v^{\prime}-\varpi\times x|\leq 6N,\cr|v^{\prime\prime}-\varpi\times x|\leq 9N}}
%e^{-\nu_0(t-s)}{\bf1}_{A_{\alpha}}|h(s,X^{\prime_{\bf cl}},v^{\prime\prime})|dv^{\prime}dv^{\prime\prime}ds_1ds\cr
&=&\frac{C e^{-\frac{\nu_0t}{2}}}{N}\sup_s\Big\{e^{\frac{\nu_0s}{2}}\|U(s,0)h_0\|_{\infty}\Big\}\cr
&&\times\left\{\sup_{v}\int_{|v^{\prime}-X_{\bf cl}(s_1)|\leq 5N}
K^{\varpi, X_{\bf cl}(s_1)}_{w_{\varpi}}\big(V_{\bf cl}(s_1),~v^{\prime}\big)dv^{\prime}\right.\cr
&&\qquad\left.+\sup_{v^{\prime}}\int_{|v^{\prime\prime}-\varpi\times X^{\prime}_{\bf cl}(s)|\leq 8N}
K^{\varpi, X^{\prime}_{\bf cl}(s)}_{w_{\varpi}}\big(V^{\prime}_{\bf cl}(s),~v^{\prime\prime}\big)dv^{\prime\prime}\right\}\cr
&&+\int^t_0\int^{s_1}_0\int_{\substack{|v^{\prime}-\varpi\times x|\leq 6N,\cr
|v^{\prime\prime}-\varpi\times x|\leq 9N}}e^{-\nu_0(t-s)}
K^{\varpi, X_{\bf cl}(s_1)}_{N}\big(V_{\bf cl}(s_1),v^{\prime}\big)
K^{\varpi, X^{\prime}_{\bf cl}(s)}_{N}\big(V_{\bf cl}^{\prime}(s),v^{\prime\prime}\big)\cr
&&\quad\times|h(s,X^{\prime}_{\bf{cl}}(s),v^{\prime\prime})|dsds_1dv^{\prime}dv^{\prime\prime}\cr
&\leq&\frac{C e^{-\frac{\nu_0t}{2}}}{N}\sup_s\Big\{e^{\frac{\nu_0s}{2}}\|U(s,0)h_0\|_{\infty}\Big\}\cr
&&+\int^t_0\int^{s_1}_0\int_{\substack{|v^{\prime}-\varpi\times x|\leq 6N,\cr |v^{\prime\prime}-\varpi\times x|\leq 9N}}
e^{-\nu_0(t-s)}{\bf1}_{A_{\alpha}}|h(s,X^{\prime}_{\bf cl}(s),v^{\prime\prime})|
dv^{\prime}dv^{\prime\prime}ds_1ds,
\end{eqnarray*}
where we used the boundedness of
\begin{eqnarray*}
K^{\varpi, X_{\bf cl}(s_1)}_{N}\big(V_{\bf cl}(s_1),v^{\prime}\big)
K^{\varpi, X^{\prime}_{\bf cl}(s)}_{N}\big(V_{\bf cl}^{\prime}(s),v^{\prime\prime}\big).
\end{eqnarray*}
We further divide the estimate as
\begin{eqnarray*}
&&\int^t_0\int^{s_1}_0\int_{\substack{|v^{\prime}-\varpi\times x|\leq 6N,\cr |v^{\prime\prime}-\varpi\times x|\leq 9N}}e^{-\nu_0(t-s)}{\bf1}_{A_{\alpha}}|h(s,X^{\prime}_{\bf cl}(s),v^{\prime\prime})|
dv^{\prime}dv^{\prime\prime}ds_1ds\cr
&&\hspace{1cm}\leq\int_{\substack{\alpha(X_{\bf cl}(s_1),v^{\prime})<\varepsilon\cr|v^{\prime}-\varpi\times x|\leq 6N, |v^{\prime\prime}-\varpi\times x|\leq 9N}}
+\int_{\substack{\alpha(X_{\bf cl}(s_1),v^{\prime})\geq \varepsilon \cr|v^{\prime}-\varpi\times x|\leq 6N, |v^{\prime\prime}-\varpi\times x|\leq 9N}}\cr
&&\hspace{1cm}= I_1+I_2.
\end{eqnarray*}
In the case $\alpha(X_{\bf cl}(s_1),v^{\prime})\leq \varepsilon$,~
$\xi^2(X_{\bf cl}(s_1))+[v^{\prime}\cdot\nabla\xi(X_{\bf cl}(s_1),v^{\prime})]\leq \varepsilon$. Hence for small
$\varepsilon$, $X_{\bf cl}(s_1)\sim\partial\Omega$ and $|\nabla \xi( X_{\bf cl}(s_1))|\geq \frac{1}{2}$. The first term $I_1$ is bounded by
\begin{eqnarray*}
\begin{split}
I_1&=C_N\int^t_0\int^{s_1}_0e^{-\nu_0(t-s)}\|h(s)\|_{\infty}dsds_1
\int_{\substack{\alpha(X_{\bf cl}(s_1),v^{\prime})<\varepsilon\cr|v^{\prime}-\varpi\times x|\leq 6N, |v^{\prime\prime}-\varpi\times x|\leq 9N}}\cr
&\leq C_N\sup_{t\geq s}e^{-\frac{\nu_0}{2}(t-s)}\|h(s)\|_{\infty}
\int_{\substack{\left|v^{\prime}\cdot\frac{\nabla\xi(X_{\bf cl}(s_1),v^{\prime})}{|\nabla\xi(X_{\bf cl})(s_1)|}\right|\leq \varepsilon\cr
|v^{\prime}-\varpi\times x|\leq 6N, |v^{\prime\prime}-\varpi\times x|\leq 9N}}\cr
&\leq C_N\sup_{t\geq s}e^{-\frac{\nu_0}{2}(t-s)}\|h(s)\|_{\infty}
\int_{\substack{\left|(v^{\prime}-\omega\times X_{\bf cl}(s_1))\cdot\frac{\nabla\xi(X_{\bf cl}(s_1),v^{\prime})}{|\nabla\xi(X_{\bf cl})(s_1)|}\right|\leq \varepsilon\cr
|v^{\prime}-\varpi\times x|\leq 6N, |v^{\prime\prime}-\varpi\times x|\leq 9N}}\cr
&\leq C_N\varepsilon\sup_{t\geq s}e^{-\frac{\nu_0}{2}(t-s)}\|h(s)\|_{\infty},
\end{split}
\end{eqnarray*}
where we used the rotational symmetry:
\[
\omega\times X_{\bf cl}(s_1)\cdot\frac{\nabla\xi(X_{\bf cl}(s_1),v^{\prime})}{|\nabla\xi(X_{\bf cl})(s_1)|}=0.
\]
On the other hand, to estimate $I_2$, we first note that
%\begin{eqnarray}\label{Abar}
%A_{\alpha}(x,v)\in\bar{A}_{\alpha}(x,v)\equiv\left\{(x,v): x\in\bar{\Omega},~ \frac{1}{N}\leq|v|\leq 4N, \mbox{ and } \alpha(x,v)\geq \frac{1}{N} \right\}.
%\end{eqnarray}
%and
\begin{eqnarray}\label{vvprime}
\begin{split}
&|v^{\prime}|\leq |v^{\prime}-\varpi\times x|+|\varpi|\max_{x\in\Omega} |x|\leq 7N,\cr
&|v^{\prime\prime}|\leq |v^{\prime\prime}-\varpi\times x|+|\varpi|\max_{x\in\Omega} |x|\leq 10N.
\end{split}
\end{eqnarray}
From (\ref{vvprime}), we can write $I_2$ as
\begin{eqnarray*}
\begin{split}
I_2&\leq C_N\int^t_0\int^{s_1}_0
\int_{\substack{\alpha(X_{\bf cl}(s_1),v^{\prime})\geq\varepsilon\cr|v^{\prime}|\leq 7N, |v^{\prime\prime}|\leq 10N}}
e^{-\nu_0(t-s)}{\bf1}_{A_{\alpha}}|h(s,X^{\prime}_{\bf cl}(s),v^{\prime\prime})|dv^{\prime}dv^{\prime\prime}ds_1ds,\cr
&=C_N\sum_{k,k^{\prime}}\int^{t_k(v)}_{t_{k+1}(v)}\int^{t^{\prime}(v^{\prime})}_{t^{\prime}_{k^{\prime}+1}(v^{\prime})}
\int_{\substack{\alpha(X_{\bf cl}(s_1),v^{\prime})\geq\varepsilon\cr|v^{\prime}|\leq 7N, |v^{\prime\prime}|\leq 10N}}
{\bf1}_{A_{\alpha}}e^{-\nu_0(t-s)}
|h(s,x^{\prime}_{k^{\prime}}+(s-t^{\prime}_{k^{\prime}}))v^{\prime}_{k^{\prime}},v^{\prime\prime})|.
\end{split}
\end{eqnarray*}
Therefore, in a small velocity regime, the rotation does not have big effect and we can consider this problem as in \cite{GuoConti}.
We now study $x^{\prime}_{k^{\prime}}+(s-t^{\prime}_{k^{\prime}})v^{\prime}_{k^{\prime}}$. By repeatedly using Lemma \ref{Velocity Lemma},
we deduce for $(t,s,v)\in \Bar{A}_{\alpha}$ and $\alpha(X(s_1),v^{\prime})\geq \varepsilon, |v^{\prime}|\leq 7N$:
\begin{eqnarray*}
\alpha(t_{\ell})&=&\{v_{\ell}\cdot n_{x_{\ell}}\}\geq e^{-\{C_{\xi}N-1\}T_0}\alpha(s_1)\geq C_{T_0,\xi,N}>0;\cr
\alpha(t^{\prime}_{\ell})&=&\{v^{\prime}_{\ell}\cdot n_{x^{\prime}_{\ell}}\}\geq e^{-\{C_{\xi}N-1\}T_0}\alpha(X_{\bf cl}(s_1),v^{\prime})\geq C_{T_0,\xi,N,\varepsilon}>0.
\end{eqnarray*}
Therefore, applying Lemma \ref{1234} (3) yields
\begin{eqnarray*}
t_{\ell}-t_{\ell+1}\geq\frac{C_{T_0,\varepsilon, N}}{N^2},\quad t^{\prime}_{\ell}-t^{\prime}_{\ell+1}\geq\frac{C_{T_0,\varepsilon, N, {\varepsilon}}}{4N^2},
\end{eqnarray*}
So that
\begin{eqnarray*}
k\leq \frac{T_0N^2}{C_{T_0,\varepsilon, N}}= C_{T_0,\varepsilon, N},\quad
k^{\prime}\leq \frac{T_0N^2}{C_{T_0,\varepsilon, N,\varepsilon}}= C_{T_0,\xi, N,\varepsilon}.
\end{eqnarray*}
We therefore further split the $s_1$-ingegral as
\begin{eqnarray*}
&&C_{k,N}\int^{t_k}_{t_{k+1}}
\int_{\substack{|v^{\prime}|\leq 7N,\cr |v^{\prime\prime}|\leq 10N}}
\sum_{\substack{k\leq C_{T_0,\varepsilon, N},\cr k^{\prime}\leq C_{T_0,\varepsilon, N}}}
\int^{t^{\prime}}_{t^{\prime}_{k^{\prime}+1}}
{\bf1}_{A_{\alpha}}e^{-\nu_0(t-s)}|h(s,x^{\prime}_{k^{\prime}}+(s-t^{\prime}_{k^{\prime}}))v^{\prime}_{k^{\prime}},v^{\prime\prime})|\cr
&&=\int^{t_k-\varepsilon}_{t_{k+1}+\varepsilon}+\int^{t_k-\varepsilon}_{t_{k}}+\int^{t_{k+1}+\varepsilon}_{t_{k+1}}.
\end{eqnarray*}
Since $\sum_{k^{\prime}}\int^{t^{\prime}_{\prime}}_{t^{\prime}_{k^{\prime}+1}}=\int^{s_1}_0$, the last two terms above make small contribution as
\begin{eqnarray*}
\varepsilon C_{K,N}\sup_{0\leq s\leq t}e^{-\nu_0(t-s)}\|h(s)\|_{\infty}\int^{T_0}_0\int_{\substack{|v^{\prime}|\leq 7N,\cr|v^{\prime\prime}|\leq 10N}}
=\varepsilon C_{K,N}\sup_{0\leq s\leq t }e^{-\nu_0(t-s)}\|h(s)\|_{\infty}.
\end{eqnarray*}
For the main contribution $\int^{t_k-\varepsilon}_{t_{k+1}+\varepsilon}$, we fixed both $k$ and $k^{\prime}$. By lemma \ref{CoveringLemma}, on the set
$O^c_{t_i,x_i,v_i}\cap[t_{k+1}+\varepsilon, t_k-\varepsilon]\times[0,T_0]\times\{|v^{\prime}|\leq N\}$, we can define a change of variable
\begin{eqnarray*}
y_{k^{\prime}}=x^{\prime}_{k^{\prime}}+(s-t^{\prime}_{k^{\prime}})v^{\prime}_{k^{\prime}},
\end{eqnarray*}
so that
\begin{eqnarray*}
\det\left(\frac{\partial y_{k^{\prime}}}{\partial v^{\prime}}\right)>\delta_{N,T_0,\varepsilon, k, k^{\prime}}.
\end{eqnarray*}
On $O^c_{t_i,x_i,v_i}\cap[t_{k+1}+\varepsilon, t_k-\varepsilon]\times[0,T_0]\times\{|v^{\prime}|\leq N\}$.
By the Implicit Function Theorem, there are a finite open covering $\cup^m_{j=1}V_j$ of
$O^c_{t_i,x_i,v_i}\cap[t_{k+1}+\varepsilon, t_k-\varepsilon]\times[0,T_0]\times\{|v^{\prime}|\leq N\}$, and smooth function
$F_j$ such that $v^{\prime}=F_j(t,x,v,y,s_1,s)$ in $V_j$. We therefore have
\begin{eqnarray*}
&&\hspace{-0.3cm}\sum_{\substack{k\leq C_{T_0,\varepsilon, N},\cr k^{\prime}\leq C_{T_0,\varepsilon, N}}}\int^{t}_{t_{k+1}}
\int_{|v^{\prime}|\leq 3N}
\int^{t^{\prime}}_{t^{\prime}_{k^{\prime}+1}}\cr
&&\hspace{-0.2cm}\quad\leq
\sum_{\substack{k\leq C_{T_0,\varepsilon, N},\cr k^{\prime}\leq C_{T_0,\varepsilon, N}}}\int^{t}_{t_{k+1}}
\int_{|v^{\prime}|\leq 3N}
\int^{t^{\prime}}_{t^{\prime}_{k^{\prime}+1}}{\bf1}_{O_{t_i,x_i,v_i}}
+\sum_{\substack{j, k\leq C_{T_0,\varepsilon, N},\cr k^{\prime}\leq C_{T_0,\varepsilon, N}}}\int^{t}_{t_{k+1}}
\int_{|v^{\prime}|\leq 3N}
\int^{t^{\prime}}_{t^{\prime}_{k^{\prime}+1}}{\bf 1}_{V_{j}}\cr
&&\hspace{-0.2cm}\quad=I+II.
\end{eqnarray*}
Since $\sum_{k^{\prime}}\int^{t^{\prime}_{k^{\prime}}}_{t^{\prime}_{k^{\prime}+1}}=\int^{s_1}_0\leq \int^{T_0}_0$ and
$|O_{t_i,x_i,v_i}|<\varepsilon$, we have from Lemma \ref{CoveringLemma}
\[
I\leq C_{N^{\varepsilon}}e^{-\frac{\nu_0}{2}t}\sup_s\{e^{\frac{\nu_0}{2}s}\|h(s)\|_{\infty}\}.
\]
To estimate $II$, we change variable $v^{\prime}\rightarrow y_{k^{\prime}}=x^{\prime}_{k^{\prime}}+(s-t^{\prime}_{k^{\prime}})v^{\prime}_{k^{\prime}}$
on each $V_j$ to get
\begin{eqnarray*}
&&\hspace{-0.5cm}C_{\varepsilon, N}\int_{V_j}
\int_{|v^{\prime\prime}|\leq 4N}\int^{t^{\prime}_k(v^{\prime})}_{t^{\prime}_{k^{\prime}+1}}
 ~e^{-\nu_0(t-s)}|h(s,x^{\prime}_{k^{\prime}}+(s-t^{\prime}_{k^{\prime}}))v^{\prime}_{k^{\prime}},v^{\prime\prime})|
dsdv^{\prime}dv^{\prime\prime}\cr
&&\hspace{-0.5cm}\leq C_{\varepsilon, T_0,  N}\int_{V_j}
\int_{|v^{\prime\prime}|\leq 4N}
\int^{s_1}_0{\bf1}_{x^{\prime}_{k^{\prime}}+(s-t^{\prime}_{k^{\prime}})v^{\prime}_{k^{\prime}}\in\Omega}~ e^{-\nu_0(t-s)}
|h(s,x^{\prime}_{k^{\prime}}+(s-t^{\prime}_{k^{\prime}}))v^{\prime}_{k^{\prime}},v^{\prime\prime})|
dsds_1dv^{\prime}dv^{\prime\prime}\cr
&&\hspace{-0.5cm}\leq C_{\varepsilon,  T_0, N}\int^{s_1}_0\int_{V_j}
\int_{|v^{\prime\prime}|\leq 4N}{\bf1}_{y\in \Omega}~e^{-\nu_0(t-s)}
|h(s,y_{k^{\prime}},v^{\prime\prime})|
\frac{1}{\left|~\det\left\{\frac{\partial y_{k^{\prime}}}{\partial v^{\prime}}\right\}\right|}
dy_{k^{\prime}}dv^{\prime\prime}dsds_1\cr
&&\hspace{-0.5cm}\leq\frac{C_{\varepsilon,  T_0, N}}{\delta}
\int^t_0\int^{s_1}_0e^{-\nu_0t}\int_{|v^{\prime\prime}|\leq 4N}{\bf1}_{y\in \Omega}~e^{\nu_0s}
\left\{\int_{\Omega}h^2(s,y,v^{\prime\prime})|dy\right\}^{\frac{1}{2}}dv^{\prime\prime}dsds_1\cr
&&\hspace{-0.5cm}\leq C_{\varepsilon, T_0, N, k,k^{\prime}}\int^t_0\|f(s)\|ds,
\end{eqnarray*}
where $f=\frac{h}{w_{\varpi}}$. We therefore conclude, by summing over $j,k$ and $k^{\prime}$, and collecting terms
\begin{eqnarray}\label{Step1FinalEstimate}
\begin{split}
&\|h(t,x,v){\bf 1}_{A_{\alpha}}\|_{\infty}\leq \{1+C_Kt\}e^{-\nu_0t}\|h_0\|_{\infty}\cr
&\qquad+\left\{\frac{C}{N}+C_{N,T_0^{\varepsilon}}\right\}
\sup_{s}e^{\frac{\nu_0}{2}(t-s)}\|h(s)\|_{\infty}+C_{\varepsilon, N, T_0}\int^t_0\|f(s)\|ds.
\end{split}
\end{eqnarray}
%\begin{eqnarray}&&{Ce^{-\nu_0t}\|h_0\|_{\infty}}_{A}+{Cte^{-\nu_0t}\|h_0\|_{\infty}}_{B}
%+{C_{\varepsilon,K}e^{-\frac{\varepsilon}{8}N^2}\sup_{s\geq 0}\{e^{\frac{\nu_0}{2}s}\|U(s,0)h_0\|_{\infty}\}}_{Case I}.\cr
%&&\frac{C e^{-\frac{\nu_0t}{2}}}{N}\sup_s\Big\{e^{-\frac{\nu_0s}{2}}\|U(s,0)h_0\|_{\infty}\Big\}\cr
%&& {C_N\varepsilon\sup_{t\geq s}e^{-\frac{\nu_0}{2}(t-s)}\|h(s)\|_{\infty}}_{I_1}\cr
%&&{\varepsilon C_{K,N}\sup_{0\leq s\leq t }e^{-\nu_0(t-s)}\|h(s)\|_{\infty}}_{\int^{t-\varepsilon}_0}\cr
%&&{C_{N^{\varepsilon}}e^{-\frac{\nu_0}{2}t}\sup_s\{e^{\frac{\nu_0}{2}s}\|h(s)\|_{\infty}\}}_{Q}\cr
%&&C_{\varepsilon, T_0, N, k,k^{\prime}}\int^t_0\|f(s)\|ds_{\mbox{core}}\end{eqnarray}
%%%%%%%%%%%%%%%%%%%%%%%%%%%%%%%%%%%%%%%%%%%%%%%%%%Step I %%%%%%%%%%%%%%%%%%%%%%%%%%%%%%%%%%%%%%%%%%%%%%%%%%%%%%%%%%%%%%%%%%
\noindent{\bf  Estimate of $h(t,s,v)$:}
We now further plug (\ref{Step1FinalEstimate}) back in $h(t,x,v)=G(t,s)h_0+\int^t_0G(t-s_1)K^{\varpi,x}_{w_{\varpi}}h(s_1)ds_1$ to get
\begin{eqnarray}\label{||h||}
\|h(t)\|_{\infty}\leq e^{-\nu_0t}\|h_0\|_{\infty}+\int^t_0e^{-\nu_0\{t-s\}}\|K^{\varpi,x}_wh\|_{\infty}(s_1)ds.
\end{eqnarray}
But $\{K^{\varpi,x}_w(v,v^{\prime})h\}(s_1,x,v)=\int K^{\varpi,x}_w(v,v^{\prime})(v,v^{\prime})h(v,v^{\prime})dv^{\prime}$ and we
split it as
\begin{eqnarray*}
&&\hspace{-0.4cm}\int K^{\varpi,x}_w(v,v^{\prime})h(s_1,x,v^{\prime})\{1-{\bf1}_{A_{\alpha}(x,v^{\prime})}\}dv^{\prime}
+\int K^{\varpi,x}_w(v,v^{\prime})h(s_1,x,v^{\prime}){\bf1}_{A_{\alpha}(x,v^{\prime})}dv^{\prime}.
\end{eqnarray*}
The first term in bounded by
\begin{eqnarray*}
&&\left(\int_{\substack{|v^{\prime}-\varpi\times x|\geq N\cr or~|v^{\prime}|\leq\frac{1}{N}}}|K^{\varpi,x}_{w_{\varpi}}(v,v^{\prime})|dv^{\prime}
+\int_{\alpha(x,v^{\prime})\leq\frac{1}{N}}|K^{\varpi,x}_{w_{\varpi}}(v,v^{\prime})|dv^{\prime}\right)\|h(s_1)\|_{\infty}.
\end{eqnarray*}
The first term is obviously $o(1). $From $\alpha(x,v^{\prime})\leq \frac{1}{N}$, we have $\xi^2(x)+[v^{\prime}\cdot \nabla\xi(x)]^2\leq \frac{1}{N}$.
Therefore, for $N$ large, we have $~x\sim\partial\Omega$. On the other hand, we have $|\nabla\xi(x)|\geq\frac{1}{2}$ and the rotational symmetry assumption
implies $\varpi\times x\cdot \frac{\nabla n(x)}{|\nabla n(x)|}=0$. Therefore we have
\begin{eqnarray*}
\begin{split}
\int_{\alpha(x,v^{\prime})\leq\frac{1}{N}}|K^{\varpi,x}_w(v,v^{\prime})|dv^{\prime}
&\leq\int_{\big|v^{\prime}\cdot\frac{\nabla\xi(x)}{|\nabla\xi(x)|}\big|\leq\frac{2}{\sqrt{N}}}|K^{\varpi,x}_w(v,v^{\prime})|dv^{\prime}\cr
&=\int_{\big|(v^{\prime}-\varpi\times x)\cdot\frac{\nabla\xi(x)}{|\nabla\xi(x)|}\big|\leq\frac{2}{\sqrt{N}}}|K^{\varpi,x}_w(v,v^{\prime})|dv^{\prime}\cr
&=o(1)
\end{split}
\end{eqnarray*}
as $N\rightarrow\infty$. We apply (\ref{Step1FinalEstimate}) to the second term to bound $\|K_w^{\varpi,x}(s_1)\|_{\infty}$ as
\begin{eqnarray*}
\{1+C_Kt\}e^{-\nu_0t}\|h_0\|_{\infty}+\left\{o(1)+\frac{C}{N}+C_{N,T_0^{\varepsilon}}\right\}
\sup_{s}e^{\frac{\nu_0}{2}(t-s)}\|h(s)\|_{\infty}+C_{\varepsilon, N, T_0}\int^t_0\|f(s)\|ds.
\end{eqnarray*}
Hence, by (\ref{||h||}), $\|h(t)\|_{\infty}$ is bounded by
\begin{eqnarray*}
\begin{split}
&e^{-\nu_0t}\|h_0\|_{\infty}+\int^t_0e^{-\nu_0}\{1+C_Ks_1\}\|h_0\|_{\infty}\cr
&\quad+\int^t_0\left\{\big\{o(1)+\frac{C}{N}+C_{N,T_0,{\varepsilon}}\big\}
\sup_{s}e^{\frac{\nu_0}{2}(t-s)}\|h(s)\|_{\infty}+C_{\varepsilon, N, T_0}\int^t_0\|f(s)\|ds\right\}ds_1\cr
&\quad\leq\{1+C_Kt^2\}e^{-\nu_0t}\|h_0\|_{\infty}+C\left\{o(1)+\frac{C}{N}+C_{N,T_0,{\varepsilon}}\right\}
\sup_{s\leq t}\Big\{e^{\frac{\nu_0}{2}(t-s)}\|h(s)\|_{\infty}\Big\}\cr
&\quad+C_{\varepsilon, N, T_0}\int^t_0\|f(s)\|ds.
\end{split}
\end{eqnarray*}
We choose $T_0$ large such that $2\{1+C_KT^2_0\}e^{-\frac{\nu_0}{4}}T_0=e^{-\lambda T_0}$, for some $\lambda>0$.
We then further choose $N$ large, and then $\varepsilon$ sufficiently small such that
$\left\{o(1)+\frac{C}{N}+C_{N,T_0,\varepsilon}\right\}<\frac{1}{2}$. We then have
\begin{eqnarray*}
\sup_{0\leq s \leq t}\Big\{e^{\frac{\nu_0}{4}s}\|h(s)\|_{\infty}\Big\}\leq 2\{1+C_Kt^2\}\|h_0\|_{\infty}+C_{T_0}\int^t_0\|f(s)\|ds.
\end{eqnarray*}
Choosing $s=t=T_0$, we deduce the finite-time estimate (\ref{finite time estimate}), and our theorem follows from Lemma \ref{finite time estimate Lemma}.
\end{proof}
%%%%%%%%%%%%%%%%%%%%%%%%%%%%%%%%%%%%%%%%%%%%%%%%%%%%%%%%%End_Proof%%%%%%%%%%%%%%%%%%%%%%%%%%%%%%%%%%%%%%%%%%%%%%%%%%%%%%%%%%%%%%%%%%%%%%%%%%%%%%%
%%%%%%%%%%%%%%%%%%%%%%%%%%%%%%%%%%%%%%%%%%%%%%%%%%%%%%%%%%%%%%%%%%%%%%%%%%%%%%%%%%%%%%%%%%%%%%%%%%%%%%%%%%%%%%%%%%%%%%%%%%%%%%%%%%%%%%%%%%%%%%%%%%%%%%%%%%%%%%%
% %%%%%%%%%%%%%%%%%%%%%%%%%%%%%%%%%%%%%%%%%%%%%%%%%%%%%%%%%%%%%%%%%%%%%%%%%%%%%%%%%%%%%%%%%%%%%%%%%%%%%%%%%%%%%%%%%%%%%%%%%%%%%%%%%%%%%%%%%%%%%%%%%%%%%%%%%%%%%
% %
% %
% %               Section: Nonlinear Exponential Decay
% %
% %
% %%%%%%%%%%%%%%%%%%%%%%%%%%%%%%%%%%%%%%%%%%%%%%%%%%%%%%%%%%%%%%%%%%%%%%%%%%%%%%%%%%%%%%%%%%%%%%%%%%%%%%%%%%%%%%%%%%%%%%%%%%%%%%%%%%%%%%%%%%%%%%%%%%%%%%%%%%%%%
%%%%%%%%%%%%%%%%%%%%%%%%%%%%%%%%%%%%%%%%%%%%%%%%%%%%%%%%%%%%%%%%%%%%%%%%%%%%%%%%%%%%%%%%%%%%%%%%%%%%%%%%%%%%%%%%%%%%%%%%%%%%%%%%%%%%%%%%%%%%%%%%%%%%%%%%%%%%%%%
\section{Nonlinear exponential decay}
 We consider the following iteration:
\begin{eqnarray*}\label{NBE_Iteration}
&&\{\partial_t+v\cdot \nabla_x+\nu^{\varpi}-K^{\varpi,x}_{w_{\varpi}}\}h^{m+1}=w_{\varpi}\Gamma\left(\frac{h^m}{w_{\varpi}},\frac{h^m}{w_{\varpi}}\right),\cr
&&\hspace{1cm}h^{m+1}(t,v,x)|_{\gamma_-}=h^{m+1}(t,v,R(x)v)|_{\gamma_+},\cr
&&\hspace{1cm}h^{m+1}(0,x,v)=h_0(x,v),\cr
&&\hspace{1cm}h^0(x,v)\equiv 0.
\end{eqnarray*}
By the Duhamel Principle
\begin{eqnarray*}
h^{m+1}=U(t)h_0+\int^t_0U(t-s)w_{\varpi}\Gamma\left(\frac{h^m}{w_{\varpi}},\frac{h^m}{w_{\varpi}}\right)(s)ds.
\end{eqnarray*}
We take $L^{\infty}$ norm on both sides and apply Theorem \ref{LInftyDecayTheorem} to obtain
\begin{eqnarray}\label{LIntfyEstimate}
\|h^{m+1}\|_{\infty}\leq Ce^{-\lambda t}\|h_0\|_{\infty}+\left\|\int^t_0U(t-s)w\Gamma\left(\frac{h^m}{w},\frac{h^m}{w}\right)(s)ds\right\|_{\infty}.
\end{eqnarray}
We observe
\begin{eqnarray*}
U(t-s)=G(t-s)+\int^t_s G(t-s_1)K^{\varpi,x}_{w_{\varpi}}U(s_1-s)ds_1
\end{eqnarray*}
to get
\begin{eqnarray}\label{GKGK for Gamma}
\begin{split}
&\int^t_0U(t-s)w_{\varpi}\Gamma\left(\frac{h^m}{w_{\varpi}},\frac{h^m}{w_{\varpi}}\right)(s)ds\cr
&\hspace{1cm}=\int^t_0G(t-s)w_{\varpi}\Gamma\left(\frac{h^m}{w_{\varpi}},\frac{h^m}{w_{\varpi}}\right)(s)ds\cr
&\hspace{1cm}+\int^t_0\int^t_s G(t-s)K^{\varpi,x}_{w_{\varpi}}U(s_1-s)w_{\varpi}\Gamma\left(\frac{h^m}{w_{\varpi}},\frac{h^m}{w_{\varpi}}\right) ds_1ds\cr
&\hspace{1cm}=I_1+I_2.
\end{split}
\end{eqnarray}
We can estimate $I_1$ using Lemma \ref{Gamma Estimate} as follows
\begin{eqnarray*}
\begin{split}
\left|\int^t_0G(t-s)w_{\varpi}\Gamma\left(\frac{h^m}{w_{\varpi}},\frac{h^m}{w_{\varpi}}\right)(s)ds\right|
&=\left|\int^t_0e^{-\mathcal{V}(t-s,x,v)}\left\{w\Gamma\left(\frac{h^m}{w_{\varpi}},\frac{h^m}{w_{\varpi}} \right)\right\}\big(s,X(s),V(s)\big)ds\right|\cr
&\leq C\int^t_0e^{-\nu_0(t-s)}\nu(v)\|h^m(s)\|^2_{\infty}ds\cr
&\leq Ce^{-\frac{\nu_0}{2}t}\left\{\sup_{0\leq s\leq \infty}e^{\frac{\nu_0}{2}s}\|h^m(s)\|_{\infty}\right\}^2.
\end{split}
\end{eqnarray*}
On the other hand, for any given initial datum $\tilde{h}_0$, We define the semigroup $\tilde{U}(t)\tilde{h}_0$ as the solution operator solving the following initial
boundary value problem:
\begin{eqnarray*}
&&\left\{\partial_t+v\cdot\nabla_x+\nu^{\varpi}-K^{\varpi,x}_{w_{\varpi}/\sqrt{1+|v-\varpi\times x|^2}}\right\}\{\tilde{U}(t)\tilde{h}_0\}=0,\cr
&&\hspace{1cm}\tilde{U}(t)\tilde{h}_0(t,x,v)=\tilde{U}(t)\tilde{h}_0(t,x,R(x)v),\cr
&&\hspace{1cm}\tilde{U}(t)\tilde{h}_0=\tilde{h}_0,
\end{eqnarray*}
where $\tilde h_0=\frac{h_0}{\sqrt{1+|v-\varpi\times x|^2}}$.
Direct computation shows that $\sqrt{1+|v-\varpi\times x|^2}\tilde U(t)\tilde{h}_0$ solves the original linear Boltzmann equation (\ref{LinearBE}).
Therefore, we deduce from the uniqueness in the $L^{\infty}$ class:
\begin{eqnarray*}
U(t)h_0\equiv\sqrt{1+|v-\varpi\times x|^2}\tilde{U}(t)\left\{\frac{h_0}{\sqrt{1+|v-\varpi\times x|^2}}\right\}.
\end{eqnarray*}
Therefore, we can rewrite $K^{\varpi,x}_{w_{\varpi}}U(s_1,s)w\Gamma\left(\frac{h^m}{w},\frac{h^m}{w}\right)(s)$ to get
\begin{eqnarray*}\label{GKGK for Gamma3}
&&\hspace{-0.2cm}\int^t_0\int^t_s G(t-s)K^{\varpi,x}_{w_{\varpi}}U(s_1-s)w_{\varpi}\Gamma\left(\frac{h^m}{w_{\varpi}},\frac{h^m}{w_{\varpi}}\right) ds_1ds\cr
&&=\int^t_0\int^t_se^{-\mathcal{V}(t-s_1,x,v)}\left\{\int K^{\varpi,X_{\bf cl}(s_1)}_{w_{\varpi}}
(V_{\bf cl}(s_1),v^{\prime})\{\sqrt{1+|v^{\prime}-\varpi\times X_{\bf cl}(s_1)|^2}\}dv^{\prime}\right\}\cr
&&\quad\times\left\|\tilde{U}(t)(s_1-s)\left\{\frac{w_{\varpi}}{\sqrt{1+|v^{\prime}-\varpi\times x|^2}}\Gamma\left(\frac{h^m}{w_{\varpi}},\frac{h^m}{w_{\varpi}}\right)(s)\right\}\right\|_{\infty}ds_1ds.
\end{eqnarray*}
Since $w^{-2}_{\varpi}(1+|v-\varpi\times x|)^3\in L^1$, the new weight
$\left\{\frac{w_{\varpi}}{\sqrt{1+|v-\varpi\times x|^2}}\right\}^{-2}(1+|v-\varpi\times x|)\in L^1$ so that Theorem \ref{LInftyDecayTheorem}
is valid for $\tilde{U}$. Since $\frac{\nu(v^{\prime})}{\sqrt{1+|v^{\prime}|^2}}\leq C_{\rho}$, from the proof of Lemma \ref{K_wEstimate},
\begin{eqnarray*}
\int_{\mathbb{R}^3}K^{\varpi,X_{\bf cl}(s_1)}_{w_{\varpi}}
(V_{\bf cl}(s_1),v^{\prime})\{\sqrt{1+|v^{\prime}-\varpi\times X_{\bf cl}(s_1)|^2}\}dv^{\prime}<+\infty.
\end{eqnarray*}
Hence, the integral can be estimated as follows
\begin{eqnarray*}\label{GKGK for Gamma3}
&&\hspace{-1cm}\int^t_0\int^t_se^{-\mathcal{V}(t-s_1,x,v)}
\left\|\tilde{U}(t)(s_1-s)\left\{\frac{w_{\varpi}}{\sqrt{1+|v^{\prime}-\varpi\times x|^2}}\Gamma\left(\frac{h^m}{w_{\varpi}},\frac{h^m}{w_{\varpi}}\right)(s)\right\}\right\|_{\infty}ds_1ds\cr
&&\leq C\int^t_0\int^t_se^{-\nu_0(t-s)}\|h^m(s)\|^2_{\infty}ds_1ds\cr
&&\leq Ce^{-\frac{\nu_0}{2}t}\left\{\sup_{0\leq s\leq\infty}e^{\frac{\nu_0}{2}s}\|h^m(s)\|\right\}^2.
\end{eqnarray*}
This implies that
\begin{eqnarray*}
\sup_{0\leq t\leq \infty}\{e^{\frac{\lambda}{2}s}\|h^m(s)\|_{\infty}\}\leq C\|h_0\|_{\infty},
\end{eqnarray*}
for $\|h_0\|_{\infty}$ sufficiently small.
Moreover, substracting $h^{m+1}-h^m$ yields
\begin{eqnarray}\label{NBE_Iteration}
\begin{split}
&\{\partial_t+v\cdot \nabla_x+\nu^{\varpi}-K^{\varpi,x}_{w_{\varpi}}\}\{h^{m+1}-h^m\}\cr
&\hspace{2cm}=w_{\varpi}\Gamma^{\varpi}\left(\frac{h^m}{w_{\varpi}},\frac{h^m}{w_{\varpi}}\right)-w_{\varpi}\Gamma^{\varpi}
\left(\frac{h^{m-1}}{w_{\varpi}},\frac{h^{m-1}}{w_{\varpi}}\right)\cr
%&\hspace{2cm}w\Gamma\left(\frac{h^{m}}{w_{\varpi}},\frac{h^{m}}{w_{\varpi}}\right)-w_{\varpi}\Gamma^{\varpi}\left(\frac{h^{m-1}}{w},\frac{h^{m-1}}{w_{\varpi}}\right)\cr
&\hspace{2cm}=w_{\varpi}\Gamma^{\varpi}\left(\frac{h^m-h^{m-1}}{w_{\varpi}},\frac{h^{m-1}}{w_{\varpi}}\right)
-w_{\varpi}\Gamma^{\varpi}\left(\frac{h^{m-1}}{w_{\varpi}},\frac{h^{m-1}-h^{m}}{w_{\varpi}}\right).
\end{split}
\end{eqnarray}
We can bound $\|h^{m+1}-h^m\|_{\infty}$ as in the previous estimate to obtain:
\begin{eqnarray*}
&&C\left\|\int^t_0U_0(t-s)w_{\varpi}\Gamma^{\varpi}\left(\frac{h^m-h^{m-1}}{w_{\varpi}},\frac{h^{m}}{w_{\varpi}}\right)(s)ds\right\|_{\infty}\cr
&&\quad +C\left\|\int^t_0U_0(t-s)w_{\varpi}\Gamma^{\varpi}\left(\frac{h^{m-1}}{w_{\varpi}},\frac{h^{m-1}-h^m}{w_{\varpi}}\right)(s)ds\right\|_{\infty}\cr
&&\leq C\sup_s\{e^{\lambda s}\{\|h^m(s)\|_{\infty}\}+\|h^{m-1}\|_{\infty}\}e^{-\lambda t}\sup_s
\{e^{\lambda s}\left\{\|h^m(s)-h^{m-1}(s)\|_{\infty}\right\}\}.
\end{eqnarray*}
Hence $h^m$ is a Cauchy sequence and the limit $h$ is the desired unique solution.\newline
The continuity, positivity and the uniqueness of the solution can be derived by the similar argument as in \cite{GuoConti}.
Therefore we omit it.
\appendix
%%%%%%%%%%%%%%%%%%%%%%%%%%%%%%%%%%%%%%%%%%%%%%%%%%%%%%%%%%%%%%%%%%%%%%%%%%%%%%%%%%%%%%%%%%%%%%%%%%%%%%%%%%%%%%%%
% %%%%%%%%%%%%%%%%%%%%%%%%%%%%%%%%%%%%%%%%%%%%%%%%%%%%%%%%%%%%%%%%%%%%%%%%%%%%%%%%%%%%%%%%%%%%%%%%%%%%%%%%%%%%%%
% %
% %
% %                     Rotational Local Maxwellian
% %
% %
% %%%%%%%%%%%%%%%%%%%%%%%%%%%%%%%%%%%%%%%%%%%%%%%%%%%%%%%%%%%%%%%%%%%%%%%%%%%%%%%%%%%%%%%%%%%%%%%%%%%%%%%%%%%%%%
%%%%%%%%%%%%%%%%%%%%%%%%%%%%%%%%%%%%%%%%%%%%%%%%%%%%%%%%%%%%%%%%%%%%%%%%%%%%%%%%%%%%%%%%%%%%%%%%%%%%%%%%%%%%%%%%
\section{Proof of Theorem \ref{MainResult} (1)}
In this appendix, we consider the derivation of the rotational local Maxwellian.
This proof basically relies on the argument of \cite{GuoConti}.
Suppose the following general form of local Maxwellian
\begin{eqnarray}\label{gMax}
\mu(x,v,t)=e^{a(x,t)+b(x,t)\cdot v+c(x,t)|v|^2}
\end{eqnarray}
satisfies the Boltzmann equation (\ref{Boltzmann Equation}).
Substituting (\ref{gMax}) into (\ref{Boltzmann Equation}) and using $Q(\mu, \mu)=0$, we obtain
\begin{eqnarray*}
\partial_t\mu+v\cdot \nabla\mu=0.
\end{eqnarray*}
This leads to the following system of equations for $a$, $b$ and $c$ for $(t,x)\in [0,\infty)\times \Omega$:
\begin{eqnarray*}\label{112}
\begin{split}
\partial_{x_i}c&=0,\quad i=1,2,3,\cr
\partial_{t}c+\partial_{x_i}b^i&=0,\quad i=1,2,3,\cr
\partial_{x_j}b^i+\partial_{x_i}b^j&=0,\quad i\neq j,\cr
\partial_{t}b^i+\partial_{x_i}a&=0,\quad i=1,2,3,\cr
\partial_{t}a&=0.
\end{split}
\end{eqnarray*}
This system gives a general form of $a,b,c$ ( Lemma 6 in \cite{GuoConti}):
\begin{eqnarray}\label{Max Explicit1}
&&\mu(t,x,v)=e^{(\frac{c_0}{2}|x|^2-b_0\cdot x+a_0)+(-c_0tx-c_1x+\varpi\times x+b_0t+b_1)\cdot v
+(\frac{c_0t^2}{2}+c_1t+c_2)|v|^2}
\end{eqnarray}
for some constants $a_0, b_0, b_1, c_0, c_1, c_2>0$.
From the specular boundary condition:
\begin{eqnarray*}
\mu(t,v,x)=\mu(t,x,R(x)v),
\end{eqnarray*}
we have for all $x\in \partial\Omega$ and $t\in \mathbb{R}^+$
\begin{eqnarray}\label{n=0}
\{-c_0tx-c_1x+\varpi\times x+b_0t+b_1\}\cdot n(x)=0,
\end{eqnarray}
which implies
\begin{eqnarray*}
c_0=c_1=b_0=0.
\end{eqnarray*}
Hence (\ref{n=0}) reduces to
%\begin{eqnarray}\label{1088}\varpi\times x\cdot n(x)\equiv0.\end{eqnarray}
\begin{eqnarray}\label{1088}
\{\varpi\times x+b_1\}\cdot n(x)\equiv0.
\end{eqnarray}
%Hence $b_0=0$ for all $x\in \partial\Omega$ and\begin{eqnarray}\label{108}\{\omega\times x\}\cdot n(x)+b_1\cdot n(x)=0.\end{eqnarray}
We now decompose $b_1$ as
\begin{eqnarray*}
b_1=\beta_1\frac{\varpi}{|\varpi|}+\beta_2\eta, \mbox{ where } |\eta|=1 \mbox{ and } \eta\perp\varpi.
\end{eqnarray*}
Then $\eta=\{\frac{\varpi}{|\varpi|}\times\eta\}\times\frac{\varpi}{|\varpi|}$. Therefore, we get
\begin{eqnarray*}
\begin{split}
b_1%&=\beta_1\frac{\varpi}{|\varpi|}+\beta_2\eta\cr
&=\beta_1\frac{\varpi}{|\varpi|}
+\beta_2\left\{\frac{\varpi}{|\varpi|}\times\eta\right\}\times\varpi\cr
&\equiv\beta_1\frac{\varpi}{|\varpi|}-x_0\times\varpi,
\end{split}
\end{eqnarray*}
where $x_0=-\beta_2\frac{\varpi}{|\varpi|}\times \eta$. By plugging $b_1$ back into (\ref{1088}), we obtain
\begin{eqnarray}\label{1324}
\beta_1\frac{\varpi}{|\varpi|}\cdot n(x)+\varpi\times (x-x_0)\cdot n(x)=0.
\end{eqnarray}
We now choose $x^{\prime}$ such that $\xi(x^{\prime})=\min_{\xi(x)=0}\varpi\cdot x$.
Then the Lagrangi Multiplier theorem implies there exists a constant $\lambda$ such that $\varpi=\lambda n(x^{\prime})$.
Setting $x=x^{\prime}$ in (\ref{1324}), we obtain $\beta_1=0$. %$\varpi\times(x^{\prime}-x_0)\cdot n(x^{\prime})=0$ and hence .
Therefore, we have from (\ref{1324})
%\begin{eqnarray*}\mu=e^{a_0+\{\varpi\times (x-x_0)\}\cdot v+c_2|v|^2}.\end{eqnarray*}
\begin{eqnarray*}\label{n=0_2}
\varpi\times (x-x_0)\cdot n(x)=0.
\end{eqnarray*}
%{\bf Case I:} $\Omega$ is the sphere.\newline
%Let $r$ denotes the radius of $\Omega$. Then we have $n(re_i)= e_i ~(i=1,2,3)$, where $r$ denotes the radius of $\Omega$.
%\begin{eqnarray*}
%&&\varpi\times (e_i-x_0)\cdot e_i=\varpi\times x_0\cdot e_i=0,
%\end{eqnarray*}
%which shows that $\varpi\times x_0=0$.\newline
We now claim that $\varpi\times x_0=0$. We first note that (\ref{cylindricallySymm}) leads to
%\begin{eqnarray}\label{117}e_3\times x\cdot n(x)=0,\end{eqnarray}
\begin{eqnarray}\label{paral1}
x_1n_2-x_2n_1=0,
\end{eqnarray}
%\begin{eqnarray}&&\varpi\times (x-x_0)\cdot n(x)=0\cr
%&&(\varpi_1,\varpi_2,\varpi_3,)\times (x-x_0)\cdot n(x)=0\cr
%&&(\varpi_1,\varpi_2,0)\times (x-x_0)\cdot n(x)+(0,0,\varpi_3)\times (x-x_0)\cdot n(x)=0\cr
%&&(\varpi_1,\varpi_2,0)\times (x-x_0)\cdot n(x)+(0,0,\varpi_3)\times (-x_0)\cdot n(x)=0\cr
%&&(\varpi_1,\varpi_2,0)\times x\cdot n(x)-(\varpi_1,\varpi_2,\varpi_3)\times x_0\cdot n(x)=0\cr\end{eqnarray}by
%\begin{eqnarray}&&x_1n_2-x_2n_1=e_3\times x\cdot n(x)=0.\cr&&(x_1,x_2)//(n_1(x),n_2(x))\cr&&n(\lambda_1e_3)= e_3,n(\lambda_2e_3)=-e_3.\cr
%&&0=\varpi\times (\lambda_1e_3-x_0)\cdot e_3=\varpi\times x_0\cdot e_3=0.\cr&&(\varpi_1,\varpi_2)//(x_{01},x_{02})\cr
%&&n(-x_1,-x_2,x_3)=n(x_1,x_2,x_3)-2\frac{(x_1,x_2,0)}{\sqrt{x_1^2+x_2^2}}\cdot n(x_1,x_2,x_3)\frac{(x_1,x_2,0)}{\sqrt{x_1^2+x_2^2}}\cr
%&&(n_1,n_2,n_3)-2\frac{n_1x_1+n_2x_2}{x_1^2+x_2^2}(x_1,x_2,0)\cr&&\left(\frac{(-x_1^2+x^2_2)n_1-2x_1x_2n_2}{x_1^2+x^2_2},\frac{-2x_1x_2n_2+(x_1^2-x^2_2)n_1}{x_1^2+x^2_2},0 \right)
%\end{eqnarray}
which implies
\begin{eqnarray*}\label{paral2}
(x_1,x_2)//(n_1,n_2).
\end{eqnarray*}
Therefore, we can choose $\hat z$, $a$ and $\bar z$ after translation in $z$ axis, if necessary such that
\begin{eqnarray}\label{e_1e_2e_3}
&&n\big((0,0,\hat z)\big)=(0,0,1),\quad n\big((a,0,0)\big)=(1,0,0),\quad n\big((0,b,0)\big)=(0,1,0).
\end{eqnarray}
%By translating $\Omega$ along the $e_3$, we can assume without loss of generality that $\bar z=0$.
Plugging  (\ref{e_1e_2e_3}) into (\ref{n=0_2}), we have
%\begin{eqnarray}\varpi_1y_0-\varpi_2x_0=0.\end{eqnarray}
\begin{eqnarray*}\label{paral2}
\begin{split}
&\varpi_1x_{02}-\varpi_2x_{01}=0,\cr
&\varpi_2x_{03}-\varpi_3x_{02}=0,\cr
&\varpi_1x_{03}-\varpi_3 x_{01}=0,
\end{split}
\end{eqnarray*}
which is equivalent to
\begin{eqnarray*}
\varpi\times x_0=0.
\end{eqnarray*}
We therefore have
\begin{eqnarray*}
0=\varpi\times\big(x-x_0\big)\cdot n(x)=\varpi\times x\cdot n(x).
\end{eqnarray*}
%Now from the specular reflection, we have for any $x\in \partial\Omega, b(t,x)\cdot n(x)=0$ or
Now, to derive additional information on $\omega$, we perform explicit calculation using (\ref{paral1}) to see
\begin{eqnarray*}
\varpi\times x\cdot n(x)%&=&\left|\begin{array}{ccc}\varpi_1&\varpi_2&\varpi_3\\x_1&x_2&x_3\\n_1&n_2&n_3\end{array}\right|\cr
&=&\varpi_1(x_2n_3-x_3n_2)-\varpi_2(x_1n_3-x_3n_1)+\varpi_3(\underbrace{x_1n_2-x_2n_1}_{=0})\cr
&=&1/x_1(n_1x_3-n_3x_1)(\varpi_1x_2-\varpi_2x_1).
\end{eqnarray*}
Hence we have either $\varpi_1x_2-\varpi_2x_1=0$ or $n_1x_3-n_3x_1=0$.\newline
%If $\varpi_1x_2-\varpi_2x_1=0$, then we have $\varpi_1=\varpi_2=0.$\newline
\noindent (i) The case of $\varpi_1x_2-\varpi_2x_1=0$: The only way this identity to be true is
\[\varpi_1=\varpi_2=0,\]
which leads to
\begin{eqnarray*}
\mu(x,v)=e^{c_0|v|^2+(0,0,\varpi_3)\times x\cdot v+a_0}.
\end{eqnarray*}
\noindent (ii) The case of $n_1x_3-n_3x_1=0$:  Together with (\ref{paral1}), we have for $x_1\neq 0$
\begin{eqnarray}\label{n1}
n(x)=\frac{n_1}{x_1}(x_1,x_2,x_3).
\end{eqnarray}
Hence we have
\begin{equation*}
1=|n(x)|=\frac{n^2_1}{x^2_1}(x^2_1+x^2_2+x^2_3).
\end{equation*}
This gives
\begin{equation}\label{putit}
n_1=\pm\frac{x_1}{\sqrt{x^2_1+x^2_2+x^2_3}}.
\end{equation}
We then put (\ref{putit}) back to (\ref{n1}) to get
\begin{equation*}
n(x)=\pm\frac{(x_1,x_2,x_3)}{\sqrt{x^2_1+x^2_2+x^2_3}}=\pm\frac{x}{|x|},
\end{equation*}
%We choose th
%\begin{equation}
%n(x)=\frac{x}{|x|}.
%\end{equation}
which implies that $\Omega$ is a sphere.\newline
In conclusion,
a Maxwellian solution (\ref{Max Explicit1}) to the Boltzmann equation (\ref{Boltzmann Equation}) in a rotationally symmetric domain endowed
with the specular reflection boundary condition reduces to the following form:
\begin{eqnarray*}
\mu(x,v)=e^{c_0|v|^2+\varpi\times x\cdot v+a_0},
\end{eqnarray*}
where $\omega_1=\omega_2=0$ except for the case of $\Omega=\mathbb{S}^3$.
%Without loss of generality, we normalize it for simplicity as:
%\begin{eqnarray*}
%\mu_{\varpi}(x,v)=e^{-\frac{|v|^2}{2}+\varpi\times x\cdot v}.
%\end{eqnarray*}
%where $\varpi=(0,0,\varpi)$ for $\mathcal{(A)}1$ and $\varpi=(\varpi_1,\varpi_2,\varpi_3)$ for $\mathcal{(A)}2$.
\section*{Acknowledgement}
Authors would like to thank Professor Yan Guo for fruitful discussions.
The research of C, Kim was supported by FRG 07-57227.
The research of S.-B. Yun was supported by the National Research Foundation of Korean Grant funded by the Korean Government (Ministry of Education,
Science and Technology). [NRF-2010-357-C00005].
\bibliographystyle{amsplain}

\begin{thebibliography}{10}
\bibitem{A-C} Arkeryd, L., Cercignani, C.: A global existence theorem for the initial boundary value problem for
the Boltzmann equation when the boundaries are not isothermal. Arch. Rational Mech. Anal. {\bf 125} (1993),no3,  271-288.
\end{thebibliography}

\end{document}